\documentclass[10pt,reqno]{amsart}
\usepackage[10pt,nona4]{optional} 

\usepackage{graphicx}
\usepackage{dsfont,amsmath,amsfonts,amscd,amssymb,mathrsfs,eufrak,float,upgreek}

\usepackage{setspace}

\newcommand{\marginparstretch}{0.6}
\let\oldmarginpar\marginpar
\renewcommand\marginpar[1]{\-\oldmarginpar[\framebox{\setstretch{\marginparstretch}\begin{minipage}{\marginparwidth}{\raggedleft\tiny #1}\end{minipage}}]{\framebox{\setstretch{\marginparstretch}\begin{minipage}{\marginparwidth}{\raggedright\tiny #1}\end{minipage}}}}

\usepackage[colorlinks]{hyperref}
\usepackage{tikz,mathrsfs}
\include{tikz-3dplot}
\usetikzlibrary{arrows,decorations.pathmorphing,decorations.pathreplacing,positioning,shapes.geometric,shapes.misc,decorations.markings,decorations.fractals,calc,patterns}

\tikzstyle{decision} = [diamond, draw, fill=blue!20,
    text width=4.5em, text badly centered, node distance=2.5cm, inner sep=0pt]
\tikzstyle{block} = [rectangle, draw, fill=blue!20,
    text width=5em, text centered, rounded corners, minimum height=3em]
\tikzstyle{line} = [draw, very thick, color=black!50, -latex']
\tikzstyle{cloud} = [draw, ellipse,fill=red!30, node distance=2.5cm,
    minimum height=3em]
\tikzstyle{cloud2} = [draw, ellipse,fill=red!30, text=white,text width=10em, node distance=2.5cm, text centered, minimum height=4em]

\tikzset{
        cvertex/.style={circle,draw=black,inner sep=1pt,outer sep=3pt},
        vertex/.style={circle,fill=black,inner sep=1pt,outer sep=3pt},
        DB/.style={circle,draw=black,circle,fill=black,inner sep=0pt, minimum size=4pt},
        DW/.style={circle,draw=black,inner sep=0pt, minimum size=4pt},
        star/.style={circle,fill=yellow,inner sep=0.75pt,outer sep=0.75pt},
        tvertex/.style={inner sep=1pt,font=\scriptsize},
        gap/.style={inner sep=0.5pt,fill=white}}

\tikzstyle{mybox} = [draw=black, fill=blue!10, very thick,
    rectangle, rounded corners, inner sep=10pt, inner ysep=20pt]
\tikzstyle{boxtitle} =[fill=blue!50, text=white,rectangle,rounded corners]

\usepackage{multirow,booktabs,longtable}
\newtheorem{thm}{Theorem}[section]
\newtheorem{prop}[thm]{Proposition}
\newtheorem{lemma}[thm]{Lemma}
\newtheorem{defin}[thm]{Definition}
\newtheorem{cor}[thm]{Corollary}

\newtheorem{conj}[thm]{Conjecture}

\theoremstyle{definition} 

\newtheorem{example}[thm]{Example}
\newtheorem{setup}[thm]{Setup}

\newtheorem{remark}[thm]{Remark}

\newtheorem{notation}[thm]{Notation}

\numberwithin{equation}{section}

\newcounter{tempenum}

\newcommand{\m}{\mathfrak{m}}
\newcommand{\n}{\mathfrak{n}}
\newcommand{\p}{\mathfrak{p}}

\renewcommand{\t}[1]{\textnormal{#1}}

\def\op{\mathop{\rm op}\nolimits}

\def\GL{\mathop{\rm GL}\nolimits}
\def\SL{\mathop{\rm SL}\nolimits}
\def\CM{\mathop{\rm CM}\nolimits}
\def\uCM{\mathop{\underline{\rm CM}}\nolimits}

\def\depth{\mathop{\rm depth}\nolimits}
\def\fl{\mathop{\mathrm{fdmod}}\nolimits}

\def\mod{\mathop{\rm mod}\nolimits}
\def\rep{\mathop{\rm Rep}\nolimits}

\def\coh{\mathop{\rm coh}\nolimits}

\def\Mod{\mathop{\rm Mod}\nolimits}

\def\refl{\mathop{\rm ref}\nolimits}

\def\proj{\mathop{\rm proj}\nolimits}

\newcommand{\pd}{\mathrm{pd}}
\def\id{\mathop{\rm inj.dim}\nolimits}

\def\uEnd{\mathop{\underline{\rm End}}\nolimits}
\def\Hom{\mathop{\rm Hom}\nolimits}

\def\RHom{\mathop{\rm {\bf R}Hom}\nolimits}

\def\End{\mathop{\rm End}\nolimits}
\def\Ext{\mathop{\rm Ext}\nolimits}
\def\Tor{\mathop{\rm Tor}\nolimits}

\def\add{\mathop{\rm add}\nolimits}
\def\Cok{\mathop{\rm Cok}\nolimits}
\def\Ker{\mathop{\rm Ker}\nolimits}

\def\rank{\mathop{\rm rank}\nolimits}
\def\rk{\mathop{\sf rk}\nolimits}
\def\Im{\mathop{\rm Im}\nolimits}
\def\Sing{\mathop{\rm Sing}\nolimits}

\def\Spec{\mathop{\rm Spec}\nolimits}

\def\ord{\mathop{\rm ord}\nolimits}

\def\gl{\mathop{\rm gl.dim}\nolimits}

\newcommand{\boldb}{\mathbf{b_I}}
\newcommand{\boldc}{\mathbf{c_I}}
\def\vdim{\mathop{\underline{\rm dim}}\nolimits}
\def\cx{\mathop{\mathrm{cx}}\nolimits}

\def\D{\mathop{\rm{D}^{}}\nolimits}
\def\Dsg{\mathop{\rm{D}_{\sf sg}}\nolimits}
\def\Db{\mathop{\rm{D}^b}\nolimits}
\def\Kb{\mathop{\rm{K}^b}\nolimits}

\def\flop{{\sf{Flop}}}
\def\twist{{\sf{Twist}}}

\def\Id{\mathop{\rm{Id}}\nolimits}

\newcommand{\K}{\mathop{{}_{}\mathbb{C}}\nolimits}

\newcommand{\con}{\mathrm{con}}
\newcommand{\CA}{\mathrm{A}_{\con}}

\def\redu{\mathop{\rm red}\nolimits}

\def\RHom{{\rm{\bf R}Hom}}
\def\sHom{\mathcal{H}om}

\newcommand\RDerived[1]{{\rm\bf R}{#1}}

\newcommand\art{\mathsf{art}}
\newcommand\cart{\mathsf{cart}}
\newcommand\alg{\mathsf{Alg}}
\newcommand\calg{\mathsf{CAlg}}
\newcommand\Sets{\mathsf{Sets}}
\newcommand\cDef{c\mathcal{D}ef}
\newcommand\Def{\mathcal{D}ef}

\newcommand{\cC}{\mathcal{C}}

\newcommand{\cE}{\mathcal{E}}
\newcommand{\cF}{\mathcal{F}}
\newcommand{\cG}{\mathcal{G}}

\newcommand{\cL}{\mathcal{L}}
\newcommand{\cM}{\mathcal{M}}
\newcommand{\cN}{\mathcal{N}}
\newcommand{\cO}{\mathcal{O}}

\newcommand{\cS}{\mathcal{S}}

\newcommand{\cU}{\mathcal{U}}
\newcommand{\cV}{\mathcal{V}}
\newcommand{\cW}{\mathcal{W}}

\newcommand{\Per}{{}^{0}\mathfrak{Per}}
\newcommand{\mPer}{{}^{-1}\mathfrak{Per}}

\begin{document}
\title{\textsc{Flops and Clusters in the Homological Minimal Model Programme}}
\author{Michael Wemyss}
\address{Michael Wemyss, The Mathematics and Statistics Building,
University of Glasgow, University Place, Glasgow, G12 8SQ, UK.}
\email{michael.wemyss@glasgow.ac.uk}
\begin{abstract}
Suppose that $f\colon X\to\Spec R$ is a minimal model of a complete local Gorenstein 3-fold, where the fibres of $f$ are at most one dimensional, so by \cite{VdB1d} there is a  noncommutative ring $\Lambda$ derived equivalent to $X$.  For any collection of curves above the origin, we show that this collection contracts to a point without contracting a divisor if and only if a certain factor of $\Lambda$ is finite dimensional, improving a result of \cite{DW2}.  We further show that the mutation functor of \cite[\S6]{IW4} is functorially isomorphic to the inverse of the Bridgeland--Chen flop functor in the case when the factor of $\Lambda$ is finite dimensional.   These results then allow us to jump between all the minimal models of $\Spec R$ in an algorithmic way, without having to compute the geometry at each stage.  We call this process the Homological MMP.  

This has several applications in GIT approaches to derived categories, and also to birational geometry.  First, using mutation we are able to compute the full GIT chamber structure by passing to surfaces.  We say precisely which chambers give the distinct minimal models, and also say which walls give flops and which do not, enabling us to prove the Craw--Ishii conjecture in this setting. Second, we are able to precisely count the number of minimal models, and  also give bounds for both the maximum and the minimum numbers of minimal models based only on the dual graph enriched with scheme theoretic multiplicity. Third, we prove a bijective correspondence between maximal modifying $R$-module generators and minimal models, and for each such pair in this correspondence give a further correspondence linking the endomorphism ring and the geometry.  This lifts the Auslander--McKay correspondence to dimension three. 
\end{abstract}
\thanks{The author was supported by EPSRC grant~EP/K021400/1.}
\maketitle
\parindent 20pt
\parskip 0pt

\tableofcontents

\section{Introduction}

\subsection{Setting}
One of the central problems in the birational geometry of 3-folds is to construct, given a suitable singular space $\Spec R$, all its minimal models $X_i\to\Spec R$ and to furthermore pass between them, via flops, in an effective manner.  

The classical geometric method of producing minimal models is to take Proj of an appropriate graded ring.  It is known that the graded ring is finitely generated, so this method produces a variety equipped with an ample line bundle.  However, for many purposes this ample bundle does not tell us much information, and one of the themes of this paper, and also other homological approaches in the literature, is that we should be aiming for a much larger (ideally tilting) bundle, one containing many summands, whose determinant bundle recovers the classically obtained ample bundle.  These larger bundles, and their noncommutative endomorphism rings, encode much more information about the variety than simply the ample bundle does.

At the same time, passing between minimal models in an effective way is also a rather hard problem in general.  There are various approaches to this; one is to hope for some form of GIT chamber decomposition in which wandering around, crashing through appropriate walls, eventually yields all the projective minimal models.  Another is just to find a curve, flop, compute all the geometry explicitly, and repeat.  Neither is ideal, since both usually require a tremendous amount of calculation.  For example the GIT method needs first a calculation of the chamber structure, then second a method to determine what happens when we pass through a wall.  Without additional information, and without just computing both sides, standing in any given chamber it is very hard to tell which wall to crash through next in order to obtain a new minimal model.  

The purpose of this paper is to demonstrate, in certain cases where we have this larger tilting bundle, that the extra information encoded in the endomorphism ring can be used to produce a very effective homological method to pass between the minimal models, both in detecting which curves are floppable, and also in producing the flop.  As a consequence, this then supplies us with the map to navigate the GIT chambers, and the much finer control that this map gives means that our results imply (but are not implied by) many results in derived category approaches to GIT, braiding of flops, and faithful group actions.  We outline only some in this paper, as there are a surprising number of other corollaries.

\subsection{Overview of the Algorithm}\label{algorithm intro} We work over $\K$. Throughout this introduction, for simplicity of the exposition, the initial geometric input is a crepant projective birational morphism $X\to\Spec R$, with one dimensional fibres, where $R$ is a three dimensional normal Gorenstein complete local ring and $X$ has only Gorenstein terminal singularities.  This need not be a flopping contraction, $X$ need not be a minimal model, and $R$ need not have isolated singularities. We remark that many of our arguments work much more generally than this, see \S\ref{generalities}.

Given this input, we associate a noncommutative ring $\Lambda:=\End_R(N)$ for some reflexive $R$-module $N$, and a derived equivalence
\begin{eqnarray}
\Uppsi_X^{\phantom -}\colon\Db(\coh X)\to\Db(\mod\Lambda)\label{VdBfunctor}
\end{eqnarray}
as described in \cite[\S3]{VdB1d}.  

It is not strictly necessary, but it is helpful to keep in mind, that a  presentation of $\Lambda$ as a quiver with relations can be obtained by replacing every curve above the origin by a dot (=vertex), and just as in the two-dimensional McKay correspondence we add an additional vertex corresponding to the whole scheme--theoretic fibre.  We draw arrows between the vertices if the curves intersect, and there are rules that establish how the additional vertex connects to the others.  The loops on vertices correspond to self-extension groups, and so in the case that $X$ is smooth, the loops encode the normal bundle of the curves.  This is illustrated in Figure~\ref{Fig1}, but for details see \S\ref{perverse and tilting}.

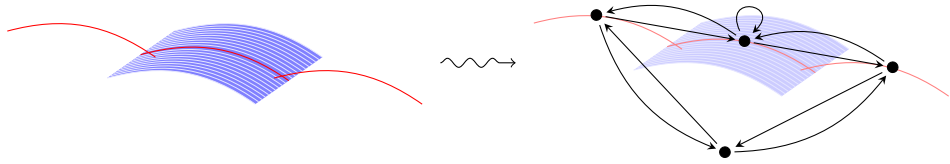
\begin{figure}[!h]
\begin{center}
\begin{tikzpicture}
\node at (0,0)
{\begin{tikzpicture} [bend angle=25, looseness=1,transform shape, rotate=-10]
\fill[fill=blue!50] (0,0,1) -- (0,0,-1) to [bend left=25] (2,0,-1) -- (2,0,1) to [bend right=25] (0,0,1);
\foreach \y in {0.1,0.2,...,1}{ 
\draw[very thin,blue!10] (0,0,\y) to [bend left=25] (2,0,\y);
\draw[very thin,blue!10] (0,0,-\y) to [bend left=25] (2,0,-\y);}
\draw[red] (0,0,0) to [bend left=25] (2,0,0);
\draw[red] (1.8,0,0) to [bend left=25] (3.8,0,0);
\draw[red] (-1.8,0,0) to [bend left=25] (0.2,0,0);
\node (0) at (1,-1.25,0) {};
\end{tikzpicture} };

\node at (7,0.2) {
\begin{tikzpicture} [bend angle=25, looseness=1,transform shape, rotate=-10,>=stealth]
\fill[fill=blue!20] (0,0,1) -- (0,0,-1) to [bend left=25] (2,0,-1) -- (2,0,1) to [bend right=25] (0,0,1);
\foreach \y in {0.1,0.2,...,1}{ 
\draw[very thin,blue!10] (0,0,\y) to [bend left=25] (2,0,\y);
\draw[very thin,blue!10] (0,0,-\y) to [bend left=25] (2,0,-\y);}
\draw[red!50] (0,0,0) to [bend left=25] (2,0,0);
\draw[red!50] (1.8,0,0) to [bend left=25] (3.8,0,0);
\draw[red!50] (-1.8,0,0) to [bend left=25] (0.2,0,0);
\filldraw [black] (1,0.25,0) circle (2pt);
\filldraw [black] (3,0.25,0) circle (2pt);
\filldraw [black] (-1,0.25,0) circle (2pt);
\filldraw [black] (1,-1.25,0) circle (2pt);
\node (1) at (-1,0.25,0) {};
\node (2) at (1,0.25,0) {};
\node (3) at (3,0.25,0) {};
\node (0) at (1,-1.25,0) {};
\draw[->,black] (1) -- (2);
\draw[->,black] (2) -- (3);
\draw[->,black]  (2) edge [in=55,out=125,loop,looseness=8] (2);
\draw[->,black,bend right] (2) to (1);
\draw[->,black, bend right] (3) to (2);
\draw[->,black] (0) -- (1);
\draw[->,black] (3) -- (0);
\draw[->,bend right,black] (1) to (0);
\draw[->, bend right,black] (0) to (3);
\end{tikzpicture}};
\draw [->,decorate, 
decoration={snake,amplitude=.6mm,segment length=3mm,post length=1mm}] 
(3,0.4) -- (4,0.4);
\end{tikzpicture}
\end{center}
\caption{From geometry to algebra.}\label{Fig1}
\end{figure}

At its heart, this paper contains two key new ideas.  The first is that certain factors of the algebra $\Lambda$ encode noncommutative deformations of the curves, and thus detects which curves are floppable.  The second is that when curves flop we should not view the flop as a variation of GIT, rather we should view the flop as a change in the algebra (via a universal property) whilst keeping the GIT stability constant.  See \ref{mut moduli gives flop}.  Specifically, we prove that the mutation functor of \cite[\S6]{IW4} is functorially isomorphic to the inverse of the Bridgeland--Chen flop functor \cite{Bridgeland, Chen} when the curves are floppable.  It is viewing the flop via this universal property that gives us the new extra control over the process; indeed it is the mutated algebra that contains exactly the information needed to iterate, without having to explicitly calculate the geometry at each step.  

This new viewpoint, and the control it gives, in fact implies many results in GIT, specifically chamber structures and wall crossing, and also many results in the theory of noncommutative minimal models, in particular producing an Auslander--McKay correspondence in dimension three.   We describe the GIT results in \S\ref{GIT intro}, and the other results in \S\ref{other applications intro}.  In the remainder of this subsection we sketch the algorithm that jumps between the minimal models of $\Spec R$.  The process, which we call the Homological MMP, is run as illustrated in Figure~\ref{Fig2} on page \pageref{Fig2}.

The initial input is the crepant morphism $X\to\Spec R$ above, where $X$ has only Gorenstein terminal singularities. 

\medskip
\textbf{Step 1: Contractions.}  The first task is to determine which subsets of the curves contract to points without contracting a divisor, and can thus be flopped.  Although this is usually obvious at the input stage (we generally understand the initial input), it becomes important after the flop if we are to continue running the programme.

Let $C$ be the scheme-theoretic fibre above the unique closed point of $\Spec R$, so that taking the reduced scheme structure we obtain $\bigcup_{i=1}^n C_i$ with each $C_i\cong \mathbb{P}^1$.  We pick a subset of the curves, say $I\subseteq \{1,\hdots,n\}$, and ask whether $\bigcup_{i\in I}C_i$ contracts to a point without contracting a divisor.  Corresponding to each curve $C_i$ is an idempotent $e_i$ in the algebra $\Lambda:=\End_R(N)$ from \eqref{VdBfunctor}, and we set $\Lambda_I:= \Lambda/\Lambda(1-\sum_{i\in I}e_i)\Lambda$.  Our first main result, a refinement of \cite{DW2}, is the following.
 
\begin{thm}[={\ref{contract on f}}]\label{contraction thm intro}
$\bigcup_{i\in I}C_i$ contracts to a point without contracting a divisor if and only if $\dim_\mathbb{C}\Lambda_I<\infty$.
\end{thm}
 
In fact \ref{contraction thm intro} is true regardless of the  singularities on $X$ and $\Spec R$, and needs no assumptions on crepancy.  Contracting the curves $\bigcup_{i\in I}C_i$, which we can do at will since $R$ is complete local, yields a diagram
\begin{eqnarray}
\begin{array}{c}
\begin{tikzpicture}
\node (X) at (-1.33,0) {$X$};
\node (Xcon) at (0,-0.8) {$X_{\con}$};
\node (R) at (0,-2) {$\Spec R$};
\draw[->] (X) -- node[above right] {$\scriptstyle g$} (Xcon);
\draw[->] (Xcon) -- node[right] {$\scriptstyle h$} (R);
\draw[->] (X) -- node[below left] {$\scriptstyle f$} (R);
\end{tikzpicture}
\end{array}\label{contraction f and g}
\end{eqnarray}
By \cite{DW2} there is a contraction algebra $\CA$, constructed with respect to the morphism $g$, that detects whether the curves in $I$ contract to a point without contracting a divisor.  The subtlety in the proof of \ref{contraction thm intro} is that $\Lambda$ and thus $\Lambda_I$ is constructed with respect to the morphism $f$, so to establish \ref{contraction thm intro} requires us to relate the algebras $\CA$ and $\Lambda_I$.  It turns out that they are isomorphic, but this can only be established by appealing to a universal property.  There is not even any obvious morphism between them.

\medskip
\textbf{Step 2: Mutation and Flops.}  We again pick a subset of curves $\{ C_i\mid i\in I\}$, but for simplicity in this introduction we assume that there is only one curve $C_i$ (i.e.\ $I=\{i\}$), although this paper does also cover the general situation, and all of the theorems stated here have multi--curve analogues.  The curve $C_i$ corresponds to an indecomposable summand $N_i$ of the $R$-module $N$.  By Step 1 the curve $C_i$ flops if and only if $\dim_{\mathbb{C}}\Lambda_i<\infty$.  Regardless of whether it flops or not, we can always categorically mutate the module $N$ with respect to the summand $N_i$, in the sense of \cite[\S 6]{IW4}, to produce another module $\upnu_i N$ (possibly equal to $N$), together with a derived equivalence
\[
\Upphi_i\colon \Db(\mod\Lambda)\to\Db(\mod \upnu_i\Lambda)
\]
where $\upnu_i\Lambda:=\End_R(\upnu_i N)$.  See \S\ref{mut prelim} for definitions and details.  This requires no assumptions on the singularities of $X$, but does require $R$ to be normal Gorenstein.  Note that the categorical mutation used here is inspired by, but in many ways is much different than, the mutation in cluster theory and elsewhere in the literature.  The main point is that the mutation here tackles the situation where there are loops, 2-cycles, and no superpotential, which is the level of generality needed to apply the results to possibly singular minimal models.   Consequently, this mutation is not just a simple combinatorial rule (unlike, say, Fomin--Zelevinsky mutation from cluster theory), however in practice $\upnu_i\Lambda$ can still be calculated easily.

The following is our next main result. When $C_i$ flops, we denote the flop by $X^+$.

\begin{thm}[={\ref{flop=mut general thm}}, {\ref{MMmodule under flop cor 2}}]\label{mut functor thm intro}
With the notation as above,
\begin{enumerate}
\item\label{mut thm 1 intro} The irreducible curve $C_i$ flops if and only if $\upnu_i N\neq N$.
\item\label{mut thm 2 intro}  If $\Gamma$ denotes the natural algebra associated to the flop $X^+$ \cite{VdB1d}, then $\Gamma\cong\upnu_i\Lambda$.
\item\label{mut thm 3 intro} If further $X\to\Spec R$ is a minimal model,  and $\dim_\mathbb{C}\Lambda_i<\infty$, then
\[
\Upphi_i\cong \Uppsi_{X^+}\circ\flop\circ\Uppsi_X^{-1} 
\]
where $\Uppsi$ are the functors in \eqref{VdBfunctor}, and $\flop$ is the inverse of the flop functor of Bridgeland--Chen \cite{Bridgeland,Chen}.
\end{enumerate}
\end{thm} 
\noindent
 The first part of the theorem allows us later to give a lower bound on the number of minimal models, and it turns out that the third part is one half of a dichotomy, namely
\[
\Upphi_i\cong\begin{cases} \Uppsi_{X^+}\circ\flop\circ\Uppsi_X^{-1} & \mbox{if }\dim_\mathbb{C}\Lambda_i<\infty\\
 \Uppsi_{X}\circ\twist\circ\Uppsi_X^{-1} & \mbox{if }\dim_\mathbb{C}\Lambda_i=\infty, \end{cases}
\]
where $\twist$ is a Fourier--Mukai twist-like functor over a noncommutative one-dimensional scheme.  We do not give the details here, as a more general treatment is given in \cite{DW4}.

We also remark that the proof of \ref{mut functor thm intro} does not need or refer to properties of the generic hyperplane section, so there is a good chance that in future we will be able to remove the assumption that $X$ has only Gorenstein terminal singularities, see \ref{B2}.

However, of the results in \ref{mut functor thm intro}, it is part two that is the key, since it allows us to iterate.  First, \ref{mut functor thm intro}\eqref{mut thm 2 intro} allows us to immediately read off the dual graph of the flop without explicitly calculating it in coordinates, since the dual graph can be read off from the mutated quiver.  Second, and most importantly, combining \ref{mut functor thm intro}\eqref{mut thm 2 intro} with \ref{contraction thm intro} (applied to $\upnu_i\Lambda$) allows us to detect which curves are contractible {\em after the flop} by inspecting factor algebras of the form $\upnu_i\Lambda/\upnu_i\Lambda(1-e)\upnu_i\Lambda$.  There is no way of seeing this information on the original algebra $\Lambda$, which is one of the main reasons why fixing $\Lambda$ and changing the stability there does not lend itself easily to iterations.   Hence we do not change GIT stability, we instead change the algebra by plugging $\upnu_i\Lambda$ back in as the new input, and continue the programme in an algorithmic way.  This is summarised in Figure~\ref{Fig2}.

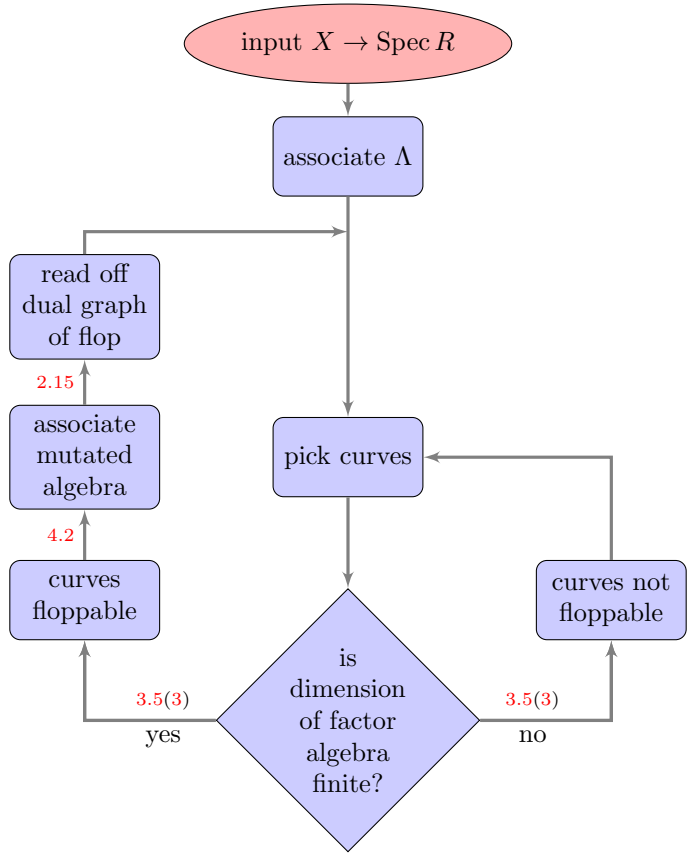
\begin{figure}[!h]{\vspace{2mm}}
\begin{center}
\begin{tikzpicture}[scale=1.75, node distance = 1.75cm, auto]
\node [cloud] (A) {input $X\to\Spec R$};
\node [block, below of=A,node distance=1.5cm] (B) {associate $\Lambda$};
\node [block, below of=B, node distance=4cm] (C) {pick curves};
\node [decision, below of=C, node distance=3.5cm] (Q) {is dimension of factor algebra finite?};
\coordinate [below of=B, node distance=1cm] (ch);
\coordinate [below of=B, node distance=2cm] (ch1);
\coordinate [below of=C, node distance=1.9cm] (ch2);
\node [block, left of=ch1, node distance=3.5cm] (Ba) {read off dual graph of flop};
\node [block, left of=C, node distance=3.5cm] (Ca) {associate mutated algebra};
\node [block, left of=ch2, node distance=3.5cm] (Da) {curves floppable};
\node [block, right of=ch2,node distance=3.5cm] (Db) {curves not floppable}; 
\path [line] (A) -- (B);
\path [line] (B) -- (C);
\path [line] (C) -- (Q);
\path [line] (Q) -| node [above,pos=0.2, color=black] {\scriptsize \ref{contract on f}\eqref{contract on f 3}} (Da);
\path [line] (Q) -| node [pos=0.2, color=black] {yes} (Da);
\path [line] (Da) -- node[pos=0.5,color=black] {\scriptsize \ref{flop=mut general thm}} (Ca);
\path [line] (Ca) -- node[pos=0.5,color=black] {\scriptsize \ref{reconstruction new}} (Ba);
\path [line] (Ba) |- (ch);
\path [line] (Q) -| node [above,pos=0.2, color=black] {\scriptsize \ref{contract on f}\eqref{contract on f 3}} (Db);
\path [line] (Q) -| node [below,pos=0.2, color=black] {no} (Db);
\path [line] (Db) |- (C);
\end{tikzpicture}
\end{center}
\caption{The Homological MMP, jumping between minimal models.}\label{Fig2}
\end{figure}

\subsection{Applications to GIT}\label{GIT intro}  There are various other outputs to the Homological MMP that for clarity have not been included in Figure~\ref{Fig2}.  One such output, when the curve does flop, is obtained by combining \ref{mut functor thm intro}\eqref{mut thm 2 intro} with \cite[5.2.5]{Joe}. This shows that it is possible to output the flop as a fixed, specified, GIT moduli space of the mutated algebra. 

As notation, for any algebra $\Gamma:=\End_R(N)$ with $\Gamma\in
\CM R$, we denote the dimension vector given by the ranks of the summands of $N$ by $\rk$.  If $N$ is a generator, that is $N$ contains $R$ as a summand, then the GIT chamber decomposition $\Uptheta(\Gamma)$ associated to $\Gamma$ (with dimension vector $\rk$) has co-ordinates $\upvartheta_i$ for $i\neq 0$, where by convention $\upvartheta_0$ corresponds to the summand $R$ of $N$.  We consider the region
\[
C_+(\Gamma):=\{ \upvartheta\in\Uptheta(\Gamma)\mid \upvartheta_i>0\mbox{ for all }i>0\}.
\]
As usual, we let $\cM_{\rk,\upphi}(\Lambda)$ denote the moduli space of $\upphi$-semistable $\Lambda$-modules of dimension vector $\rk$.
\begin{cor}[{=\ref{mut moduli gives flop text}}]\label{mut moduli gives flop}
With setup $X\to\Spec R$ as above, choose $C_i$ and suppose that $\dim_\mathbb{C}\Lambda_i<\infty$, so that $C_i$ flops. Then 
\begin{enumerate}
\item\label{mut moduli gives flop 1} $\cM_{\rk,\upphi}(\Lambda)\cong X$ for all $\upphi\in C_+(\Lambda)$.
\item\label{mut moduli gives flop 2} $\cM_{\rk,\upphi}(\upnu_i\Lambda)\cong X^+$ for all $\upphi\in C_+(\upnu_i\Lambda)$.
\end{enumerate}
\end{cor}

This allows us to view the flop as changing the algebra but keeping the GIT chamber structure fixed, and so since mutation is easier to control than GIT wall crossing, \ref{GIT intro} implies, but is not implied by, results in GIT.   Mutation always induces a derived equivalence, and it turns out that it is possible to track the moduli space in \ref{mut moduli gives flop}\eqref{mut moduli gives flop 2} back across the equivalence to obtain the flop as a moduli space on the original algebra. Again, as in Step 1 in \S\ref{algorithm intro}, the subtlety is that the flop of Bridgeland--Chen is constructed as a moduli with respect to the morphism $g$ in \eqref{contraction f and g}, whereas here we want to establish the flop as a moduli with respect to global information associated to the morphism $f$.

The following moduli--tracking theorem allows us to do this.  Later, we prove it in much greater generality, and with multiple summands.

\begin{prop}[={\ref{stab change b and c}\eqref{stab change b and c 1}}]\label{SY NS formula intro}
Let $S$ be a $d$-dimensional complete local normal Gorenstein ring,  $M\in\refl R$ with $\Lambda:=\End_S(M)\in\CM R$, and suppose that $\upnu_iM$ satisfies the technical assumptions in \ref{main stab track}.   Consider the minimal left $\add (\oplus_{j\neq i}M_j)$-approximation of $M_i$, namely
\[
0\to M_i\to\bigoplus_{j\neq i}M_j^{\oplus b_j}
\]
Suppose that $\upbeta$ is a dimension vector, and $\upvartheta$ is a stability condition on $\Lambda$ with $\upvartheta_i>0$.  Then as schemes $\cM_{\upbeta,\upvartheta}(\Lambda)\cong \cM_{\upnu_{i}\upbeta,\upnu_{i}\upvartheta}(\upnu_i\Lambda)$, where the vectors $\upnu_i\upbeta$ and $\upnu_i\upvartheta$ are given by
\[
(\upnu_i\upbeta)_t=\left\{\begin{array}{cl} \upbeta_t&\mbox{if }t\neq i \\ \left(\sum_{j\neq i}b_j\upbeta_j\right)-\upbeta_i&\mbox{if }t=i \end{array}\right.
\qquad
(\upnu_i\upvartheta)_t=\left\{\begin{array}{cl} \upvartheta_t+b_t\upvartheta_i&\mbox{if }t\neq i \\ -\upvartheta_i&\mbox{if }t=i \end{array}\right.
\]
\end{prop}
The technical assumptions in \ref{SY NS formula intro} hold for flopping contractions, and they also hold automatically for any noncommutative crepant resolution (=NCCR) or maximal modification algebra (=MMA) in dimension three.  Thus \ref{SY NS formula intro} can be applied to situations where the fibre is two--dimensional, and we expect to be able to extend some of the techniques in this paper to cover general minimal models of general Gorenstein 3-folds.  We also remark that \ref{SY NS formula intro} is known in special situations; it generalises \ \cite[3.6, 4.20]{SY}, which dealt with Kleinian singularities, and \cite[6.12]{NS}, which dealt with specific examples of smooth 3-folds with mutations of NCCRs given by quivers with potentials at vertices with no loops.

It is also possible to track moduli from $\upnu_i\Lambda$ to moduli on $\Lambda$, see \ref{stab change b and c}\eqref{stab change b and c 2}.  This leads to the following corollary.

\begin{cor}[{=\ref{see the flop}}]\label{flop of original intro}
With the running hypothesis $f\colon X\to \Spec R$ as above, assume that either $f$ is a flopping contraction, or a minimal model. Let $\Lambda:=\End_R(N)$ from \eqref{VdBfunctor}, where $N$ automatically has $R$ as a summand, and consider the GIT chamber decomposition $\Uptheta$ associated to $\Lambda$, with co-ordinates $\upvartheta_i$ for $i\neq 0$ (where $\upvartheta_0$ corresponds to the summand $R$ of $N$).  Pick an indecomposable non-free summand $N_i$, and consider the $b_j$ defined in \ref{SY NS formula intro} (for the case $M:=N$). Then the region
\[
\upvartheta_i<0,\quad\upvartheta_j+b_j\upvartheta_i>0\mbox{ for all }j\neq i
\]
defines a chamber in $\Uptheta(\Lambda)$, and for any parameter $\upvartheta$ inside this chamber,
\[
\cM_{\rk, \upvartheta}(\Lambda)\cong\left\{  \begin{array}{ll}
X^+&\mbox{if $C_i$ flops}\\
X&\mbox{else,}
\end{array}\right.
\]
where $X^+$ denotes the flop of $X$ at $C_i$.  Thus the flop, if it exists, is obtained by crashing through the single wall $\upvartheta_i=0$ in $\Uptheta(\Lambda)$.
\end{cor}

Of course, our viewpoint is that \ref{flop of original intro} should be viewed as a \emph{consequence} of the Homological MMP, since without the extra data the Homological MMP offers, it is hard to say which should be the next wall to crash through, and then which wall to crash through after that.  The information in the next chamber needed to iterate is contained in $\upnu_i\Lambda$, not the original $\Lambda$.  Mutation allows us to successfully track this data, and as a consequence we obtain the following corollary.

\begin{cor}[={\ref{Craw Ishii true}\eqref{Craw Ishii true 1}}]\label{path in GIT intro}
There exists a connected path in the GIT chamber decomposition of $\Lambda$ where every minimal model of $\Spec R$ can be found, and each wall crossing in this path corresponds to the flop of a single curve.
\end{cor}

We remark that \ref{path in GIT intro} was verified in specific quotient singularities in \cite[1.5]{NS}, and is also implicit in the setting of $cA_n$ singularities in \cite[\S6]{IW5}, but both these papers relied on direct calculations.  The Homological MMP removes the need to calculate.

The following conjecture is an extension to singular minimal models of a conjecture posed by Craw--Ishii \cite{CI}, originally for quotient singularities and their NCCRs.

\begin{conj}[Craw--Ishii]\label{CI conj}
Suppose that $S$ is an arbitrary complete local normal Gorenstein 3-fold with rational singularities, and $\End_R(N)$ is an MMA where $R\in\add N$.  Then every projective minimal model of $\Spec R$ can be obtained as a quiver GIT moduli space of $\End_R(N)$.
\end{conj}

There are versions of the conjecture for rings $R$ that are not complete local, but in the absence of a grading, which for example exists for quotient singularities, there are subtleties due to the failure of Krull--Schmidt.  Nevertheless, a direct application of \ref{path in GIT intro} gives the following result.

\begin{cor}[={\ref{Craw Ishii true}\eqref{Craw Ishii true 2}}]\label{CI true intro}
The Craw--Ishii conjecture is true for all compound du Val (=cDV) singularities.
\end{cor}

In fact we go further than \ref{path in GIT intro} and \ref{CI true intro}, and describe the whole GIT chamber structure.  In principle this is hard, since obtaining the numbers $b_j$ needed in \ref{SY NS formula intro} directly on the 3-fold is difficult without explicit knowledge of $\Lambda$ or indeed without knowing the explicit equation defining $R$. However, the next result asserts that mutation is preserved under generic hyperplane sections, and this allows us to obtain the numbers $b_j$ by reducing to the case of Kleinian surface singularities, about which all is known. 

\begin{lemma}[={\ref{chamber 3fold=chamber surface prep}}]\label{mut generic hyper intro}
With the setup $X\to\Spec R$ as above, if $g$ is a sufficiently generic hyperplane section, then $\Lambda/g\Lambda\cong\End_{R/gR}(N/gN)$, and minimal approximations are preserved under tensoring by $R/gR$.
\end{lemma}

For a more precise wording, see \ref{chamber 3fold=chamber surface prep}.  Now by Reid's general elephant conjecture, true in the setting here by \cite[1.1, 1.14]{Pagoda}, cutting by a generic hyperplane section yields
\[
\begin{array}{c}
\begin{tikzpicture}[yscale=1.25]
\node (Xp) at (-1,0) {$X_2$}; 
\node (X) at (1,0) {$X$};
\node (Rp) at (-1,-1) {$\Spec (R/g)$}; 
\node (R) at (1,-1) {$\Spec R$};
\draw[->] (Xp) to node[above] {$\scriptstyle $} (X);
\draw[->] (Rp) to node[above] {$\scriptstyle $} (R);
\draw[->] (Xp) --  node[left] {$\scriptstyle \upvarphi$}  (Rp);
\draw[->] (X) --  node[right] {$\scriptstyle f$}  (R);
\end{tikzpicture}
\end{array}
\]
where $R/g$ is an ADE singularity and $\upvarphi$ is a partial crepant resolution.  Since $N\in\CM R$ and $g$ is not a zero-divisor on $N$, necessarily $N/gN\in\CM R/g$, and so any indecomposable summand $N_i$ of $N$ cuts to $N_i/gN_i$, which must correspond to a vertex in an ADE Dynkin diagram via the Auslander--McKay correspondence. This then allows us to obtain the numbers $b_j$ using Auslander--Reiten (=AR) theory, using the knitting--type constructions on the known AR quivers, as in \cite{IW}.  We refer the reader to \S\ref{calculate chamber subsection} for details, in particular the example \ref{knitting example}.

Once we have obtained the $b_j$ for all exchange sequences, which in particular depends only on the curves which  appear in the partial resolution $X_2$, we are able to use this data to do two things.   First, we are able to compute the full GIT chamber structure.

\begin{cor}[={\ref{Cplus is chamber}, \ref{surface intersection root}, \ref{chamber 3fold=chamber surface}}]\label{chamber intro}
In the setup $X\to\Spec R$ above, suppose that $f$ is a minimal model, or a flopping contraction.  Set $\Lambda:=\End_R(N)$ from \eqref{VdBfunctor}.  Then
\begin{enumerate}
\item $C_+(\Lambda)$ is a chamber in $\Uptheta$.
\item For sufficiently generic $g\in R$, the chamber structure of $\Uptheta$ for $\Lambda$ is the same as the chamber structure for $\End_{R/gR}(N/gN)$. There are a finite number of chambers, and the walls are given by a finite collection of hyperplanes containing the origin.  The co-ordinate hyperplanes $\upvartheta_i=0$ are included in this collection.
\item Tracking all the chambers $C_+(\upnu_{i_t}\hdots\upnu_{i_1}\Lambda)$ through mutation, via knitting combinatorics, gives the full chamber structure of $\Uptheta$.
\end{enumerate}
\end{cor}

We list and draw some examples in \ref{knitting example} and \S\ref{examples section}.  In the course of the proof of \ref{chamber intro}, if $\Pi$ denotes the preprojective algebra of an extended Dynkin diagram and $e$ is an idempotent containing the extending vertex, then in \ref{surface intersection root} we describe the chamber structure of $\Uptheta(e\Pi e)$ by intersecting hyperplanes with a certain subspace in a root system, a result which may be of independent interest.  It may come as a surprise that the resulting chamber structures are \emph{not} in general the root system of a Weyl group, even up to an appropriate change of parameters, and this has implications to the braiding of flops \cite{DW3} and faithful group actions \cite{HW}.  It also means, for example, that any naive extension of \cite{TodaResPub} or \cite{BIKR, DH} is not possible, since root systems and Weyl groups do not necessarily appear.  However, this phenomenon will come as no surprise to Pinkham \cite[p366]{Pinkham}.
 
Second, we are able to give minimal as well as maximal bounds on the number of minimal models, based only on the curves which appear in the partial resolution $X_2$.  The Homological MMP enriches the GIT chamber structure not only with the mutated quiver (allowing us to iterate), but by \ref{mut generic hyper intro} it also enriches it with the information of the curves appearing after cutting by a generic hyperplane.  Certainly if two minimal models $X$ and $Y$ cut under generic hyperplane section to two different curve configurations, then $X$ and $Y$ must be different minimal models.  The surface curve configurations obtained via mutation can be calculated very easily using knitting combinatorics, so keeping track of this extra information (see e.g.\ \ref{enriched E7 example}) allows us to enhance the chamber structure, and to improve upon the results of \cite{Pinkham} as follows.

\begin{cor}[{=\ref{minimal bound}}]\label{minimal bound intro}
Suppose that $R$ is a $cDV$ singularity, with a minimal model $X\to\Spec R$.  Set $\Lambda:=\End_R(N)$ as in \eqref{VdBfunctor}.  By passing to a general hyperplane section $g$ as in \ref{chamber intro}, the number of minimal models of $\Spec R$ is bounded below by the number of different curve configurations obtained in the enhanced chamber structure of $\Uptheta(\Lambda/g\Lambda)$.
\end{cor}

A closer analysis (see e.g.\ \ref{knitting example 2}) reveals that it is possible to obtain better lower bounds, also by tracking mutation, but we do not detail this here. See \S\ref{calculate chamber subsection} and \S\ref{GIT examples section}.

\subsection{Auslander--McKay Correspondence}\label{other applications intro}  There are also purely algebraic outputs of the Homological MMP.   One such output is that we are able to lift the Auslander--McKay correspondence from dimension two \cite{Aus3} to $3$-fold compound du Val singularities.  One feature is that for 3-folds, unlike for surfaces, there are two correspondences.  First, there is a correspondence \eqref{basic correspondence intro} between maximal modifying (=MM) $R$-module generators and minimal models, and then for each such pair there is a further correspondence (in parts \eqref{NCvsC min model intro 1} and \eqref{NCvsC min model intro 2} below) along the lines of the classical Auslander--McKay Correspondence.  Parts \eqref{NCvsC min model intro 3} and \eqref{NCvsC min model intro 4}, the relationship between flops and mutation, describe how these two correspondences relate.

\begin{cor}[={\ref{NCvsC min model}, \ref{new thm in email}}]\label{NCvsC min model intro}
Let $R$ be a complete local cDV singularity.  Then there exists a one-to-one correspondence
\begin{eqnarray}
\begin{array}{c}
\begin{tikzpicture}
\node (A) at (-1,0) {$\{ \mbox{basic MM $R$-module generators}\}$};
\node (B) at (6,0) {$
\{\mbox{minimal models $f_i\colon X_i\to\Spec R$} \}
$};
\draw[<->] (1.85,0) -- node [above] {$\scriptstyle $} (3,0);
\node at (0,0.175) {$\phantom -$};
\end{tikzpicture}
\end{array}\label{basic correspondence intro}
\end{eqnarray}
where the left-hand side is taken up to isomorphism, and the right-hand side is taken up to isomorphism of the $X_i$ compatible with the morphisms $f_i$.  Under this correspondence
\begin{enumerate}
\item\label{NCvsC min model intro 1} For any fixed MM generator, its non-free indecomposable summands are in one-to-one correspondence with the exceptional curves in the corresponding minimal model.
\item\label{NCvsC min model intro 2} For any fixed MM generator $N$, the quiver of $\uEnd_R(N)$ (for definition see \ref{stable end defin}) encodes the dual graph of the corresponding minimal model.
\item\label{NCvsC min model intro 3} The full mutation graph of the MM generators coincides with the full flops graph of the minimal models.
\item\label{NCvsC min model intro 4} The derived mutation groupoid of the MM generators is functorially isomorphic to the derived flops groupoid of the minimal models.   
\end{enumerate}
\end{cor}

For all undefined terminology, and the detailed description of the bijection maps in \eqref{basic correspondence intro}, we refer the reader to \S\ref{general prelim}, \S\ref{AM subsection 1} and \S\ref{AM section 2}.  We remark that the graphs in \eqref{NCvsC min model intro 3} are simply the framework to express the relationship between flops and mutation on a combinatorial level, and the derived groupoids in \eqref{NCvsC min model intro 4} are the language to express the relationship on the level of functors.

In addition to \ref{NCvsC min model intro}, we also establish the following.  For unexplained terminology,  we again refer the reader to \S\ref{AM section 2}.

\begin{cor}[{=\ref{NCvsC min model}, \ref{NCvsC min model full}, \ref{isolated AM}}]\label{finite number intro}
Let $R$ be a complete local cDV singularity. Then
\begin{enumerate}
\item $R$ admits only finitely many MM generators, and any two such modules are connected by a finite sequence of mutations.
\item The mutation graph of MM generators can be viewed as a subgraph of the skeleton of the GIT chamber decomposition of $\Uptheta(\Lambda)$. 
\setcounter{tempenum}{\theenumi}
\end{enumerate}
If further $R$ is isolated, then
\begin{enumerate}
\setcounter{enumi}{\thetempenum}
\item\label{finite no intro 3} The mutation graph of MM generators coincides with the skeleton of the GIT chamber decomposition.  In particular, the number of basic MM generators equals the number of chambers. 
\end{enumerate}
\end{cor}
Although \eqref{finite no intro 3} is simply a special case, the setting when $R$ has only isolated singularities is particularly interesting since it relates maximal rigid and cluster tilting objects in certain Krull--Schmidt Hom-finite $2$-CY triangulated categories to birational geometry.  

We also remark that the above greatly generalises and simplifies \cite{BIKR, DH}, which considered isolated $cA_n$ singularities with smooth minimal models and observed the connection to the Weyl group $S_n$, \cite{NS} which considered specific quotient singularities, again with smooth minimal models, and \cite{IW5} which considered general $cA_n$ singularities.  All these previous works relied heavily on direct calculation, manipulating explicit forms.

Based on the above results, we offer the following conjecture.
\begin{conj}\label{1d finite conjecture}
Let $R$ be a Gorenstein $3$-fold with only rational singularities.  Then $R$ admits only a finite number of basic MM generators if and only if the minimal models of $\Spec R$ have one-dimensional fibres (equivalently, $R$ is cDV).
\end{conj}

The direction $(\Leftarrow)$ is true by \ref{finite number intro}.  Although we cannot yet prove $(\Rightarrow)$, by strengthening some results in \cite{Bergh} to cover non-isolated singularities, we do show the following as a corollary of a more general $d$-dimensional result.
\begin{prop}[=\ref{CT PB cor}]
Suppose that $R$ is a complete local $3$-dimensional normal Gorenstein ring, and suppose that $R$ admits an NCCR (which by \cite{VdBNCCR} implies that the minimal models of $\Spec R$ are smooth). If $R$ admits only finitely many basic MM generators up to isomorphism, then $R$ is a hypersurface singularity.
\end{prop}

\subsection{Generalities}\label{generalities}
In this paper we work over an affine base, restrict to complete local rings, work over one-dimensional fibres and sometimes restrict to  minimal models.  Often these assumptions are not necessary, and are mainly made just for technical simplification of the notation and exposition.  In Appendix \ref{gen appendix} we outline questions and conjectures for when $R$ is not Gorenstein, including flips and other aspects of the MMP.

\subsection{Notation and Conventions}\label{conventions}
Everything in this paper takes place over the complex numbers $\mathbb{C}$, or any algebraically closed field of characteristic zero.  All complete local rings appearing are the completions of finitely generated $\mathbb{C}$-algebras at some maximal ideal.   Throughout modules will be left modules, and  for a ring $A$, $\mod A$ denotes the category of finitely generated left $A$-modules, and $\fl A$ denotes the category of finite length left $A$-modules.  For $M\in\mod A$ we denote by $\add M$ the full subcategory consisting of summands of finite direct sums of copies of $M$.  We say that $M$ is a generator if $R\in\add M$,  and we denote by $\proj A:=\add A$ the category of finitely generated projective $A$-modules.  Throughout we use the letters $R$ and $S$ to denote commutative noetherian rings,  whereas Greek letters $\Lambda$ and $\Gamma$ will denote noncommutative noetherian rings.

We use the convention that when composing maps $fg$, or $f\cdot g$, will mean $f$ then $g$, and similarly for quivers $ab$ will mean $a$ then $b$.   Note that with this convention $\Hom_R(M,X)$ is a $\End_R(M)$-module and $\Hom_R(X,M)$ is a $\End_R(M)^{\rm op}$-module.  Functors will use the opposite convention, but this will  always  be notated by the composition symbol $\circ$, so throughout $F\circ G$ will mean $G$ then $F$.

\subsection{Acknowledgements}  
It is a pleasure to thank Alastair Craw and Alastair King for many discussions and comments over the eight years since this work began, and to thank Will Donovan and Osamu Iyama for many  discussions and ideas, some of which are included here, and for collaborating on the foundational aspects \cite{IW4, IW5, IW6, DW1, DW2, DW3}.  It is also a pleasure to thank the referee for many helpful comments and suggestions.  Thanks are also due to Gavin Brown, Alvaro Nolla de Celis, Hailong Dao, Eleonore Faber, Martin Kalck, Joe Karmazyn, Miles Reid and Yuhi Sekiya.

\section{General Preliminaries}

We begin by outlining the necessary preliminaries on aspects of the MMP, MM modules, MMAs, perverse sheaves, and mutation.  With the exception of \ref{reconstruction new},  \ref{Ext 2 thm}, \ref{Ext 3 thm}, \ref{mutmutI general} and \ref{surfaces hack} nothing in this section is original to this paper, and so the confident reader can skip to \S\ref{contractions and deformations section}.

\subsection{General Background}\label{general prelim}
If $(R,\m)$ is a commutative noetherian local ring and $M\in\mod R$, recall that the \emph{depth} of $M$ is defined to be
\[
\depth_R M:=\inf \{ i\geq 0\mid \Ext^i_R(R/\m,M)\neq 0 \}.
\]
We say that $M\in\mod R$ is \emph{maximal Cohen-Macaulay} (=CM) if $\depth_R M=\dim R$.   In the non-local setting, if $R$ is an arbitrary commutative noetherian ring we say that $M\in\mod R$ is \emph{CM} if $M_\p$ is CM for all prime ideals $\p$ in $R$, and we denote the category of CM $R$-modules by $\CM R$.  We say that $R$ is a \emph{CM ring} if $R\in\CM R$, and if further $\id_{R}R<\infty$, we say that $R$ is \emph{Gorenstein}.  Throughout, we denote $(-)^*:=\Hom_R(-,R)$ and let $\refl R$ denote the category of \emph{reflexive} $R$-modules, that is those $M\in\mod R$ for which the natural morphism $M\to M^{**}$ is an isomorphism.

Singular $d$-CY algebras are a convenient language that unify the commutative Gorenstein algebras and the mildly noncommutative algebras under consideration.
\begin{defin}
Let $\Lambda$ be a module finite $R$-algebra, then for $d\in\mathbb{Z}$ we call $\Lambda$ $d$-Calabi--Yau (=$d$-CY) if there is a functorial isomorphism 
\[
\Hom_{\D(\Mod \Lambda)}(x,y[d])\cong D\Hom_{\D(\Mod \Lambda)}(y,x)
\]
for all $x\in\Db(\fl \Lambda)$, $y\in\Db(\mod\Lambda)$, where $D=\Hom_{\mathbb{C}}(-,\mathbb{C})$.  Similarly we call $\Lambda$ singular $d$-Calabi--Yau (=$d$-sCY) if the above functorial isomorphism holds for all $x\in\Db(\fl\Lambda)$ and $y\in\Kb(\proj \Lambda)$, where $\Kb(\proj \Lambda)$ denotes the subcategory of $\Db(\mod\Lambda)$ consisting of perfect complexes.
\end{defin}
When $\Lambda=R$, it is known \cite[3.10]{IR} that $R$ is $d$-sCY if and only if $R$ is Gorenstein and equi-codimensional with $\dim R=d$.   One noncommutative source of $d$-sCY algebras are maximal modification algebras, introduced in \cite{IW4} as the notion of a noncommutative minimal model.  
\begin{defin}\label{MMdefin}
Suppose that $R$ is a normal $d$-sCY algebra.  We call $N\in\refl R$ a \emph{modifying module} if $\End_R(N)\in\CM R$, and we say that $N\in\refl R$ is a \emph{maximal modifying (MM) module} if it is modifying and it is maximal with respect to this property.  Equivalently, $N\in\refl R$ is an MM module if and only if
\[
\add N=\{ X\in\refl R\mid \End_{R}(N \oplus X)\in\CM R  \}.
\]
If $N$ is an MM module, we call $\End_R(N)$ a \emph{maximal modification algebra (=MMA).}  
\end{defin}
The notion of a smooth noncommutative minimal model, called a noncommutative crepant resolution, is due to Van den Bergh \cite{VdBNCCR}.
\begin{defin}
Suppose that $R$ is a normal $d$-sCY algebra.  By a \emph{noncommutative crepant resolution} (NCCR) of $R$ we mean $\Lambda:=\End_R(N)$ where $N\in\refl R$ is such that $\Lambda\in\CM R$ and $\gl\Lambda=d$. 
\end{defin}
In the setting of the definition, provided that $N$ is nonzero, it is equivalent to ask for $\Lambda\in\CM R$ and $\gl\Lambda<\infty$ \cite[4.2]{VdBNCCR}.  Note that any modifying module $N$ gives rise to a $d$-sCY algebra $\End_R(N)$ by \cite[2.22(2)]{IW4}, and $\End_R(N)$ is $d$-CY if and only if $\End_R(N)$ is an NCCR \cite[2.23]{IW4}.  Further, an NCCR is precisely an MMA with finite global dimension, that is, a smooth noncommutative minimal model.  On the base $R$, those NCCRs where $N\in \CM R$ can be characterised in terms of CT modules \cite[5.9(1)]{IW4}.
\begin{defin}\label{CT defin}
Suppose that $R$ is a normal $d$-sCY algebra.  We say that $N\in\CM R$ is a CT module if
\[
\add N=\{ X\in\CM R\mid \Hom_R(N,X)\in\CM R\}.
\]
\end{defin}

Throughout this paper we will freely use the language of terminal, canonical and compound Du Val (=cDV) singularities in the MMP, for which we refer the reader to \cite{CKM, Pagoda, KollarMori} for a general overview.  Recall that a normal scheme $X$ is defined to be {\em $\mathds{Q}$-factorial} if for every Weil divisor $D$, there exists $n\in\mathbb{N}$ for which $nD$ is Cartier.  Also, if $X$ and $X_{\con}$ are normal, then recall that a projective birational morphism $f\colon X\to X_{\con}$ is called {\em crepant} if $f^*\omega_{X_{\con}}=\omega_X$. A {\em $\mathds{Q}$-factorial terminalisation}, or \emph{minimal model}, of $X_{\con}$ is a crepant projective birational morphism $f\colon X\to X_{\con}$ such that $X$ has only $\mathds{Q}$-factorial terminal singularities. When $X$ is furthermore smooth, we call $f$ a {\em crepant resolution}.  

The following theorem, linking commutative and noncommutative minimal models, will be used implicitly throughout.
\begin{thm}{\cite[4.16, 4.17]{IW5}}\label{Min model - MMA}
Let $f\colon X\to\Spec R$ be a projective birational morphism, where $X$ and $R$ are both Gorenstein normal varieties of dimension three, and $X$ has at worst terminal singularities. If $X$ is derived equivalent to some ring $\Lambda$, then the following are equivalent.
\begin{enumerate}
\item $X\to\Spec R$ is a minimal model.
\item $\Lambda$ is an MMA of $R$.
\end{enumerate}
\end{thm}

The result is also true when $R$ is complete local, see \cite[4.19]{IW5}.     

Throughout this paper, we require the ability to contract curves.  Suppose that $f\colon X\to \Spec R$ is a projective birational morphism where $R$ is complete local, such that $\RDerived f_*\cO_X=\cO_R$, with at most one-dimensional fibres.  Choose a subset of curves $\bigcup_{i\in I}C_i$ in $X$ above the unique closed point of $\Spec R$, then since $R$ is complete local we may factorise $f$ into
\[
X\xrightarrow{g} X_{\con}\xrightarrow{h} \Spec R
\]
where $g$ contracts $C_j$ to a closed point if and only if $j\in I$, and further $g_*\cO_X=\cO_{X_{\con}}$, see e.g.\ \cite[p25]{KollarFlops} or \cite[\S2]{SS}.  Further, by the vanishing theorem \cite[1-2-5]{KMM} $\RDerived g_*\cO_X=\cO_{X_{\con}}$, which since $\RDerived f_*\cO_X=\cO_{R}$ in turn implies that $\RDerived h_*\cO_{X_{\con}}=\cO_{R}$.

Recall that a $\mathds{Q}$-Cartier divisor $D$ is called $g$-nef if $D\cdot C\geq 0$ for all curves contracted by $g$, and $D$ is called $g$-ample if $D\cdot C> 0$ for all curves contracted by $g$.  There are many (equivalent) definitions of flops in the literature, see e.g.\ \cite{KollarSurvey}.  We will use the following.

\begin{defin}\label{flopsdefin}
Suppose that $f\colon X\to \Spec R$ is a crepant projective birational morphism, where $R$ is complete local, with at most one-dimensional fibres.  Choose $\bigcup_{i\in I}C_i$ in $X$, contract them to give $g\colon X\to X_{\con}$, and suppose that $g$ is an isomorphism away from $\bigcup_{i\in I}C_i$.  Then we say that $g^+\colon X^+\to X_{\con}$ is the flop of $g$ if for every line bundle $\cL=\cO_X(D)$ on $X$ such that $-D$ is $g$-nef, then the proper transform of $D$ is $\mathds{Q}$-Cartier, and $g^+$-nef. 
\end{defin}

The following is obvious, and will be used later.

\begin{lemma}\label{check flop}
With the setup in \ref{flopsdefin}, suppose that $D_i$ is a Cartier divisor on $X$ such that $D_i\cdot C_j=\delta_{ij}$ for all $i,j\in I$ (such a $D_i$ exists since $R$ is complete local), let $D'_i$ denote the proper transform of $-D_i$ to $X^+$.  Then if $D_i'$ is Cartier and there is an ordering of the exceptional curves $C_i^+$ of $g^+$ such that $D'_i \cdot C^+_j  =\delta_{ij}$, then $g^+\colon X^+\to X_{\con}$ is the flop of $g$.
\end{lemma}

\subsection{Perverse Sheaves and Tilting}\label{perverse and tilting}  Some of the arguments in this paper are not specific to dimension three, and are not specific to crepant morphisms.  Consequently, at times we will refer to the following setup.

\begin{setup}\label{general setup}
(General Setup).  Suppose that $f\colon X\to \Spec R$ is a projective birational morphism, where $R$ is complete local, $X$ and $R$ are noetherian and normal, such that $\RDerived f_*\cO_X=\cO_R$ and the fibres of $f$ have dimension at most  one. 
\end{setup}

However, some parts will require the following restriction.

\begin{setup}\label{crepant setup}
(Crepant Setup).  Suppose that $f\colon X\to \Spec R$ is a crepant projective birational morphism between $d\leq 3$ dimensional schemes, where $R$ is complete local normal Gorenstein, and the fibres of $f$ have dimension at most  one.   Further
\begin{enumerate}
\item If $d=2$ we allow $X$ to have canonical Gorenstein singularities, so $X\to\Spec R$ is a partial crepant resolution of a Kleinian singularity.
\item If $d=3$ we further assume that $X$ has only Gorenstein terminal singularities.
\end{enumerate}  
By Kawamata vanishing, it is automatic that $\RDerived f_*\cO_X=\cO_R$.  We will not assume that $X$ is $\mathds{Q}$-factorial unless explicitly stated. 
\end{setup}

Now if $g\colon X\to X_{\con}$ is a projective birational morphism satisfying $\RDerived g_*\cO_X=\cO_{X_{\con}}$, the category of perverse sheaves relative to $g$, denoted $\Per(X,X_{\con})$, is defined to be 
\[
\Per (X,X_{\con}):=\left\{ a\in\Db(\coh X)\left| 
\begin{array}{c}
H^i(a)=0\mbox{ if }i\neq 0,-1\\
g_*H^{-1}(a)=0\mbox{, }\RDerived g_* H^0(a)=0\\ 
\Hom(c,H^{-1}(a))=0\mbox{ for all }c\in\cC_g 
\end{array}\right. \right\}
\]
where $\cC_g:=\{ c\in\coh X\mid \RDerived g_*c=0\}$.  In the setup of \ref{general setup},  it is well-known \cite[3.2.8]{VdB1d} that there is a vector bundle $ \cV_X$, described below,  inducing a derived equivalence
\begin{eqnarray}
\begin{array}{c}
\begin{tikzpicture}
\node (a1) at (0,0) {$\Db(\coh X)$};
\node (a2) at (6,0) {$\Db(\mod \End_X(\cV_X))$};
\node (b1) at (0,-1.25) {$\Per (X,R)$};
\node (b2) at (6,-1.25) {$\mod\End_X(\cV_X) $};
\draw[->] (a1) -- node[above] {$\scriptstyle\Uppsi_X:=\RHom_X(\cV_X,-)$} node [below] {$\scriptstyle\sim$} (a2);
\draw[->] (b1) --  node [below] {$\scriptstyle\sim$} (b2);
\draw[right hook->] (b1) to (a1);
\draw[right hook->] (b2) to (a2);
\end{tikzpicture}
\end{array}\label{derived equivalence}
\end{eqnarray}
The bundle $\cV_X$ is constructed as follows.  Consider $C=\pi^{-1}(\m)$ where $\m$ is the unique closed point of $\Spec R$, then giving $C$ the reduced scheme structure, write $C^{\redu}=\bigcup _{i=1}^nC_i$ where each $C_i\cong\mathbb{P}^1$.  Since $R$ is complete local, we can find Cartier divisors $D_i$ with the property that $D_i\cdot C_j=\delta_{ij}$, and set $\cL_i:=\cO_X(D_i)$.  If the multiplicity of $C_i$ is equal to one, set $\cM_i:=\cL_i$, else define $\cM_i$ to be given by the maximal extension
\begin{eqnarray}
0\to\cO_{X}^{\oplus(r-1)}\to\cM_i\to\cL_i\to 0\label{max extension}
\end{eqnarray}
associated to a minimal set of $r-1$ generators of $H^1(X,\cL_i^{*})$ \cite[3.5.4]{VdB1d}.  

\begin{notation}
With notation as above, in the general setting of \ref{general setup}, 
\begin{enumerate}
\item Set $\cN_i:=\cM_i^*$, and $\cV_X:=\cO_{X}\oplus \bigoplus_{i=1}^n\cN_i$.
\item Set $N_i:=H^0(\cN_i)$ and $N:=H^0(\cV_X)$.
\end{enumerate}
\end{notation}
By \cite[3.5.5]{VdB1d}, $\cV_X$ is a basic progenerator of $\Per(X,R)$, and furthermore is a tilting bundle on $X$.   Note that   $\rank_{R}^{\phantom 1}N_i$ is equal to the scheme-theoretic multiplicity of the curve $C_i$ \cite[3.5.4]{VdB1d}. 

\begin{remark}\label{simple across Db}
Under the derived equivalence $\Uppsi_X$ in \eqref{derived equivalence}, the coherent sheaves $\cO_{C_i}(-1)$ belong to $\Per (X,R)$ and correspond to simple left $\End_X(\cV_X)$-modules $S_i$.
\end{remark}

Unfortunately, at this level of generality $\End_X(\cV_X)\ncong\End_R(N)$ (see e.g.\ \cite[\S2]{DW2}). However, in the crepant setup of \ref{crepant setup}, this does hold, which later will allow us to reduce many problems to the base $\Spec R$.  

\begin{lemma}[{\cite[3.2.10]{VdB1d}}]\label{flop up=down}
In the setup of \ref{crepant setup}, $\End_X(\cV_X)\cong\End_R(N)$.
\end{lemma}

\begin{notation}\label{N notation}
In the setup of \ref{general setup}, pick a subset  $\bigcup_{i\in I}C_i$ of curves above the origin, indexed by a (finite) set $I$.  We set
\begin{enumerate}
\item $\cN_I:=\bigoplus_{i\in I}\cN_i$ and $\cN_{I^c}:=\cO_X\oplus\bigoplus_{j\notin I}\cN_j$, so that $\cV_X=\cN_I\oplus \cN_{I^c}$.
\item $N_I:=\bigoplus_{i\in I}N_i$ and $N_{I^c}:=R\oplus\bigoplus_{j\notin I}N_j$, so that $N=N_I\oplus N_{I^c}$.
\end{enumerate}  
\end{notation}

 The following result is implicit in the literature.  

\begin{prop}\label{KIWY}
Under the general setup of \ref{general setup}, choose a subset of curves $\bigcup_{i\in I}C_i$ and contract them to obtain $X\to X_{\con}\to\Spec R$.  Let $e_{I^c}$ be the idempotent in $\End_X(\cV_X)$ corresponding to the summand $\cN_{I^c}$, and let $e$ be the idempotent in $\End_{X_{\con}}(\cV_{X_{\con}})$ corresponding to the summand $\cO_{X_{\con}}$.  Then
\begin{enumerate}
\item\label{KIWY 1} $\cV_{X_{\con}}\cong g_*\cN_{I^c}\cong \RDerived g_*\cN_{I^c}$, and $\End_X(\cN_{I^c})\cong\End_{X_{\con}}(\cV_{X_{\con}})$.  
\item\label{KIWY 2} The following diagram commutes  
\[
\begin{array}{c}
\begin{tikzpicture}
\node (a1) at (0,0) {$\Db(\coh X)$};
\node (a2) at (4.5,0) {$\Db(\mod \End_X(\cV_X))$};
\node (b1) at (0,-1.5) {$\Db(\coh X_{\con})$};
\node (b2) at (4.5,-1.5) {$\Db(\mod \End_{X_{\con}}(\cV_{X_{\con}}))$};
\node (c1) at (0,-3) {$\Db(\coh \Spec R)$};
\node (c2) at (4.5,-3) {$\Db(\mod R)$};
\draw[->] (a1) -- node[above] {$\scriptstyle\Uppsi_X$} node [below] {$\scriptstyle\sim$} (a2);
\draw[->] (b1) -- node[above] {$\scriptstyle\Uppsi_{X_{\con}}$} node [below] {$\scriptstyle\sim$} (b2);
\draw[->] (c1) --  node [above] {$\scriptstyle\sim$} (c2);
\draw[->] (a1) to node[left] {$\scriptstyle \RDerived g_*$} (b1);
\draw[->] (b1) to node[left] {$\scriptstyle \RDerived h_*$} (c1);
\draw[->] (a2) to node[right] {$\scriptstyle e_{I^c}(-)$} (b2);
\draw[->] (b2) to node[right] {$\scriptstyle e(-)$} (c2);
\end{tikzpicture}
\end{array}
\]
\end{enumerate}
Further, under the crepant setup of \ref{crepant setup}, $\End_X(\cV_X)\cong\End_R(N)$ and $\End_{X_{\con}}(\cV_{X_{\con}})\cong \End_R(N_{I^c})$, so $\End_R(N_{I^c})$ is derived equivalent to $X_{\con}$ via the tilting bundle $\cV_{X_{\con}}$. 
\end{prop}
\begin{proof}
(1) As in \S\ref{general prelim}, by the vanishing theorem $\RDerived g_*\cO_X\cong\cO_{X_{\con}}$.  Given this,  the proof of \cite[4.6]{KIWY} (which considered surfaces and $\mPer$ instead) shows that $g^*(\cV_{X_{\con}}^*)\cong \cO_X\oplus_{j\notin I}\cM_j$. Thus
\[
\sHom_X(g^*\cV_{X_{\con}},\cO_X)\cong g^*\sHom_X(\cV_{X_{\con}},\cO_{X_{\con}})\cong 
\cO_X\oplus_{j\notin I}\cM_j
\]
and so dualizing gives $g^*\cV_{X_{\con}}\cong \cO_X\oplus_{j\notin I}\cN_j:=\cN_{I^c}$, where the right-hand side is a summand of $\cV_X$.  Applying $\RDerived g_*$ and using the projection formula
\[
\cV_{X_{\con}}\cong \RDerived g_*g^*\cV_{X_{\con}}\cong  \RDerived g_* \cN_{I^c}
\]
and so inspecting cohomology shows that $g_*\cN_{I^c}\cong \RDerived g_*\cN_{I^c}\cong \cV_{X_{\con}}$.  It follows that
\[
\End_{X_{\con}}(\cV_{X_{\con}})=\End_{X_{\con}}(g_*\cN_{I^c})\cong \Hom_X(g^*g_*\cN_{I^c},\cN_{I^c})\cong \End_X(\cN_{I^c}),
\]
and chasing through shows this isomorphism is a ring isomorphism.\\
(2) The commutativity of the top diagram follows from the functorial isomorphisms
\begin{align*}
\RHom_{X_{\con}}(\cV_{X_{\con}},\RDerived g_* (-))
&\cong 
\RHom_{X}(g^*\cV_{X_{\con}},-)\\
&\cong 
\RHom_{X}(\cN_{I^c},-)\\
&\cong 
e_{I^c}\RHom_{X}(\cV_X,-).
\end{align*}
with the bottom diagram being similar.  The last statements then follow from \ref{flop up=down}.
\end{proof}

The following is an easy extension of  \cite[3.2]{WemGL2}, and will be needed later to read off the dual graph after the flop.
\begin{thm}\label{reconstruction new}
In the general setup of \ref{general setup}, set $\Lambda:=\End_X(\cV_X)$. Then $\Lambda^{\op}$ can be written as a quiver with relations, where the quiver is given as follows: for every exceptional curve $C_i$ associate a vertex labelled $i$, and also associate a vertex $\star$ corresponding to $\cO_{X}$.  Then the number of arrows between the vertices is precisely
\begin{center}
\begin{tabular}{*3l}
\toprule
&Number of arrows&If setup \ref{crepant setup}, $d=3$ and $X$ is smooth \\
\midrule
$\star\rightarrow\star$ & $\dim_\mathbb{C}\Ext_X^1(\omega_C,\omega_C)$.\\
$i\rightarrow\star$ & $\dim_\mathbb{C}\Hom_X(\cO_{C_i}(-1),\omega_C)$\\
$\star\rightarrow i$& $\dim_\mathbb{C}\Ext_X^2(\omega_C,\cO_{C_i}(-1))$\\
$i\to i$&$\dim_\mathbb{C}\Ext_X^1(\cO_{C_i},\cO_{C_i})$& $=\left\{ \begin{array}{rl} 0& \mbox{if }(-1,-1)\mbox{-curve}\\ 1 &\mbox{if }(-2,0)\mbox{-curve}\\ 2&\mbox{if }(-3,1)\mbox{-curve} \end{array}\right.$\\
$i\rightarrow j$ & $\left\{ \begin{array}{rl} 1& \mbox{if }C_i\cap C_j=\{\mathrm{pt}\}\\ 0 &\mbox{else} \end{array}\right.$\\
\bottomrule\\
\end{tabular}
\end{center}
where in the bottom row $i\neq j$.
\end{thm}

\subsection{Mutation}\label{mut prelim}
Throughout this subsection $R$ denotes a normal $d$-sCY complete local commutative algebra, with $d\geq 2$, and $M\in\refl R$ denotes a basic modifying module $M$.   We summarise and extend the theory of mutation from \cite[\S6]{IW4} and \cite[\S 5]{DW1}.  

\begin{setup}\label{setup}
With assumptions as above, given the basic modifying $R$-module $M$, set $\Lambda:=\End_R(M)$ and pick a summand $M_I$ of $M$.
\begin{enumerate}
\item Denote $M_{I^c}$ to be the complement of $M_I$, so that
\[
M=M_I\oplus M_{I^c}.
\]
\item  We define $[M_{I^c}]$ to be the two-sided ideal of $\Lambda$ consisting of morphisms $M\to M$ which factor through a member of $\add M_{I^c}$.  We define $\Lambda_I:=\Lambda/[M_{I^c}]$.  Equivalently, if $e_{I}$ denotes the idempotent of $\Lambda=\End_R(M)$ corresponding to the summand $M_{I}$ of $M$, then $\Lambda_I=\Lambda/\Lambda (1-e_{I})\Lambda$.
\end{enumerate}
\end{setup}

Given our choice of summand $M_I$, we then mutate.  In the theory of mutation, the complement submodule $M_{I^c}$ is fixed, and the summand $M_I$ changes in a universal way.  Recall from \S\ref{general prelim} that $(-)^*:=\Hom_R(-,R)$. 

\begin{setup}\label{setup2} With the setup as in \ref{setup}, write $M_I=\bigoplus_{i\in I}M_i$ as a direct sum of indecomposables.  For each $i\in I$, consider a \emph{minimal right $(\add M_{I^c})$-approximation}
\[
V_i\xrightarrow{a_i}M_i
\]
of $M_i$, which by definition means that
\begin{enumerate}
\item $V_i\in\add M_{I^c}$ and $(\cdot a_i)\colon\Hom_R(M_{I^c},V_i)\to\Hom_R(M_{I^c},M_i)$ is surjective,
\item If $g\in\End_R(V_i)$ satisfies $a_i=ga_i$, then $g$ is an automorphism.
\end{enumerate}
Since $R$ is complete, such an $a_i$ exists and is unique up to isomorphism.  Denote $K_i:=\Ker a_i$, so there is an exact sequence
\begin{eqnarray}
0\to K_i\xrightarrow{c_i} V_i\xrightarrow{a_i}M_i\label{K0prime}
\end{eqnarray}
such that  
\begin{eqnarray}
0\to \Hom_R(M_{I^c},K_i)\xrightarrow{\cdot c_i} \Hom_R(M_{I^c},V_i)\xrightarrow{\cdot a_i}\Hom_R(M_{I^c},M_i)\to 0\label{K0aprime}
\end{eqnarray}
is exact.  Summing the sequences \eqref{K0prime} over all $i\in I$ gives an exact sequence
\begin{eqnarray}
0\to K_I\xrightarrow{c} V_I\xrightarrow{a}M_I\label{K0}
\end{eqnarray}
such that
\begin{eqnarray}
0\to \Hom_R(M_{I^c},K_I)\xrightarrow{\cdot c} \Hom_R(M_{I^c},V_I)\xrightarrow{\cdot a}\Hom_R(M_{I^c},M_I)\to 0\label{K0a}
\end{eqnarray}
is exact. 

Dually, for each $i\in I$, consider a minimal right $(\add M_{I^c}^*)$-approximation
\[
U_i^*\xrightarrow{b_i}M_i^*
\]
of $M_i^*$, and denote $J_i:=\Ker b_i$.  Thus 
\begin{gather}
0\to J_i\xrightarrow{d_i} U_i^*\xrightarrow{b_i}M_i^*\label{K1prime}\\
0\to \Hom_R(M_{I^c}^*,J_i)\xrightarrow{\cdot d_i}
\Hom_R(M_{I^c}^*,U_i^*)\xrightarrow{\cdot b_i}
\Hom_R(M_{I^c}^*,M_i^*) 
\to 0
\label{K1aprime}
\end{gather}
are exact.  Summing over all $i\in I$ gives exact sequences
\begin{gather}
0\to J_I\xrightarrow{d} U_I^*\xrightarrow{b}M_I^*\label{K1}\\
0\to \Hom_R(M_{I^c}^*,J_I)\xrightarrow{\cdot d}
\Hom_R(M_{I^c}^*,U_I^*)\xrightarrow{\cdot b}
\Hom_R(M_{I^c}^*,M_I^*) 
\to 0.
\label{K1a}
\end{gather}
\end{setup}

\begin{defin}\label{mut defin main}
With notation as above, 
\begin{enumerate}
\item We define the \emph{right mutation} of $M$ at $M_I$  as
\[
\upmu_{I}M:=M_{I^c}\oplus K_I,
\]
that is we remove the summand $M_I$ and replace it with $K_I$.
\item We define the \emph{left mutation} of $M$ at $M_I$ as
\[
\upnu_{I}M:=M_{I^c}\oplus (J_I)^*.
\]
\end{enumerate}
\end{defin}

In this level of generality, $\upnu_IM$ is not necessarily isomorphic to $\upmu_IM$.

\begin{remark}
Even if $M_I=M_i$ is indecomposable, when we view $\End_R(M)$ as a quiver with relations, with arrows $a$, and left projective $\End_R(M)$-modules $P_j$ corresponding to the indecomposable summands $M_j$, it is a common misconception that mutation can be defined using simply the arrows into (respectively, out of) the vertex $i$.   Indeed, we could consider the combinatorially defined morphisms
\[
\bigoplus_{\substack{\mathrm{head}(a)=i\\ \mathrm{tail}(a)\neq i}}P_{\mathrm{tail}(a)}\to P_i\quad\mbox{and}\quad
P_i\to\bigoplus_{\substack{\mathrm{tail}(a)=i\\ \mathrm{head}(a)\neq i}}P_{\mathrm{head}(a)}
\] 
which by reflexive equivalence arise from morphisms 
\[
\bigoplus_{\substack{\mathrm{head}(a)=i\\ \mathrm{tail}(a)\neq i}}M_{\mathrm{tail}(a)}\to M_i\quad\mbox{and}\quad
M_i\to\bigoplus_{\substack{\mathrm{tail}(a)=i\\ \mathrm{head}(a)\neq i}}M_{\mathrm{head}(a)}.
\]
However these morphisms are not approximations in general.  In other words, the mutation defined in \ref{mut defin main} above is \emph{not} in general a vertex tilt in the sense of Bridgeland--Stern \cite{BS}, and in full generality there is no simple combinatorial description of the decomposition of $U_I$ or $V_I$.  In the case of cDV singularities, we do give a combinatorial description later in \S\ref{chamber red to surface subsection} and \S\ref{calculate chamber subsection} by relating the problem to partial crepant resolutions of ADE singularities.  
\end{remark}

One of the key properties of mutation is that it always gives rise to a derived equivalence. With the setup as above, for the case of left mutation $\upnu_{I}M$, the derived equivalence between $\End_R(M)$ and $\End_R(\upnu_{I}M)$ is given by a tilting $\End_R(M)$-module $T_I$ constructed as follows.  By \eqref{K1D} there is an exact sequence
\[
0\to M_I\xrightarrow{b^*} U_I\to J_I^*
\]
obtained by dualizing \eqref{K1}.  Applying $\Hom_R(M,-)$ induces $(\cdot b^* )\colon\Hom_R(M,M_I)\to\Hom_R(M,U_I)$, so denoting the cokernel by $C_I$ we obtain an exact sequence
\begin{eqnarray}
0\to \Hom_R(M,M_I)\xrightarrow{\cdot b^*} \Hom_R(M,U_I)\to C_I\to 0.
\label{defin of CI}
\end{eqnarray}
The tilting $\End_R(M)$-module $T_I$ is defined to be $
T_I:=\Hom_R(M,M_{I^c})\oplus C_I$.  It turns out that $\End_\Lambda(T_I)\cong \End_R(\upnu_IM)$ \cite[6.7, 6.8]{IW4}, and there is always an equivalence 
\begin{equation}
\Upphi_I:=\RHom(T_I,-)\colon\Db(\mod\End_{R}(M))\to \Db(\mod\End_{R}(\upnu_{I}M)),\label{post referee}
\end{equation}
which is called the {\em mutation functor} \cite[6.8]{IW4}.  It is never the identity functor. On the other hand $\upnu_{I}M=M$ can happen (see e.g.\ \ref{proj res thm 1}).  Note that, by construction, $T_I$ has the structure of a $\Lambda$-$\Gamma$ bimodule, where $\Gamma:=\upnu_I\Lambda:=\End_\Lambda(T_I)\cong\End_R(\upnu_IM)$.  The following is elementary.
\begin{lemma}\label{pd 1 for TI both sides}
With notation as above, the following statements hold.
\begin{enumerate}
\item\label{pd 1 for TI both sides 1} $T_I$ is a tilting $\Lambda$-module with $\pd_\Lambda T_I=1$.
\item\label{pd 1 for TI both sides 2} $T_I$ is a tilting $\Gamma^{\op}\cong\End_R((\upnu_IM)^*)$-module, with $T_I\cong\Hom_R((\upnu_IM)^*,M_{I^c}^*)\oplus D_I$ where $D_I$ arises from the exact sequence
\[
0\to\Hom_R((\upnu_IM)^*,J_I)\xrightarrow{\cdot d}\Hom_R((\upnu_IM)^*,U_I^*)\to D_I\to 0
\]
of $\Gamma^{\op}$-modules.  Thus $\pd_{\Gamma^{\op}}T_I=1$ and $\End_{\Gamma^{\op}}(T_I)\cong \Lambda^{\op}$.
\end{enumerate}
\end{lemma}
\begin{proof}
(1) is \cite[6.8]{IW4},  and (2) follows from (1), see for example  \cite[2.2]{SY} or \cite[4.1]{KellerTilting}, \cite[2.6]{BBtilting}.  As a sketch proof, by \eqref{defin of CI} 
\[
0\to \Hom_R(M,M)\to \Hom_R(M,U_I\oplus M_{I^c})\to C_I\to 0.
\]
is exact, and applying $\Hom_\Lambda(-,T_I)$ gives an exact sequence
\begin{eqnarray}
0\to \Hom_\Lambda(C_I,T_I)\to\Hom_\Lambda(\Hom_R(M,U_I\oplus M_{I^c}),T_I)\to T_I\to 0\label{left gamma op for T}
\end{eqnarray}
of $\Gamma^{\op}$-modules.  Under the isomorphism $\End_\Lambda(T_I)\cong\End_R(\upnu_IM)$, 
\begin{align*}
\Hom_\Lambda(C_I,T_I) &\cong \Hom_R(J_I^*,\upnu_IM)\cong\Hom_R((\upnu_IM)^*,J_I),\\
\Hom_\Lambda(\Hom_R(M,U_I\oplus M_{I^c}),T_I) &\cong \Hom_R(U_I\oplus M_{I^c},\upnu_IM)\cong \Hom_R((\upnu_IM)^*,U_I^*\oplus M_{I^c}^*),
\end{align*}
so \eqref{left gamma op for T} is isomorphic to 
\[
0\to\Hom_R((\upnu_IM)^*,J_I)\xrightarrow{{\cdot d\choose 0}}\Hom_R((\upnu_IM)^*,U_I^*\oplus M_{I^c}^*)\to T_I\to 0,
\]
proving the statements by applying the analysis in (1) to $\End_R((\upnu_IM)^*)$.
\end{proof}

For our purposes later we will require the finer information encoded in the following two key technical results.  They are both an extension of \cite[\S6]{IW4} and \cite[\S4]{DW1}, and are proved using similar techniques, so we postpone the proofs until Appendix~\ref{appendix mut}.

\begin{thm}[{=\ref{Ext 2 thm mut section}}]\label{Ext 2 thm}
Suppose that $\upnu_I M\cong M$.  Then
\begin{enumerate}
\item\label{Ext 2 thm 1} $T_I=\Lambda(1-e_I)\Lambda$ and $\Gamma:=\End_\Lambda(T_I)\cong \Lambda$.
\item\label{Ext 2 thm 2} $\Omega_\Lambda\Lambda_I=T_I$, thus $\pd_\Lambda \Lambda_I=2$ and $\Ext^1_\Lambda(T_I,-)\cong\Ext^2_\Lambda(\Lambda_I,-)$.
\end{enumerate}
\end{thm}

\begin{thm}[{=\ref{Ext 3 thm mut section}}]\label{Ext 3 thm}
Suppose that $d= 3$, $\upnu_I \upnu_IM\cong M$ and $\dim_{\mathbb{C}}\Lambda_I<\infty$.  As above, set $\Gamma:=\End_\Lambda(T_I)\cong\End_R(\upnu_IM)$.  Then 
\begin{enumerate}
\item\label{Ext 3 thm 1} $T_I\cong \Hom_R(M,\upnu_IM)$.
\item\label{Ext 3 thm 2} $\Omega_\Lambda^2\Lambda_I=T_I$, thus $\pd_\Lambda \Lambda_I=3$ and $\Ext^1_\Lambda(T_I,-)\cong\Ext^3_\Lambda(\Lambda_I,-)$.
\end{enumerate}
\end{thm}

The following, one of the main results in \cite{IW4}, will allow us to establish properties non-explicitly when we restrict to minimal models and mutate at single curves.

\begin{thm}\label{basic2}
Suppose that $d=3$, and $M$ is a maximal modifying $R$-module with indecomposable summand $M_j$.  Set $\Lambda:=\End_R(M)$.   Then
\begin{enumerate}
\item\label{basic2 1} We have $\upmu_{j}(M)\cong\upnu_{j}(M)$.
\item\label{basic2 2} Always $\upnu_j\upnu_j(M)\cong M$.
\item\label{basic2 3} $\upnu_j(M)\ncong M$ if and only if $\dim_\mathbb{C}\Lambda_j<\infty$.
\item\label{basic2 4} $\upnu_j(M)\cong M$ if and only if $\dim_\mathbb{C}\Lambda_j=\infty$.
\end{enumerate}
\end{thm}
\begin{proof}
(1) and (2) are special cases of \cite[6.25]{IW4}.\\
(3)($\Rightarrow$) is \cite[6.25(2)]{IW4}, and (4)($\Rightarrow$) is \cite[6.25(1)]{IW4}. (3)($\Leftarrow$) is the contrapositive of (4)($\Rightarrow$), and (4)($\Leftarrow$) is the contrapositive of (3)($\Rightarrow$).
\end{proof}

\begin{remark}\label{dichotomy}
Theorem~\ref{basic2}\eqref{basic2 3}\eqref{basic2 4} shows that there is a dichotomy in the theory of mutation depending on whether the dimension of $\Lambda_j$ is finite or not.  In the flops setting, this dichotomy will correspond to the fact that in a 3-fold, an irreducible curve may or may not be floppable.  In either case we will obtain a derived equivalence from mutation, and the results \ref{Ext 2 thm} and \ref{Ext 3 thm} will allow us to control it.
\end{remark}

The above \ref{basic2} will allow us to easily relate flops and mutations in the case when $d=3$ and the singularities of $X$ are $\mathds{Q}$-factorial.  When we want to drop the $\mathds{Q}$-factorial assumption, or consider $d=2$ with canonical singularities, we will need the following.

\begin{prop}\label{mutmutI general}
With the crepant setup of \ref{crepant setup}, and notation from \ref{N notation}, choose a subset $\bigcup_{i\in I}C_i$ of curves above the origin and contract them to obtain $X\to X_{\con}\to \Spec R$.  If $X_{\con}$ has only isolated hypersurface singularities, then $\upnu_I\upnu_IN\cong N$ in such a way that $N_i$ mutates to $J_i^*$ mutates to $N_i$.
\end{prop}
\begin{proof}
Denote $\mathbb{F}:=\Hom_R(N_{I^c}^*,-)$.  The choice of curves gives us a summand $N_I$ such that $N_{I^c}$ is a generator.  This being the case, the right-hand morphisms in all the exchange sequences are all surjective.  By \eqref{K1aprime}
\[
0\to \mathbb{F}J_i\to \mathbb{F}U^*_i\to \mathbb{F}N_i^*\to 0
\] 
is exact.  Denoting $\Delta:=\End_R(N^*_{I^c})$, since $\mathbb{F}U^*_i$ is a projective $\Delta$-module, this shows that $\Omega_\Delta (\mathbb{F}N_i^*)\cong \mathbb{F}J_i$.  If we denote the minimal $\add N_{I^c}^*$-approximation of $J_i$ by 
\[
0\to L_i\to W_i\to J_i\to 0,
\]
then $\upnu_I\upnu_I$ takes $N_i$ to $J_i^*$ to $L_i^*$.  We claim that $L_i^*\cong N_i$, so by reflexive equivalence it suffices to prove that $\mathbb{F}L_i\cong \mathbb{F}N_i^*$.  Since 
\[
0\to \mathbb{F}L_i\to \mathbb{F}W_i\to \mathbb{F}J_i\to 0
\]
is exact and $\mathbb{F}W_i$ is a projective $\Delta$-module, $\Omega_\Delta (\mathbb{F}J_i)\cong \mathbb{F}L_i$, so the result follows if we can establish that $\Omega_\Delta^2(\mathbb{F}N_i^*)\cong \mathbb{F}N_i^*$.  Since $\mathbb{F}N_i^*\in\CM\Delta$ has no $\Delta$-projective summands, and $\Omega_\Delta=[-1]$ on the category $\uCM\Delta$, it suffices to show that $[2]=\Id$ on $\uCM\Delta$.

But by \ref{KIWY} $\Delta$ is derived equivalent to $X_{\con}$ and so
\[
\uCM\Delta\simeq\Dsg(\Delta)\simeq\Dsg(X_{\con})\simeq\bigoplus_{x\in\Sing X_{\con}}\uCM\widehat{\cO}_{X_{\con},x}
\]
by Orlov \cite{Orlovcompletion}, since all categories under consideration are idempotent complete.  Since each of $\widehat{\cO}_{X_{\con},x}$ are hypersurfaces, $[2]=\Id$ for each of the categories on the right-hand side, so since the above are triangle equivalences, $[2]=\Id$ for the left-hand category.
\end{proof}

In the study of terminal (and even smooth) $3$-folds, canonical surfaces appear naturally via hyperplane sections, and in this setting $\pd_\Lambda\Lambda_i$ can be infinite, which is very different to \ref{Ext 2 thm} and \ref{Ext 3 thm}.  The next result will allow us to bypass this problem.

\begin{prop}\label{surfaces hack}
Suppose that $R$ is a normal complete local $2$-sCY commutative algebra, and $M\in\CM R$ is basic.  Choose a summand $M_I$, set $\Lambda:=\End_R(M)$ and denote the simple $\Lambda$-modules by $S_j$. Assume that $\upnu_I\upnu_IM\cong M$.  If $x\in\fl\Lambda$ with $\Hom_\Lambda(x,S_i)=0$ for all $i\in I$, then $\Ext^1_\Lambda(T_I,x)=0$.
\end{prop}
\begin{proof}
Since $\Lambda$ is $2$-sCY, $x$ has finite length and $C_I$ has finite projective dimension, 
\[
\Ext^1_\Lambda(T_I,x)\cong\Ext^1_\Lambda(C_I,x)\cong D\Ext^1_\Lambda(x,C_I),
\]
so it suffices to show that $\Ext^1_\Lambda(x,C_I)=0$. By the assumption $\upnu_I\upnu_IM\cong M$,  it follows that $J_I^*\cong K_I$.  
Since $0\to\Hom_R(M,M_I)\to \Hom_R(M,U_I)\to\Hom_R(M,K_I)$ is exact, splicing we obtain exact sequences
\begin{gather}
0\to\Hom_R(M,M_I)\to \Hom_R(M,U_I)\to C_I\to 0\label{CI splice 1}\\
0\to C_I\to \Hom_R(M,K_I)\to F_I\to 0\label{CI splice 2}
\end{gather}
But by \eqref{K1Da}
\[
0\to\Hom_R(M_{I^c},M_I)\to \Hom_R(M_{I^c},U_I)\to\Hom_R(M_{I^c},K_I)\to 0
\]
is exact, so $F_I$ is a finitely generated $\Lambda_I$-module.  But since $d=2$ and $R$ is normal, necessarily $\Lambda_I$ has finite length, hence so too has $F_I$.  Thus the assumptions then imply that $\Hom_\Lambda(x,F_I)=0$, so $\Ext^1_\Lambda(x,C_I)\hookrightarrow\Ext^1_\Lambda(x,\Hom_R(M,K_I))$.

Hence it suffices to show that $\Ext^1_\Lambda(x,\Hom_R(M,K_I))=0$.  But
\[
0\to\Hom_R(M,K_I)\xrightarrow{\cdot c} \Hom_R(M,V_I)\to \Hom_R(M,M_I)\to\Lambda_I\to 0
\]
is exact, so denoting $E_I:=\Cok(\cdot c)$, then $\Hom_\Lambda(x,E_I)$ embeds inside 
\[
\Hom_\Lambda(x,\Hom_R(M,M_I))\cong D\Ext^2_\Lambda(\Hom_R(M,M_I),x)=0,
\]
so $\Hom_\Lambda(x,E_I)=0$.  This in turns implies that $\Ext^1_\Lambda(x,\Hom_R(M,K_I))$ embeds inside $\Ext^1_\Lambda(x,\Hom_R(M,V_I))$, which is zero. Thus $\Ext^1_\Lambda(x,\Hom_R(M,K_I))=0$, as required.
\end{proof}

\section{Contractions and Deformation Theory}\label{contractions and deformations section}

The purpose of this section is use noncommutative deformations to detect whether a divisor has been contracted to a curve, in such a manner that is useful for iterations,  improving \cite{DW2}.  This part of the Homological MMP does not need any restriction on singularities, so throughout this section we adopt the general setup of \ref{general setup}.

\subsection{Background on Noncommutative Deformations}
With the setup $f\colon X\to \Spec R$ of \ref{general setup}, set $\Lambda:=\End_X(\cV_X)$.  Given any $E\in\coh X$, there is an associated classical commutative deformation functor
\[
\cDef_E\colon\cart_1\to\Sets
\]
where $\cart_1$ denotes the category of local commutative artinian $\mathbb{C}$-algebras.  The definition of this functor, which we do not state here, involves a flatness condition over the test object $R\in\cart_1$.  

Noncommutative deformations add two new features to this classical picture.  First, the test objects are enlarged from commutative artinian rings to allow certain (basic) noncommutative artinian $\mathbb{C}$-algebras.  This thickens the universal sheaf.  Second, they allow us to deform a finite collection $\{E_i\mid i\in I\}$ of objects whilst remembering Ext information between them. 

For the purposes of this paper, we will not deform coherent sheaves, but rather their images under the derived equivalence in \S\ref{perverse and tilting}.  Deforming on either side of the derived equivalence turns out to give the same answer \cite{DW2}, but the noncommutative side is slightly easier to formulate. Thus we input a finite collection $\{S_i\mid i\in I\}$ of simple $\Lambda$-modules, and define the associated noncommutative deformation functor as follows.

As preparation, recall that an \emph{$n$-pointed $\mathbb{C}$-algebra} $\Gamma$ is an associative $\mathbb{C}$-algebra, together with $\mathbb{C}$-algebra morphisms $p\colon \Gamma \to \mathbb{C}^n$ and $i\colon \mathbb{C}^n \to \Gamma$ such that $i p = \Id$.  A morphism of $n$-pointed $\mathbb{C}$-algebras $\uppsi\colon (\Gamma,p,i)\to(\Gamma',p',i')$ is a ring homomorphism $\uppsi\colon\Gamma\to\Gamma'$ such that
\[
\begin{tikzpicture}
\node (A) at (0,0) {$\mathbb{C}^n$};
\node (B1) at (1.5,0.75) {$\Gamma$};
\node (B2) at (1.5,-0.75) {$\Gamma'$};
\node (C) at (3,0) {$\mathbb{C}^n$};
\draw[->] (A) -- node[above] {$\scriptstyle i$} (B1);
\draw[->] (A) -- node[below] {$\scriptstyle i'$} (B2);
\draw[->] (B1) -- node[above] {$\scriptstyle p$} (C);
\draw[->] (B2) -- node[below] {$\scriptstyle p'$} (C);
\draw[->] (B1) -- node[right] {$\scriptstyle \uppsi$}(B2);
\end{tikzpicture}
\]
commutes.  We denote the category of $n$-pointed $\mathbb{C}$-algebras by $\alg_n$, and denote the full subcategory consisting of those objects that are commutative rings by $\calg_n$.  Furthermore, denote by $\art_n$ the full subcategory of $\alg_n$ consisting of objects $(\Gamma,p,i)$ for which $\dim_{\mathbb{C}}\Gamma<\infty$ and the augmentation ideal $\n:=\Ker(p)$ is nilpotent.  The full subcategory of $\art_n$ consisting of those objects that are commutative rings is denoted $\cart_n$.

Given $\Gamma\in\art_n$, the morphism $i$ produces $n$ idempotents $e_1,\hdots,e_n\in\Gamma$, and we denote $\Gamma_{ij}:=e_i\Gamma e_j$.
\begin{defin}{\cite{Laudal}}
Fix a finite collection $\cS:=\{S_i\mid i\in I\}$ of left $\Lambda$-modules.  Then 
\begin{enumerate}
\item For $\Gamma\in\art_{|I|}$, we say that $M\in\Mod \Lambda\otimes_{\mathbb{C}}\Gamma^{\op}$ (i.e.\ a $\Lambda$-$\Gamma$ bimodule) is \emph{$\Gamma$-matric-free} if 
\[
M \cong (S_i\otimes_{\mathbb{C}} \Gamma_{ij})
\]
as $\Gamma^{\op}$-modules, where the right-hand side is the matrix built by varying $i,j\in\{1,\hdots,n\}$, which has an obvious $\Gamma^{\op}$-module structure.
\item The noncommutative deformation functor
\[
\Def_{\cS}\colon\art_n\to\Sets
\]
is defined by sending
\[
(\Gamma,\n)\mapsto \left. \left \{ (M,\delta)
\left|\begin{array}{l}M\in\Mod\Lambda\otimes_{\mathbb{C}}\Gamma^{\op}\\ M \t{ is $\Gamma$-matric-free}\\ \delta=(\delta_i)\mbox{ with }\delta_i\colon M\otimes_\Gamma (\Gamma/\n)e_i\xrightarrow{\sim} S_i\end{array}\right. \right\} \middle/ \sim \right.
\]
where $(M,\delta)\sim (M',\delta')$ if there exists an isomorphism $\tau\colon M\to M'$ of bimodules such that 
\[
\begin{tikzpicture}
\node (a1) at (0,0) {$M\otimes_\Gamma (\Gamma/\n)e_i$};
\node (a2) at (3,0) {$M'\otimes_\Gamma (\Gamma/\n)e_i$};
\node (b) at (1.5,-1) {$S_i$};
\draw[->] (a1) -- node[above] {$\scriptstyle \tau \otimes 1 $} (a2);
\draw[->] (a1) -- node[gap] {$\scriptstyle \delta_i$} (b);
\draw[->] (a2) -- node[gap] {$\scriptstyle \delta'_i$} (b);
\end{tikzpicture}
\]
commutes, for all $i$.
\item The commutative deformation functor is defined to be the restriction of $\Def_{\cS}$ to $\cart_{|I|}$, and is denoted $\cDef_{\cS}$.
\end{enumerate}
\end{defin}

Given the general setup $f\colon X\to \Spec R$ of \ref{general setup}, choose a subset of curves $\bigcup_{i\in I}C_i$ above the unique closed point and contract them as in \S\ref{general prelim} to factorise $f$ as 
\[
X\xrightarrow{g}X_{\con}\xrightarrow{h}\Spec R
\]
with $\RDerived g_*\cO_X=\cO_{X_{\con}}$ and $\RDerived h_*\cO_{X_{\con}}=\cO_{R}$.  By \ref{simple across Db}, across the derived equivalence the coherent sheaves $\cO_{C_i}(-1)\in\coh X$ ($i\in I$) correspond to simple left $\Lambda$-modules $S_i$. The following is the $d=3$ special case of the main result of \cite{DW2}.

\begin{thm}[Contraction Theorem]\label{contraction theorem}
With the general setup in \ref{general setup}, if $d=3$ then $f$ contracts $\bigcup_{i\in I}C_i$ to a point without contracting a divisor if and only if $\Def_{\cS}$ is representable.
\end{thm}

\subsection{Global and Local Contraction Algebras}\label{section 3.2}
We maintain the notation from the general setup of the previous subsection. In this subsection we detect whether $g$ contracts a curve without contracting a divisor by using the algebra $\Lambda=\End_X(\cV_X)$, constructed in \S\ref{perverse and tilting}   using the morphism $f$.  This will allow us to iterate.

\begin{defin}\label{LambdaI global}
For $\Lambda=\End_X(\cV_X)$, with notation from \ref{N notation}  define $[\cN_{I^c}]$ to be the 2-sided ideal of $\Lambda$ consisting of morphisms that factor through $\add\cN_{I^c}$, and set $\Lambda_I:=\Lambda/[\cN_{I^c}]$.
\end{defin}

In \cite{DW2} the prorepresenting object of $\Def_{\cS}$ was constructed locally with respect to the morphism $g$, using the following method.  Let $x\in X_{\con}$ be the closed point above which sits $\bigcup_{i\in I}C_i$.  Choose an affine neighbourhood $U_{\con}:=\Spec R'$ containing $x$,  set $U:=g^{-1}(U_{\con})$, let $\mathfrak{R}'$ be the completion of $R'$ at $x$, and consider the formal fibre $\mathfrak{U}\to \Spec \mathfrak{R}'$.  This morphism satisfies all the assumptions of the general setup of \ref{general setup}, and we define the \emph{contraction algebra} to be $\CA^I:=\End_{\mathfrak{U}}(\cV_{\mathfrak{U}})/[\cO_{\mathfrak{U}}]$.  By \cite{DW1, DW2} the contraction algebra is independent  of choice of $U$, and 
\begin{eqnarray}
\Def_{\cS}(-)\cong \Hom_{\alg_{|I|}}(\CA^{I},-)\label{iso of def to alg}.
\end{eqnarray}

Even in the case $|I|=1$, comparing $\Lambda_I$ and $\CA^I$  directly is a subtle problem.  If $|I|=1$ and the scheme-theoretic multiplicity of $C_i$ with respect to $g$ is $n$, then we can view $\CA^{\{i\}}$ as factor of an endomorphism ring of an indecomposable rank $n$ bundle on $\mathfrak{U}$.  On the other hand, $\Lambda_i$ can be viewed as a factor of an endomorphism ring of an indecomposable rank $m$ bundle on $X$, where $m$ is the scheme theoretic multiplicity of $C_i$ with respect to the morphism $f$.  The next example demonstrates that $m\neq n$ in general.

\begin{example}\label{m not n}
Consider the $cD_4$ singularity $R:=\mathbb{C}[[u,x,y,z]]/(u^2-xyz)$, which is isomorphic to the toric quotient singularity $\mathbb{C}^3/G$ where $G=\mathbb{Z}_2\times\mathbb{Z}_2\leq \SL(3,\mathbb{C})$ .   We consider $X^+=G\mathrm{-Hilb}(\mathbb{C}^3)$, and one of its flops
\[
\begin{tikzpicture}[>=stealth]
\node (A1) at (3,0) {$\begin{tikzpicture}
\node[name=s,regular polygon,regular polygon sides=3, minimum size=1cm] at (0,0) {};
\coordinate (C1) at (s.corner 1)  {};
\coordinate (C2) at (s.corner 2)  {};
\coordinate (C3) at (s.corner 3)  {};
\coordinate (E1) at (s.side 1) {};
\coordinate (E2) at (s.side 2) {};
\coordinate (E3) at (s.side 3) {};
\draw[red] (E2) -- (E3); 
\node at (C1) [vertex] {};
\node at (C2) [vertex] {};
\node at (C3) [vertex] {};
\node at (E1) [vertex] {};
\node at (E2) [vertex] {};
\node at (E3) [vertex] {};
\draw (C1) -- (C2);
\draw (C2) -- (C3);
\draw (C3) -- (C1);
\draw (E1) -- (E2);
\draw (E1) -- (E3);
\filldraw[circle=2pt] (C1);
\end{tikzpicture}$};
\node (A2) at (0,0) {$\begin{tikzpicture}
\node[name=s,regular polygon,regular polygon sides=3, minimum size=1cm] at (0,0) {}; 
\coordinate (C1) at (s.corner 1)  {};
\coordinate (C2) at (s.corner 2)  {};
\coordinate (C3) at (s.corner 3)  {};
\coordinate (E1) at (s.side 1) {};
\coordinate (E2) at (s.side 2) {};
\coordinate (E3) at (s.side 3) {};
\draw[red] (C3) -- (E1);
\node at (C1) [vertex] {};
\node at (C2) [vertex] {};
\node at (C3) [vertex] {};
\node at (E1) [vertex] {};
\node at (E2) [vertex] {};
\node at (E3) [vertex] {};
\draw (C1) -- (C2);
\draw (C2) -- (C3);
\draw (C3) -- (C1);
\draw (E1) -- (E2);
\draw (E1) -- (E3);
\filldraw[circle=2pt] (C1);
\end{tikzpicture}$};
\node (B) at (1.5,-1) {$\begin{tikzpicture}
\node[name=s,regular polygon,regular polygon sides=3, minimum size=1cm] at (0,0) {}; 
\coordinate (C1) at (s.corner 1)  {};
\coordinate (C2) at (s.corner 2)  {};
\coordinate (C3) at (s.corner 3)  {};
\coordinate (E1) at (s.side 1) {};
\coordinate (E2) at (s.side 2) {};
\coordinate (E3) at (s.side 3) {};
\node at (C1) [vertex] {};
\node at (C2) [vertex] {};
\node at (C3) [vertex] {};
\node at (E1) [vertex] {};
\node at (E2) [vertex] {};
\node at (E3) [vertex] {};
\draw (C1) -- (C2);
\draw (C2) -- (C3);
\draw (C3) -- (C1);
\draw (E1) -- (E2);
\draw (E1) -- (E3);
\filldraw[circle=2pt] (C1);
\end{tikzpicture}$};
\node (C) at (1.5,-2.5) {$\begin{tikzpicture}
\node[name=s,regular polygon,regular polygon sides=3, minimum size=1cm] at (0,0) {}; 
\coordinate (C1) at (s.corner 1)  {};
\coordinate (C2) at (s.corner 2)  {};
\coordinate (C3) at (s.corner 3)  {};
\coordinate (E1) at (s.side 1) {};
\coordinate (E2) at (s.side 2) {};
\coordinate (E3) at (s.side 3) {};
\node at (C1) [vertex] {};
\node at (C2) [vertex] {};
\node at (C3) [vertex] {};
\node at (E1) [vertex] {};
\node at (E2) [vertex] {};
\node at (E3) [vertex] {};
\draw (C1) -- (C2);
\draw (C2) -- (C3);
\draw (C3) -- (C1);
\filldraw[circle=2pt] (C1);
\end{tikzpicture}$};
\draw[->] (A1)--(B);
\draw[->] (A2)--(B);
\draw[->] (A1)--(C);
\draw[->] (A2)--(C);
\draw[->] (B)--(C);
\node (X1) at (5,0) {$X$};
\node (X2) at (8,0) {$X^+$};
\node (Xcon) at (6.5,-1) {$X_{\con}$};
\node (R) at (6.5,-2.5) {$\Spec R$};
\draw[->] (X1) -- node[above right] {$\scriptstyle g$} (Xcon);
\draw[->] (X2) -- node[above left, pos=0.4] {$\scriptstyle g'$} (Xcon);
\draw[->] (Xcon) -- node[right] {$\scriptstyle $} (R);
\draw[->] (X1) -- node[below left] {$\scriptstyle f$} (R);
\draw[->] (X2) -- node[below right] {$\scriptstyle f'$} (R);
\end{tikzpicture}
\]
Locally, being the Atiyah flop, the scheme theoretic multiplicity of the flopping curve with respect to $g$ is one, but with respect to $f$ the scheme theoretic multiplicity is two.  This can be calculated directly, but it also follows from the example \ref{Z2Z2example} later, once we have proved \ref{MMmodule under flop cor}. 
\end{example}

The following is the main result of this section.

\begin{thm}\label{contract on f}
With the general setup in \ref{general setup}, suppose that $d=3$ and  pick a subset $\bigcup_{i\in I}C_i$ of curves above the origin. Then 
\begin{enumerate}
\item\label{contract on f 1} $\Def_{\cS}\cong \Hom_{\alg_{|I|}}(\Lambda_I,-)$.
\item\label{contract on f 2} $\Lambda_I\cong\CA^I$.
\item\label{contract on f 3} $\bigcup_{i\in I}C_i$ contracts to point without contracting a divisor $\Leftrightarrow\dim_{\mathbb{C}}\Lambda_I<\infty$. 
\end{enumerate}
\end{thm}
\begin{proof}
(1)  Arguing exactly as in \cite[3.1]{DW1}, since $S_i=\mathbb{C}$ as $\Lambda$-modules, if we denote the natural homomorphisms $\Lambda\to S_i$ by $q_i$, then for $(\Gamma,\n)\in\art_{|I|}$,
\begin{align*}
\Def_{\cS}(\Gamma)
&\cong\left. \left \{ \begin{array}{cl} \bullet& \mbox{A left $\Lambda$-module structure on $(S_i\otimes_{\mathbb{C}}\Gamma_{ij})$ such that}\\
&\mbox{$(S_i\otimes_{\mathbb{C}}\Gamma_{ij})$ becomes a $\Lambda$-$\Gamma$ bimodule.}\\
\bullet & \delta=(\delta_i)\mbox{ such that } \delta_i\colon  (S_i\otimes_{\mathbb{C}}\Gamma_{ij})\otimes_\Gamma (\Gamma/\n)e_i \xrightarrow{\sim}S_i\\ &\mbox{ as $\Lambda$-modules}
\end{array} \right\} \middle/ \sim \right. \\
&\cong\left. \left \{ \begin{array}{l}  \mbox{A $\mathbb{C}$-algebra homomorphism $\Lambda\to\Gamma$ such that the}\\
\mbox{composition $\Lambda\to\Gamma\to (\Gamma/\n)e_i=S_i$ is $q_i$ for all $i\in I$}\end{array} \right\} \middle/  \sim  \right.\\
&\cong \Hom_{\alg_{|I|}}(\Lambda_I,\Gamma).
\end{align*}
(2) Since $R$ is complete, both $\CA^I$ and $\Lambda_I$ belong to the pro-category of $\art_{|I|}$, so $\Def_{\cS}$ is prorepresented by both $\Lambda_I$ and $\CA^I$.  Hence by uniqueness of prorepresenting object,  $\CA^I\cong \Lambda_I$.\\
(3) This follows by combining \eqref{contract on f 1} and \ref{contraction theorem}.
\end{proof}

\section{Mutation, Flops and Twists}\label{flops and twists section}

\subsection{Flops and Mutation}  We now consider the crepant setup of \ref{crepant setup} with $d=3$, namely $f\colon X\to \Spec R$ is a crepant projective birational morphism, with one dimensional fibres, where $R$ is a complete local Gorenstein algebra, and  $X$ has at worst Gorenstein terminal singularities.  As in \S\ref{perverse and tilting}, we consider $\cV_X$, the basic progenerator of $\Per(X,R)$, set $N:=H^0(\cV_X)$ and by \ref{flop up=down} denote $\Lambda:=\End_X(\cV_X)\cong\End_R(N)$.

\begin{remark}
It also follows from \ref{flop up=down} that $\End_X(\cV_X)/[\cN_{I^c}]\cong\End_R(N)/[N_{I^c}]$ and so the $\Lambda_I$ defined in \ref{LambdaI global} and \ref{setup} coincide.  This allows us to link the previous contraction section to mutation.  
\end{remark}

We aim to prove the following theorem.
\begin{thm}\label{flop=mut general thm}
With the crepant setup as above, with $d=3$, pick a subset of curves $\bigcup_{i\in I}C_i$ above the origin, and suppose that $\dim_{\mathbb{C}}\Lambda_I<\infty$ (equivalently, by \ref{contract on f}, the curves flop).  Denote the flop by $X^+$, then
\begin{enumerate}
\item\label{flop=mut general thm 1} $\upnu_IN\cong H^0(\cV_{X^+})$, where $\cV_{X^+}$ is the basic progenerator of $\Per(X^+,R)$.
\item\label{flop=mut general thm 2} The following diagram of equivalences is naturally commutative
\[
\begin{array}{c}
\begin{tikzpicture}[scale=1.2]
\node (a1) at (0,0) {$\D(\coh X)$};
\node (a2) at (2.5,0) {$\D(\coh X^+)$};
\node (b1) at (0,-1) {$\D(\mod\Lambda)$};
\node (b2) at (2.5,-1) {$\D(\mod\upnu_I\Lambda)$};
\draw[->] (a1) to node[above] {$\scriptstyle \flop$} (a2);
\draw[->] (b1) to node[above] {$\scriptstyle \Upphi_I$} (b2);
\draw[->] (a1) to node[left] {$\scriptstyle \Uppsi_X$} (b1);
\draw[->] (a2) to node[right] {$\scriptstyle \Uppsi_{X^+}$} (b2);
\end{tikzpicture}
\end{array}
\]
where $\flop$ is the inverse of the Bridgeland--Chen flop functor, $\Uppsi_X$ and $\Uppsi_{X^+}$ are the tilting equivalences in \eqref{VdBfunctor}, and $\Upphi_I$ is the mutation functor in \eqref{post referee}.
\end{enumerate}
\end{thm}
The proof of \ref{flop=mut general thm} will be split into two stages.  Stage one, proved in this subsection, is to establish \ref{flop=mut general thm}\eqref{flop=mut general thm 1} in the case where $X$ is a minimal model and $I=\{i \}$.  Stage two is then to prove \ref{flop=mut general thm} in full generality, lifting the $\mathds{Q}$-factorial and $|I|=1$ assumption.  The second stage uses the Auslander--McKay correspondence in \S\ref{AM subsection 1}, together with the Bongartz completion to pass to the minimal model, before we then contract back down.  Thus, the full proof of \ref{flop=mut general thm} will not finally appear until \S\ref{flops and mu section 2}.

To establish functoriality, the following results will be useful later.

\begin{prop}\label{Toda argument}
Suppose that $f\colon X\to\Spec R$ and $f'\colon X'\to\Spec R$ both satisfy the crepant setup \ref{crepant setup}, and admit a common contraction
\[
\begin{tikzpicture}
\node (X) at (-1.33,0) {$X$};
\node (Xprime) at (1.33,0) {$X'$};
\node (Xcon) at (0,-0.8) {$X_{\con}$};
\node (R) at (0,-2) {$\Spec R$};
\draw[->] (X) -- node[above right] {$\scriptstyle g$} (Xcon);
\draw[->] (Xprime) -- node[above left, pos=0.4] {$\scriptstyle g'$} (Xcon);
\draw[->] (Xcon) -- node[right] {$\scriptstyle $} (R);
\draw[->] (X) -- node[below left] {$\scriptstyle f$} (R);
\draw[->] (Xprime) -- node[below right] {$\scriptstyle f'$} (R);
\end{tikzpicture}
\]
Suppose that $g$ and $g'$ contract the same number of curves, and denote the contracted curves by $\{C_i\mid i\in I\}$ and  $\{C'_i\mid i\in I\}$ respectively.
If $\Uptheta\colon\Db(\coh X)\to\Db(\coh X')$ is a Fourier--Mukai equivalence that satisfies
\begin{enumerate}
\item $\RDerived g'_*\circ \Uptheta\cong \RDerived g_*$
\item $\Uptheta(\cO_X)\cong \cO_{X'}$
\item $\Uptheta(\cO_{C_i}(-1))\cong \cO_{C'_i}(-1)$ for all $i\in I$, 
\end{enumerate}
then $\Uptheta\cong\upphi_*$ where $\upphi\colon X\to X'$ is an isomorphism such that $g'\circ\upphi=g$,
\end{prop}
\begin{proof}
This is identical to \cite[\S7.6]{DW1}, which itself is based on \cite{TodaResPub}.  As in \cite[7.17]{DW1}, from properties (1) and (3) it follows that $\Uptheta$ takes $\Per(X,X_{\con})$ to $\Per(X',X_{\con})$.  The argument is then word-for-word identical to the proof of \cite[7.17, 7.18]{DW1}, since although there it was assumed that $X$ was projective, this is not needed in the proof.
\end{proof}

\begin{cor}\label{Per0 global general}
Suppose that $f\colon X\to\Spec R$ and $f'\colon X'\to\Spec R$ both satisfy the crepant setup \ref{crepant setup}.  If $H^0(\cV_X)\cong H^0(\cV_{X'})$, then there is an isomorphism $X\cong X'$ compatible with $f$ and $f'$.  
\end{cor}
\begin{proof}
Temporarily denote $N:=H^0(\cV_X)$ and $M:=H^0(\cV_{X'})$.  Since $N\cong M$ by assumption, certainly they have the same number of indecomposable summands, so since $\cV_X$ and $\cV_{X'}$ are basic, it follows that the numbers of curves contracted by $f$ and $f'$ are the same.  Further, we have a diagram of equivalences
\[
\begin{array}{c}
\begin{tikzpicture}[scale=1.2]
\node (a1) at (0,0) {$\D(\coh X)$};
\node (a2) at (3.5,0) {$\D(\coh X')$};
\node (b1) at (0,-1) {$\D(\mod\End_X(\cV_X))$};
\node (b2) at (3.5,-1) {$\D(\mod\End_{X'}(\cV_{X'}))$};
\draw[->] (b1) to node[above] {$\scriptstyle \cong$} node [below] {$\scriptstyle F$} (b2);
\draw[->] (a1) to node[left] {$\scriptstyle \Uppsi_X$} (b1);
\draw[<-] (a2) to node[right] {$\scriptstyle \Uppsi^{-1}_{X'}$} (b2);
\end{tikzpicture}
\end{array}
\]
since $\End_X(\cV_X)\cong\End_R(N)\cong \End_R(M)\cong\End_{X'}(\cV_{X'})$ by \ref{flop up=down}.  Denote the composition of equivalences by $\Uptheta$, then the composition is a Fourier--Mukai functor \cite[6.16]{DW1}, which is thus also an equivalence. Now
\[
\RDerived f'_*\circ\Uptheta\cong\RDerived f'_*\circ\Uppsi_{X'}^{-1}\circ F\circ\Uppsi_X\cong e(-)\circ \Uppsi_X\cong \RDerived f_*
\]
where the second and third isomorphisms are \ref{KIWY}. Further by \ref{simple across Db} $\Uptheta$ takes
\[
\cO_{C_i}(-1)\mapsto S_i\mapsto S_i'\mapsto\cO_{C'_i}(-1)
\]
for all exceptional curves $C_i$, where $S_i$ are simple $\End_R(N)$-modules and $S'_i$ are simple $\End_R(M)$-modules.  Lastly, $\Uptheta$ sends
\[
\cO_X\mapsto P_0\mapsto P_0'\mapsto \cO_{X'}
\]  
where $P_0\cong\Hom_R(N,R)$ and $P_0'=\Hom_R(M,R)$.  Hence applying \ref{Toda argument} with $X_{\con}=\Spec R$ gives the result.
\end{proof}

The following lemma, an easy consequence of Riedtmann--Schofield, proves \ref{flop=mut general thm}\eqref{flop=mut general thm 1} with restricted hypotheses.

\begin{lemma}\label{mut=flop single curve min model}
With the crepant setup of \ref{crepant setup}, suppose further that $d=3$ and $X$ is $\mathds{Q}$-factorial, that is $f\colon X\to \Spec R$ is a minimal model.  Choose a single curve $C_i$ above the origin, suppose that $\dim_{\mathbb{C}}\Lambda_i<\infty$ (equivalently, by \ref{contract on f}, $C_i$ flops), and let $X^+$ denote the flop of $C_i$.  Then $\upnu_iN\cong H^0(\cV_{X^+})$.
\end{lemma}
\begin{proof}
Denote the base of the contraction of $C_i$ by $X_{\con}$, set $M:=H^0(\cV_{X^+})$ and let $M_i$ denote the indecomposable summand of $M$ corresponding to $C_i^+$.  It is clear that $M\ncong N$. Applying \ref{KIWY} to both sides of the contraction, $\frac{M}{M_i}\cong H^0(\cV_{X_{\con}})\cong\frac{N}{N_i}$, so the module $M$ differs from $N$ only at the summand $N_i$.  Similarly, by \ref{basic2}\eqref{basic2 3} $\upnu_iN\ncong N$, and by definition of mutation, $\upnu_iN$ differs from $N$ only at the summand $N_i$.  Consequently, as $R$-modules, $\upnu_iN$ and $M$ share all summands except one, and neither is isomorphic to $N$.

But by \ref{Min model - MMA} both $\End_R(M)$ and $\End_R(N)$ are MMAs, and since $\End_R(\upnu_iN)$ is also derived equivalent to these, it too is an MMA \cite[4.16]{IW4}.  Further, $\Hom_R(N,\upnu_iN)$ and $\Hom_R(N,M)$ are tilting $\End_R(N)$-modules by \cite[4.17(1)]{IW4}, and by above as $\End_R(N)$-modules they share all summands except one. Hence as in \cite[6.22]{IW4}, a Riedtmann--Schofield type theorem implies that $\upnu_iN\cong M$.
\end{proof}

\begin{cor}\label{MMmodule under flop cor}
With the crepant setup of \ref{crepant setup}, suppose further that $d=3$ and $X$ is $\mathds{Q}$-factorial.  Then
\[
\upnu_iN\cong \left\{\begin{array}{ll} H^0(\cV_{X^+})&\mbox{if $C_i$ flops}\\H^0(\cV_{X})&\mbox{else.}  \end{array}\right.
\]
\end{cor}
\begin{proof}
This now follows by combining \ref{basic2} and \ref{mut=flop single curve min model}.
\end{proof}

The above allows us to verify Figure~\ref{Fig2} under restricted hypotheses.

\begin{cor}
We can run the Homological MMP in Figure~\ref{Fig2} when $d=3$ and $X$ has only $\mathds{Q}$-factorial Gorenstein terminal singularities, and we choose only irreducible curves.
\end{cor}
\begin{proof}
This now follows from \ref{contract on f}, \ref{mut=flop single curve min model} and \ref{reconstruction new}.
\end{proof}

Later in \S\ref{flops and mu section 2} we will drop the $\mathds{Q}$-factorial assumption, and also drop the restriction to single curves.

\subsection{Auslander--McKay Correspondence}\label{AM subsection 1}

Throughout this subsection we keep the crepant setup of \ref{crepant setup}, and as in the previous subsection assume further that $d=3$ and $X$ is $\mathds{Q}$-factorial.  The $R$ admitting such a setup are of course well-known to be precisely the cDV singularities \cite{Pagoda}.

\begin{defin}\label{graphs definition}
With $R$ as above, 
\begin{enumerate}
\item We define the \emph{full mutation graph} of the MM generators to have as vertices the basic MM generators (up to isomorphism of $R$-modules), where each vertex $N$ has an edge to $\upnu_IN$ provided that $\dim_{\mathbb{C}}\Lambda_I<\infty$, for $I$ running through all possible summands $N_I$ of $N$ that are not generators. 
The \emph{simple mutation graph} is defined in a similar way, but we only allow mutation at indecomposable summands.
\item We define the \emph{full flops graph} of the minimal models of $\Spec R$ to have as vertices the minimal models of $\Spec R$ (up to isomorphism of $R$-schemes), and we connect two vertices if  the corresponding minimal models are connected by a flop at some curve. The \emph{simple flops graph} is defined in a similar way, but we only connect two vertices if the corresponding minimal models differ by a flop at an irreducible curve.
\end{enumerate}
\end{defin}

The following is standard.

\begin{defin}\label{stable end defin}
For $N\in\mod R$, the \emph{stable endomorphism ring} $\uEnd_R(N)$ is defined to be the quotient of $\End_R(N)$ by the two sided ideal consisting of those morphisms $N\to N$ which factor through $\add R$.
\end{defin}

Recall from the introduction \S\ref{GIT intro} that there is a specified region $C_+$ of the GIT chamber decomposition of $\Uptheta(\End_R(N))$, and $\cM_{\rk,\upphi}(\Lambda)$ denotes the moduli space of $\upphi$-semistable $\Lambda$-modules of dimension vector $\rk$ (see \S\ref{GIT background} for more details). 

\begin{thm}\label{NCvsC min model}
With the $d=3$ crepant setup of \ref{crepant setup}, assume further that $X$ is $\mathds{Q}$-factorial. Then there exists a one-to-one correspondence
\[
\begin{array}{c}
\begin{tikzpicture}
\node (A) at (-1,0) {$\{ \mbox{basic MM $R$-module generators}\}$};
\node (B) at (6,0) {$
\{\mbox{minimal models $f_i\colon X_i\to\Spec R$} \}
$};
\draw[<->] (1.85,0) -- node [above] {$\scriptstyle $} (3,0);
\draw[|->] (1.85,-0.6) -- node [above] {$\scriptstyle F$} (3,-0.6);
\draw [<-|] (1.85,-1.1) -- node [below] {$\scriptstyle G$} (3,-1.1);
\node at (1.25,-0.6) {$N$};
\node at (4.6,-0.6) {$\cM_{\rk,\upvartheta}(\End_R(N))$} ;
\node at (0.85,-1.1) {$H^0(\cV_{X_i})$};
\node at (3.5,-1.1) {$X_i$} ;
\end{tikzpicture}
\end{array}
\]
where $\cV_{X_i}$ is the basic progenerator of $\Per(X_i,R)$, and $\upvartheta$ is any element of $C_+$.  
Elements in the set on the left-hand side are taken up to isomorphism of $R$-modules, and elements of the set on the right-hand side are taken up to isomorphism of $R$-schemes.  

Under this correspondence
\begin{enumerate}
\item\label{NCvsC min model 1} For any fixed MM generator, its non-free indecomposable summands are in one-to-one correspondence with the exceptional curves in the corresponding minimal model.
\item\label{NCvsC min model 2} For any fixed MM generator $N$, the quiver of $\uEnd_R(N)$ encodes the dual graph of the corresponding minimal model.
\item\label{NCvsC min model 3} The simple mutation graph of the MM generators coincides with the simple flops graph of the minimal models.
\end{enumerate}
In particular the number of basic MM generators is finite. 
\end{thm}
\begin{proof}
Pick a minimal model $X\to\Spec R$ and denote $N:=H^0(\cV_X)$, which we know to be an MM generator by \ref{Min model - MMA}.  We now mutate $N$ at all possible non-free indecomposable summands.  By \ref{MMmodule under flop cor}, the only new MM generators that this produces are the global sections from the progenerators of perverse sheaves of the possible flops.  We continue mutating these at the non-free indecomposable summands, then either we go back to the original $N$, or the only new MM generators are those arising from flops of flops.  Continuing in this way, since there are only a finite number of minimal models \cite[Main Theorem]{KMa}, which are connected by a finite sequence of simple flops (see e.g.\ \cite{KollarFlops}), by repeatedly mutating at non-free summands we recover only a finite number of MM generators.  By \cite[4.3]{IW6} this implies that they are all the MM generators, in particular there is only a finite number and each is isomorphic to $H^0(\cV_Y)$ for some minimal model $Y$. This shows that the function $G$ is surjective.  The fact that $G$ is injective is just \ref{Per0 global general}, and so $G$ is bijective.  The fact that its inverse is given by $F$ is precisely \cite[5.2.5]{Joe}, with the small caveat that \cite[5.2.5]{Joe} works with
the opposite algebra, but only since his conventions for composing morphisms are opposite to ours. \\
(1) Let $N$ be an MM generator, then since by the above $N\cong H^0(\cV_Y)$ for some minimal model $Y\to\Spec R$, the statement follows from the construction of the bundle $\cV_Y$ in \S\ref{perverse and tilting}.\\
(2) Since the quiver of $\uEnd_R(N)$ is the quiver of $\End_R(N)\cong\End_Y(\cV_Y)$ with the vertex $\star$ (corresponding to the summand $R$ of $N$) removed, it follows from \ref{reconstruction new} that the quiver of $\uEnd_R(N)$ is the double of the dual graph, together with some loops.\\
(3) This follows from \ref{MMmodule under flop cor} and the above argument.
\end{proof}

We extend the correspondence later in \S\ref{AM section 2}--\S\ref{AM isolated subsection}.

\begin{remark}
$\Spec R$ may be its own minimal model, in which case the correspondence in \ref{NCvsC min model} reduces to the statement that  $R$ is the only basic modifying generator.
\end{remark}

\begin{remark}
The proof of \ref{NCvsC min model} uses the fact that there are only a finite number of minimal models, and that they are connected by a finite sequence of simple flops.  We use these results only to simplify the exposition; it is possible to instead use the moduli tracking results of \S\ref{stab and mut section}, specifically \ref{chamber 3fold=chamber surface}, to give a purely homological proof of \ref{NCvsC min model}.
\end{remark}
 
 \begin{remark}
The bijection in \ref{NCvsC min model} extends to a bijection
\[
\begin{array}{c}
\begin{tikzpicture}
\node (A) at (-1,0) {$\{ \mbox{basic modifying objects in $\uCM R$}\}$};
\node (B) at (6.1,0) {$
\{\mbox{crepant modifications $f_i\colon X_i\to\Spec R$} \},
$};
\draw[->] (1.95,0.1) -- node [above] {$\scriptstyle F$} (2.9,0.1);
\draw [<-] (1.95,-0.1) -- node [below] {$\scriptstyle G$} (2.9,-0.1);
\end{tikzpicture}
\end{array}
\]
satisfying the obvious extensions of (1) and (2), where by crepant modification we mean that $X_i$ is obtained from a minimal model by contracting curves to points, and divisors to curves.  However, because of \ref{not true multiple}, comparing mutation and flop at arbitrary summands is not so well behaved when the $X_i$ are not minimal models.
\end{remark}
 
Sometimes the minimal models of $\Spec R$ can be smooth.  Recall from \ref{CT defin} the definition of a CT module.
\begin{cor}
With the setup as in \ref{NCvsC min model}, assume further that the minimal models of $\Spec R$ are smooth (equivalently, $R$ admits a CT module).  Then \ref{NCvsC min model} reduces to a one-to-one correspondence 
\[
\begin{array}{c}
\begin{tikzpicture}
\node (A) at (0,0) {$\{ \mbox{basic CT $R$-modules}\}$};
\node (B) at (6.25,0) {$
\{\mbox{crepant resolutions $f_i\colon X_i\to\Spec R$} \}
$};
\draw[<->] (2,0) -- node [above] {$\scriptstyle $} (3,0);
\end{tikzpicture}
\end{array}
\]
satisfying the same conditions \t{(1)--(3)}.
\end{cor}
\begin{proof} 
If one of (equivalently, all of) the minimal models is smooth, then $\End_X(\cV_X)\cong\End_R(N)$ has finite global dimension and hence is an NCCR.  Thus $N$ is a CT module. Since $\CM R$ has a CT module, by \cite[5.11(2)]{IW4} CT modules are precisely the MM generators.
\end{proof}

\subsection{Flops and Mutation Revisited}\label{flops and mu section 2} In this subsection we use the Auslander--McKay Correspondence to finally prove \ref{flop=mut general thm} in full generality, then run the Homological MMP in Figure~\ref{Fig2} when $X$ has only Gorenstein terminal singularities. 

We first track certain objects under the mutation functor $\Upphi_I$ in \eqref{post referee}.  As notation, suppose that $M$ is a basic modifying $R$-module, where $R$ is complete local $d$-sCY.  Then for each indecomposable summand $M_j$ of $M$, denote the corresponding simple and projective $\End_R(M)$-modules by $S_j$ and $P_j$ respectively.  For an indecomposable summand $X_j$ of $\upnu_IM$, we order the indecomposable summands of $\upnu_IM$ so that either $X_j\cong M_j$ when $j\notin I$, or $X_i\cong J_i^*$ when $i\in I$. We denote the corresponding simple and projective $\End_R(\upnu_IM)$-modules by $S'_i$ and $P'_i$ respectively.

\begin{lemma}\label{track mutation}
Suppose that $R$ is a complete local $d$-sCY normal domain, and $M\in\refl R$ is a basic modifying module.  With notation as above, choose a summand $M_I$ of $M$, and consider the mutation functor $\Upphi_I$ in \eqref{post referee}. Then
\begin{enumerate}
\item\label{track mutation 1} $\Upphi_I(P_j)=P'_j$ for all $j\notin I$.
\item\label{track mutation 2} $\Upphi_I(S_i)=S'_i[-1]$ for all $i\in I$.
\end{enumerate}  
\end{lemma}
\begin{proof}
By definition, $T_I:=(\oplus_{j\notin I} P_j)\oplus C_I$, so part (1) is obvious.  For (2),  fix $i\in I$ and consider $S_i$.  Set $P_I:=\bigoplus_{k\in I} P_k$, and note that $C_I=\bigoplus_{k\in I}C_k$, where the sequence \eqref{defin of CI} is a direct sum of exact sequences
\[
0\to P_k\to Q_k\to C_k\to 0
\]
with $P_I\notin \add Q_k$ for all $k\in I$.  Applying $\Hom_\Lambda(-,S_i)$ gives, for every $k\in I$, an exact sequence
\[
0\to \Hom_\Lambda(C_k,S_i)\to \Hom_\Lambda(Q_k,S_i)\to \Hom_\Lambda(P_k,S_i)\to \Ext^1_\Lambda(C_k,S_i)\to 0.
\]
Since $P_I\notin\add Q_k$, necessarily $\Hom_\Lambda(Q_k,S_i)=0$.  Thus by the above sequence 
\[
e_j\RHom_\Lambda(T_I,S_i)
\cong
\left\{\begin{array}{cc}
\RHom_\Lambda(P_j,S_i) &\mbox{if }j\notin I\\
\RHom_\Lambda(C_j,S_i) &\mbox{if }j\in I
\end{array}
\right.
\cong
\left\{\begin{array}{cc}
0 &\mbox{if }j\notin I\phantom{,}\\
\Hom_\Lambda(P_j,S_i)[-1] &\mbox{if }j\in I,
\end{array}
\right.
\]
and thus $e_j\RHom_\Lambda(T_I,S_i)\cong \mathbb{C}\updelta_{ij}[-1]$ follows.  From this, $\Upphi_I(S_i)\cong S'_i[-1]$.
\end{proof}

We also require the following, which does not need the crepant assumption.
\begin{lemma}\label{E in deg 0}
In the general setup of \ref{general setup}, $\Uppsi_X^{-1}\Lambda_I$ is a sheaf in degree zero, which we denote by $\cE_I$.
\end{lemma}
\begin{proof}
Clearly $\Lambda_I$ is a finitely generated $\Lambda$-module, so $\Uppsi_X^{-1}\Lambda_I\in\Per(X,R)$. Further, by \ref{KIWY} $\RDerived f_*(\Uppsi_X^{-1}\Lambda_I)=0$ since $e\Lambda_I=0$, so since the spectral sequence collapses it follows that $H^{-1}(\Uppsi_X^{-1}\Lambda_I)\in\cC_f$ and $H^{0}(\Uppsi_X^{-1}\Lambda_I)\in\cC_f$.  But since $\Uppsi_X^{-1}\Lambda_I\in\Per(X,R)$, by definition $\Hom(\cC_f,H^{-1}(\Uppsi_X^{-1}\Lambda_I))=0$. Thus $\Hom(H^{-1}(\Uppsi_X^{-1}\Lambda_I),H^{-1}(\Uppsi_X^{-1}\Lambda_I))=0$ and so $H^{-1}(\Uppsi_X^{-1}\Lambda_I)=0$.  Thus  $\Uppsi_X^{-1}\Lambda_I$ is a sheaf in degree zero.
\end{proof}

The following is one of the main results, from which \ref{flop=mut general thm} will follow easily. 
\begin{prop}\label{key flops lemma}
Assume the crepant setup of \ref{crepant setup}, with $d=3$.  We have $f\colon X\to\Spec R$ and as always set $N:=H^0(\cV_{X})$, $\Lambda:=\End_R(N)$, and pick a collection of curves $\bigcup_{i\in I} C_i$ above the origin.  If $\dim_{\mathbb{C}}\Lambda_I<\infty$ and $\upnu_IN\cong H^0(\cV_{X^+})$ for some other $f^+\colon X^+\to\Spec R$ satisfying the crepant setup \ref{crepant setup}, then
\begin{enumerate}
\item $X^+$ is the flop of $X$ at the curves $\bigcup_{i\in I} C_i$.
\item The following diagram of equivalences is naturally commutative
\[
\begin{array}{c}
\begin{tikzpicture}[scale=1.2]
\node (a1) at (0,0) {$\Db(\coh X)$};
\node (a2) at (2.5,0) {$\Db(\coh X^+)$};
\node (b1) at (0,-1) {$\Db(\mod\Lambda)$};
\node (b2) at (2.5,-1) {$\Db(\mod\upnu_I\Lambda).$};
\draw[->] (a1) to node[above] {$\scriptstyle \flop$} (a2);
\draw[->] (b1) to node[above] {$\scriptstyle \Upphi_I$} (b2);
\draw[->] (a1) to node[left] {$\scriptstyle \Uppsi_X$} (b1);
\draw[->] (a2) to node[right] {$\scriptstyle \Uppsi_{X^+}$} (b2);
\end{tikzpicture}
\end{array}
\]
\end{enumerate}
\end{prop}
\begin{proof}
(1) We first establish that there are morphisms
\[
\begin{tikzpicture}
\node (X) at (-1.33,0) {$X$};
\node (Xprime) at (1.33,0) {$X^+$};
\node (Xcon) at (0,-0.8) {$X_{\con}$};
\node (R) at (0,-2) {$\Spec R$};
\draw[->] (X) -- node[above right] {$\scriptstyle g$} (Xcon);
\draw[->] (Xprime) -- node[above left, pos=0.4] {$\scriptstyle g^+$} (Xcon);
\draw[->] (Xcon) -- node[right] {$\scriptstyle $} (R);
\draw[->] (X) -- node[below left] {$\scriptstyle f$} (R);
\draw[->] (Xprime) -- node[below right] {$\scriptstyle f^+$} (R);
\end{tikzpicture}
\]
to a common $X_{\con}$.  We define $X_{\con}$ to be the base space of the contraction of the curves $\bigcup_{i\in I}C_i$ in $X$.  By \ref{KIWY}, $\End_R(N_{I^c})$ is derived equivalent to $X_{\con}$, with $H^0(\cV_{X_{\con}})\cong N_{I^c}$.  

On the other hand, since $\End_R(\upnu_IN)$ is derived equivalent to $X^+$ via $H^0(\cV_{X^+})\cong \upnu_IN$, the summands of $\upnu_IN$ correspond to exceptional curves.  If we contract all the curves corresponding to the summand $J_I^*$, we obtain $X^+_{\con}$ say.  But again by \ref{KIWY} $\End_R(\frac{\upnu_IN}{J_I^*})= \End_R(N_{I^c})$ is derived equivalent to  $X^+_{\con}$ with $H^0(\cV_{X^+_{\con}})\cong N_{I^c}$, so by \ref{Per0 global general} $X_{\con}\cong X^+_{\con}$ and we can suppose that we are in the situation of the diagram above.  As notation, we denote $C_i^+$ to be the curve in $X^+$ corresponding to the summand $J_i^*$ of $\upnu_IN$.

We next claim that $g^+\colon X^+\to X_{\con}$ is the flop of $g$, and to do this we use \ref{check flop}.  First, $g^+\colon X^+\to X_{\con}$ does not contract a divisor to a curve, since $\Lambda_I\cong (\upnu_I\Lambda)_I$ and so $\dim_{\mathbb{C}}(\upnu_I\Lambda)_I<\infty$ by \cite[6.20]{IW4}.  Now with the notation as in \S\ref{perverse and tilting}, we let $D_i$ in $X$ be the Cartier divisor cutting exactly the curve $C_i$, and $D^+_i$ be the Cartier divisor in $X^+$ cutting exactly the curve $C_i^+$.  Since $-(-D_i)$ is $g$-ample, we let $D_i'$ denote the proper transform of $-D_i$.   In what follows, we will use the notation $[-]_X$ to denote something viewed in the class group of $X$.  Since $g$ and $g^+$ give reflexive equivalences, we will also abuse notation and for example refer to the divisor $D_i$ on $X_{\con}$, and on $X^+$.

We next claim that $D_i'$ is Cartier.  Since $X_{\con}$ has only Gorenstein terminal singularities, which locally are hypersurfaces, by \ref{mutmutI general} $\upnu_I\upnu_IN\cong N$, and further recall $\upnu_IN$ has summands $N_i$ ($i\notin I$) and $J_i^*\cong K_i$ ($i\in I$).  Further, by \ref{proj res thm} applied to $\End_R(\upnu_IN)$ there is an exact sequence
\begin{align*}
0\to \Hom_R(\upnu_IN,K_i)\to \Hom_R(\upnu_IN,V_i)&\to \Hom_R(\upnu_IN,U_i)\to\\
&\to \Hom_R(\upnu_IN,K_i)\to C\to 0,
\end{align*}
where $C$ has a finite filtration by the simple $\upnu_I\Lambda$-modules $S'_j$ ($j\in I$). Across the equivalence, this gives an exact sequence
\begin{eqnarray}
0\to \cN_i^+\xrightarrow{\updelta_i} \cW_i\xrightarrow{\upvarepsilon_i}\cU_i\to  \cN_i^+\to\cE\to 0\label{splice Per actually sheaves}
\end{eqnarray}
in $\Per(X^+,R)$ for some $g^+$-trivial bundles $\cW_i$ and $\cU_i$, and by \ref{simple across Db} and \ref{E in deg 0} $\cE$ is a sheaf with a finite filtration with factors from $\cO_{C_j}(-1)$ ($j\in I$).   Since all terms are sheaves, by splicing, considering the associated triangles and then taking cohomology, it follows that \eqref{splice Per actually sheaves} is an exact sequence in $\coh X^+$.  Let $\cF_i:=\Cok\updelta_i$, then for any closed point $x\in X^+\backslash C$, $\cE_x=0$, and so certainly $(\cF_i)_x$ is free.  Further, if $x\in C$ then \eqref{splice Per actually sheaves} localises to a finite projective resolution of $\cE_x$.  Since the $\cO_{C_j}(-1)$ are Cohen--Macaulay, $\depth_{\cO_{X^+,x}}\cE_x=1$, so by Auslander--Buchsbaum $\pd_{\cO_{X^+,x}}\cE_x=2$ and thus $(\cF_i)_x$ is free.  This shows that $\cF_i$ is a locally free sheaf.

Denote $\cG_i:=\Cok\upvarepsilon_i$.  
Since $\mathbf{R}^1 g^+_*\cU_i=0$ as in the proof of \ref{KIWY},  it follows that $\mathbf{R}^1 g^+_*\cG_i=0$.  Then since $\RDerived g^+_*\cE=0$, it then follows that $\mathbf{R}^1 g^+_*\cN_i^+=0$. Thus $\RDerived g^+_*\cN_i^+=g^+_*\cN_i^+$, and again as in the proof of \ref{KIWY}, $\RDerived g^+_*\cW_i=g^+_*\cW_i$.

As a consequence, $\RDerived g^+_*\cF_i=g^+_*\cF_i$ and there is an exact sequence
\begin{eqnarray}
0\to g^+_*\cN_i^+\to g^+_*\cW_i\to g^+_*\cF_i\to 0 \label{pushdown D ses}
\end{eqnarray}
on $X_{\con}$.  Across the equivalence with $\End_R(N_{I^c})$, by \ref{KIWY} this corresponds to a triangle
\[
\Hom_R(N_{I^c},K_i)\xrightarrow{\cdot c_i} \Hom_R(N_{I^c},V_i)\to\Uppsi_{X_{\con}}(g^+_*\cF_i)\to
\] 
But by construction 
\[
0\to \Hom_R(N_{I^c},K_i)\xrightarrow{\cdot c_i} \Hom_R(N_{I^c},V_i)\to\Hom_R(N_{I^c},N_i)\to 0
\]
is exact, and by \ref{KIWY} $g_*\cN_i$ corresponds under the equivalence to $\Hom_R(N_{I^c},N_i)$.   It follows that $g_*\cN_i\cong g^+_*\cF_i$, and so by reflexive equivalence $D_i'$, the proper transform of $-D_i$, is Cartier and is represented by $[\det\cF_i]_{X^+}$. Using the exact sequence
\[
0\to \cN_i^+\to \cW_i\to \cF_i \to 0
\]
it follows that
\[
[D_i']_{X^+}\cdot C^+_j=[\det\cF_i]_{X^+}\cdot C^+_j=-[\det\cN_i^+]_{X^+}\cdot C^+_j=\delta_{ij}
\]
for all $j\in I$, since $\cW_i$ is $g^+$-trivial.  Hence by \ref{check flop} $g^+\colon X^+\to X_{\con}$ is the flop of $g$.\\
(2) Denote $\Uptheta:=\Uppsi_{X^+}^{-1}\circ\Upphi_I\circ \Uppsi_X\colon \Db(\coh X)\to\Db(\coh X^+)$, then it is well known that this is a Fourier--Mukai functor \cite[6.16]{DW1}, and being the composition of equivalences it is also an equivalence.  Note that
\begin{enumerate}
\item $\RDerived g^+_*\circ F\cong \RDerived g_*$,
\item $F(\cO_X)\cong \cO_{X^+}$,
\item $F(\cO_{C_i}(-1))\cong \cO_{C^+_i}(-1)[-1]$ for all $i\in I$,
\end{enumerate}
when $F$ is either $\flop$ or $\Uptheta$.  For $\flop$, the inverse  of the Bridgeland--Chen flop functor, this is well--known; see e.g.\ \cite[7.16]{DW1} or \cite[Appendix B]{TodaGV}.  For the functor $\Uptheta$, denoting $\Gamma:=\End_R(N_{I^c})$, property (1) follows from the commutative diagram
\[
\begin{array}{c}
\begin{tikzpicture}
\node (a1) at (0,0) {$\Db(\coh X)$};
\node (b1) at (0,-1.5) {$\Db(\coh X_{\con})$};
\node (a2) at (3,0) {$\Db(\mod\Lambda)$};
\node (b2) at (3,-1.5) {$\Db(\mod \Gamma)$};
\node (a3) at (6,0) {$\Db(\mod \upnu_I\Lambda)$};
\node (b3) at (6,-1.5) {$\Db(\mod \Gamma)$};
\node (a4) at (9,0) {$\Db(\coh X^+)$};
\node (b4) at (9,-1.5) {$\Db(\coh X_{\con})$};
\draw[->] (a1) -- node[above] {$\scriptstyle\Uppsi_X$}  (a2);
\draw[->] (a2) -- node[above] {$\scriptstyle\Upphi_I$}  (a3);
\draw[->] (a3) -- node[above] {$\scriptstyle\Uppsi_{X^+}^{-1}$}  (a4);
\draw[->] (b1) -- node[above] {$\scriptstyle\Uppsi_{X_{\con}}$}  (b2);
\draw[->] (b2) -- node[above] {$\scriptstyle\Id$}  (b3);
\draw[->] (b3) -- node[above] {$\scriptstyle\Uppsi_{X_{\con}}^{-1}$}  (b4);
\draw[->] (a1) to node[left] {$\scriptstyle \RDerived g_*$} (b1);
\draw[->] (a2) to node[left] {$\scriptstyle e_{I^c}(-)$} (b2);
\draw[->] (a3) to node[right] {$\scriptstyle e^+_{I^c}(-)$} (b3);
\draw[->] (a4) to node[right] {$\scriptstyle \RDerived g^+_*$} (b4);
\end{tikzpicture}
\end{array}
\]
where the two outer squares commute by \ref{KIWY} and the commutativity of the inner square is obvious.  Property (2) follows from \ref{track mutation}\eqref{track mutation 1}, and property (3) from \ref{track mutation}\eqref{track mutation 2}. Hence the functor $\Uptheta^{-1}\circ\flop$ satisfies the conditions (1)--(3) of \ref{Toda argument}, so $\Uptheta^{-1}\circ\flop\cong\upphi_*$ where $\upphi\colon X\to X$ is an isomorphism compatible with the contraction.  Since $\upphi$ is necessarily the identity away from the flopping locus,  $\upphi=\Id$, so $\Uptheta\cong\flop$.
\end{proof}

Thus to prove \ref{flop=mut general thm}, by \ref{key flops lemma} we just need to establish that $\upnu_IN\cong H^0(\cV_{X^+})$ for some $X^+\to\Spec R$.  The trick in \ref{mut=flop single curve min model} in the minimal model case with $I=\{i\}$ was to use Riedtmann--Schofield.   To work in full generality requires another standard technique from representation theory, namely the Bongartz completion.

\begin{proof}[Proof of {\ref{flop=mut general thm}}]
Since $R$ admits an MM module, by Bongartz completion we may find $F\in\refl R$ such that $\End_R(\upnu_IN\oplus F)$ is an MMA \cite[4.18]{IW4}. Since $\upnu_IN$ is a generator, necessarily $F\in\CM R$.  By the Auslander--McKay Correspondence \ref{NCvsC min model} $\upnu_IN\oplus F$ is one of the finite number of MM generators, and further $\upnu_IN\oplus F=H^0(\cV_{Y})$ for some minimal model $Y$, where the non-free summands of $\upnu_IN\oplus F$ correspond to the exceptional curves for $Y\to \Spec R$.  Contracting all the curves in $Y$ that correspond to the summand $F$, as in \eqref{contraction f and g} we factorise $Y\to\Spec R$ as $Y\to X^+\to\Spec R$ for some $X^+$.  By \ref{KIWY} $\upnu_IN\cong H^0(\cV_{X^+})$, so parts \eqref{flop=mut general thm 1} and \eqref{flop=mut general thm 2} both follow from \ref{key flops lemma}.  
\end{proof}

\begin{cor}\label{run MMP general}
We can run the Homological MMP in Figure~\ref{Fig2} when $X$ has only Gorenstein terminal singularities, for arbitrary subsets of curves.
\end{cor}
\begin{proof}
This now follows from \ref{contract on f}, \ref{flop=mut general thm} and \ref{reconstruction new}.
\end{proof}

\begin{cor}\label{mut moduli gives flop text}
With the crepant setup $X\to\Spec R$ with $d=3$, choose a subset of curves $\bigcup_{i\in I}C_i$ and suppose that $\dim_\mathbb{C}\Lambda_I<\infty$, so that $\bigcup_{i\in I}C_i$ flops. Then 
\begin{enumerate}
\item\label{mut moduli gives flop text 1} $\cM_{\rk,\upvartheta}(\Lambda)\cong X$ for all $\upvartheta\in C_+(\Lambda)$.
\item\label{mut moduli gives flop text 2} $\cM_{\rk,\upvartheta}(\upnu_I\Lambda)\cong X^+$ for all $\upvartheta\in C_+(\upnu_I\Lambda)$.
\end{enumerate}
\end{cor}
\begin{proof}
Part (1) follows immediately from \cite[5.2.5]{Joe} applied to $X$.  By \ref{flop=mut general thm}, part (2) follows  from \cite[5.2.5]{Joe} applied to $X^+$.
\end{proof}

The following extends \ref{MMmodule under flop cor} by dropping the $\mathds{Q}$-factorial assumption and considering multiple curves, but now the statement is a little more subtle.

\begin{cor}\label{MMmodule under flop cor 2}
With the crepant setup of \ref{crepant setup}, suppose further that $d=3$.  Set $N:=H^0(\cV_X)$, and pick a subset of curves $\bigcup_{i\in I}C_i$.  Then
\begin{enumerate}
\item\label{MMmodule under flop cor 2 A} If $\bigcup_{i\in I}C_i$ flops (equivalently, by \ref{contract on f}, $\dim_{\mathbb{C}}\Lambda_I<\infty$), then $\upnu_IN\cong  H^0(\cV_{X^+})$.
\item\label{MMmodule under flop cor 2 B} If $I=\{i\}$, $\pd_\Lambda\Lambda_i<\infty$ and $C_i$ does not flop (equivalently,  $\dim_{\mathbb{C}}\Lambda_i=\infty$), then $\upnu_IN\cong  N$.
\end{enumerate}
\end{cor}
\begin{proof}
Part (1) is just \ref{flop=mut general thm}. For part (2), since $\Lambda_i$ is local and has finite projective dimension, by \cite[2.15]{Ramras} $\depth_R\Lambda_i=\dim_R\Lambda_i=\id_{\Lambda_i}\Lambda_i$.  The result follows using the argument of \cite[6.23(1)]{IW4}.  
\end{proof}

\begin{remark}\label{not true multiple}
The statement in \ref{MMmodule under flop cor 2}\eqref{MMmodule under flop cor 2 B} is not true for multiple curves, indeed the hypothesis in \ref{MMmodule under flop cor 2}\eqref{MMmodule under flop cor 2 B} cannot be weakened.  First, if $I\neq\{i\}$ then $\Lambda_i$ is not local and there are examples that satisfy $\upnu_IN\ncong N$ even when $\pd_\Lambda\Lambda_I<\infty$ and $\dim_{\mathbb{C}}\Lambda_I=\infty$.  Second, if $I=\{i\}$ and $\pd_\Lambda\Lambda_I=\infty$, there are examples that satisfy $\dim_{\mathbb{C}}\Lambda_i=\infty$ but $\upnu_IN\ncong N$.

There are two separate problems here, namely in general $\Lambda_I$ need not be perfect, and it need not be Cohen--Macaulay.  Both cause independent technical difficulties, and this will also be evident in \S\ref{stab and mut section}.  See also \ref{B1}.
\end{remark}

One further corollary of this section is that both commutative and noncommutative deformations of curves are preserved under flop.

\begin{cor}
With the crepant setup of \ref{crepant setup}, and $d=3$, pick a  subset $\bigcup_{i\in I}C_i$ of curves, and suppose that $\bigcup_{i\in I}C_i$ flops.  Then
\begin{enumerate}
\item The noncommutative deformation functor of $\bigcup_{i\in I}C_i$ is represented by the same ring as the noncommutative deformation functor of $\bigcup_{i\in I}C^+_i$. 
\item The statement in (1) also holds for commutative deformations.
\end{enumerate}
\end{cor}
\begin{proof}
By \ref{contract on f}, since $\bigcup_{i\in I}C_i$ flops, $\dim_{\mathbb{C}}\Lambda_I<\infty$ and the noncommutative deformations of $\bigcup_{i\in I}C_i$ are represented by $\Lambda_I$.  By \ref{contract on f} and \ref{flop=mut general thm}, the noncommutative deformations of $\bigcup_{i\in I}C^+_i$ are represented by $(\upnu_I\Lambda)_I$.  By \cite[6.20]{IW4} $\Lambda_I\cong(\upnu_I\Lambda)_I$.\\
(2) This follows by taking the abelianization of (1).
\end{proof}

\subsection{Auslander--McKay Revisited} Now that \ref{flop=mut general thm} has been established in full generality, we can extend the Auslander--McKay Correspondence in \ref{NCvsC min model}.

\begin{defin}\label{groupoids definition}
Let $R$ be as above, then
\begin{enumerate}
\item The \emph{derived mutation groupoid} is defined by the following generating set.  It has vertices $\Db(\mod\End_R(N))$, running over all isomorphism classes of basic MM generators $N$, and as arrows each vertex $\Db(\mod\End_R(N))$ has the mutation functors $\Upphi_I$ emerging, as $I$ runs through all possible summands satisfying $\dim_{\mathbb{C}}\Lambda_I<\infty$.
\item The \emph{derived flops groupoid} is defined by the following generating set.  It has vertices $\Db(\coh X)$, running over all minimal models $X$, and as arrows we connect vertices by the inverse of the Bridgeland--Chen flop functors, running through all possible combinations of flopping curves.
\end{enumerate}
\end{defin}

\begin{thm}\label{new thm in email}
The correspondence in \ref{NCvsC min model} further satisfies
\begin{enumerate}
\item[(3)$'$] The full mutation graph of the MM generators coincides with the full flops graph of the minimal models.
\item[(4)] The derived groupoid of the MM modules is functorially isomorphic to the derived flops groupoid of the minimal models.
\end{enumerate}
\end{thm}
\begin{proof}
(3)$'$ By definition, the full mutation graph and derived mutation groupoid only considers $\upnu_I$ provided that $\dim_{\mathbb{C}}\Lambda_I<\infty$, which by \ref{contract on f} is equivalent to the condition that $\bigcup_{i\in I}C_i$ flops.  Hence the result follows by combining the bijection in  \ref{NCvsC min model} with \ref{flop=mut general thm}\eqref{flop=mut general thm 1}.\\
(4) This follows by combining the bijection with \ref{flop=mut general thm}\eqref{flop=mut general thm 2}.  
\end{proof}

\section{Stability and Mutation}\label{stab and mut section}

In this section we relate stability and mutation, then use this together with the Homological MMP (proved in \ref{run MMP general}) to give results in GIT, specifically regarding chamber decompositions and later in \S\ref{CI subsection} the Craw--Ishii conjecture.

After first proving general moduli--tracking results in \S\ref{stability through mutation}, running Figure~\ref{Fig2} over all possibilities and tracking all the moduli back then computes the full GIT chamber decomposition.  We further prove in \S\ref{chamber red to surface subsection} that mutation is preserved under generic hyperplane section, which in effect means (in \S\ref{calculate chamber subsection}) that the chamber decomposition reduces to knitting on ADE surface singularities, which is very easy to calculate.  Amongst other things, this observation can be used to prove the braiding of flops in dimension three \cite{DW3}.

\subsection{GIT background}\label{GIT background}  There are two GIT approaches to moduli that could be used in this paper.  The first is quiver GIT, which relies on presenting $\Lambda^{\op}$ as (the completion of) a quiver with relations, and the second is the more abstract approach given in \cite[\S6.2]{VdBNCCR}.  For most purposes either is sufficient, so for ease of exposition we use quiver GIT.

Consider $\Lambda=\End_R(N)$ and present $\Lambda^{\op}$ as the complete path algebra of a quiver $Q$ subject to relations $I$, where the number of vertices in $Q$ equals the number of indecomposable summands of $N$.  We denote by $Q_0$ the vertex set of $Q$, and remark that under the conventions in \S\ref{conventions}, $\Lambda$-modules correspond to representations of $(Q,I)$.  Below, we will implicitly use this identification.  We call an element $\upbeta\in\mathbb{Z}_{\geq 0}^{|Q_0|}$ a \emph{dimension vector}.   We denote by $(\mathbb{Z}^{|Q_0|})^*$ the dual lattice of $\mathbb{Z}^{|Q_0|}$, and define the parameter space $\Uptheta$ by
\[
\Uptheta:=(\mathbb{Z}^{|Q_0|})^*\otimes_\mathbb{Z}\mathds{Q}.
\]
An element $\upvartheta\in\Uptheta$ is called a \emph{stability parameter}.  For a stability parameter $\upvartheta$  and a dimension vector $\upbeta$, the canonical pairing defines us
\[
\upvartheta\cdot\upbeta:=\sum_{i\in Q_0}\upvartheta_i\upbeta_i.
\]
Given $x\in\fl\Lambda=\rep(Q,I)$, let $\vdim x\in\mathbb{Z}_{\geq 0}^{|Q_0|}$ denote its dimension vector, considering $x$ as a finite dimensional representation.

\begin{defin}\label{defin stab king} \cite{King}
Given $\upvartheta\in\Uptheta$, $x\in\fl\Lambda=\rep(Q,I)$ is called $\upvartheta$-semistable if $\upvartheta\cdot\vdim x=0$ and every subobject $x'\subseteq x$ satisfies $\upvartheta\cdot\vdim x'\geq 0$.  Such an object $x$ is called $\upvartheta$-stable if the only subobjects $x'$ with $\upvartheta\cdot\vdim x'=0$ are $x$ and $0$.  Two $\upvartheta$-semistable modules are called \emph{$S$-equivalent} if they have filtrations by $\upvartheta$-stable modules which give isomorphic associated graded modules.  Further, for a given $\upbeta$, we say that $\upvartheta$ is \emph{generic} if every $\upvartheta$-semistable module of dimension vector $\upbeta$ is $\upvartheta$-stable.
\end{defin}

\begin{notation}
For any $\upvartheta\in\Uptheta$ and any dimension vector $\upbeta$, 
\begin{enumerate}
\item Denote by $\cM_{\upbeta,\upvartheta}(\Lambda)$ the moduli space of $\upvartheta$-semistable $\Lambda$-modules of dimension vector $\upbeta$.
\item Denote by $\cS_{\upvartheta}(\Lambda)$  the full subcategory of $\fl\Lambda$ which has as objects the $\upvartheta$-semistable objects, and denote by $\cS_{\upbeta,\upvartheta}(\Lambda)$  the full subcategory of $\cS_{\upvartheta}(\Lambda)$ consisting of those elements with dimension vector $\upbeta$.   
\end{enumerate}
\end{notation}

By King \cite{King} (see also \cite[6.2.1]{VdBNCCR}) $\cM_{\upbeta,\upvartheta}(\Lambda)$ is a coarse moduli space that parameterises $S$-equivalence classes of $\upvartheta$-semistable modules of dimension vector $\upbeta$.  If further $\upbeta$ is an indivisible vector and $\upvartheta$ is generic, then $\cM_{\upbeta,\upvartheta}(\Lambda)$ is a fine moduli space, and $S$-equivalence classes coincide with isomorphism classes.

\subsection{Tracking Stability Through Mutation}\label{stability through mutation}
In this subsection we track stability conditions through mutation, extending \cite{SY, NS} to work in a much greater level of generality.    Throughout, we will make use of the following setup.

\begin{setup}\label{stab setup}
Suppose that $R$ is a normal complete local $d$-sCY commutative algebra with $d\geq 2$, $M$ is a basic modifying module and $M_I$ is a summand of $M$.   Set $\Lambda:=\End_R(M)$ and $\Gamma:=\upnu_I\Lambda$. We denote the projective $\Lambda$-modules by $P_j$, the simple $\Lambda$-modules by $S_j$, and the simple $\Gamma$-modules by $S'_j$.
\end{setup}

For each indecomposable summand $M^*_i$ of $M^*_I$, consider its minimal right $\add M^*_{I^c}$-approximation
\[
\bigoplus_{j\notin I}M_j^{*\oplus b_{i,j}}\to M_i^*
\]
for some collection $b_{i.j}\in\mathbb{Z}_{\geq 0}$.  Dualizing and using \eqref{K1D} gives an exact sequence 
\begin{eqnarray}
0\to M_i\to\bigoplus_{j\notin I}M_j^{\oplus b_{i,j}}.\label{defin bs}
\end{eqnarray}
Summing the sequences \eqref{defin bs} together gives  the minimal left $\add M_{I^c}$-approximation of $M_I$, namely \eqref{K1D}.  In other words, we decompose $U_I$ as $U_I=\oplus_{i\in I}U_i$, then decompose each $U_i$ as $U_i=\oplus_{j\notin I}M_j^{\oplus b_{i,j}}$.

Applying $\Hom_R(M,-)$ to \eqref{defin bs} gives exact sequences
\begin{eqnarray}
0\to \Hom_R(M,M_i)\to\Hom_R(M,\oplus_{j\notin I}M_j^{\oplus b_{i,j}})\to C_i\to 0\label{defin of Ci}
\end{eqnarray}
for each $i\in I$, and summing the sequences in \eqref{defin of Ci} together gives \eqref{defin of CI}. Hence by definition $T_I=(\oplus_{j\notin I}P_j)\oplus(\oplus_{i\in I} C_i)$, where recall that $T_I$ is the tilting module defined in \S\ref{mut prelim}.

\begin{defin}\label{bi defin post ref}
Suppose that $\upbeta$ is a dimension vector, and $\upvartheta\in\Uptheta$ is a stability condition.  Given the data $\boldb=(b_{i,j})_{i\in I,j\notin I}$ from \eqref{defin bs}, we define the  vectors $\upnu_{\boldb}\upbeta$ and $\upnu_{\boldb}\upvartheta$ by
\[
(\upnu_{\boldb}\upbeta)_i=\left\{\begin{array}{cl} \upbeta_i&\mbox{if }i\notin I\\ \left(\sum_{j\notin I}b_{i,j}\upbeta_j\right)-\upbeta_i&\mbox{if }i\in I \end{array}\right.
\qquad
(\upnu_{\boldb}\upvartheta)_j=\left\{\begin{array}{cl} \upvartheta_j+\sum_{i\in I} b_{i,j}\upvartheta_i&\mbox{if }j\notin I \\ -\upvartheta_j&\mbox{if }j\in I. \end{array}\right.
\]
\end{defin}
Thus given the data of $\boldb=(b_{i,j})$, we thus view $\upnu_{\boldb}$ as an operation on dimension vectors, and as a (different) operation on stability parameters.  

\begin{remark}
We remark that the $b$'s are defined with respect to the mutation $\Lambda\mapsto\upnu_I\Lambda$.  When we iterate and consider another mutation $\upnu_I\Lambda\mapsto\upnu_J\upnu_I\Lambda$, the $b$'s may change for this second mutation.  This change may occur even in the situation $\upnu_I\upnu_I\Lambda\cong \Lambda$, and we are considering the mutation back $\upnu_I\Lambda\mapsto \upnu_I\upnu_I\Lambda\cong \Lambda$. The papers \cite{SY, NS} involve a global rule for $\upnu_{\boldb}\upvartheta$ (in their notation, $s_i\upvartheta$), and this is the reason why their combinatorial rule, and proofs, only work in a very restricted setting.
\end{remark}

The following two lemmas are elementary.

\begin{lemma}\label{stab basic}
For any dimension vector $\upbeta$ and any stability $\upvartheta\in\Uptheta$,
\begin{enumerate}
\item\label{stab basic 1} $\upnu_{\boldb}\upbeta\cdot \upnu_{\boldb}\upvartheta=\upbeta\cdot\upvartheta$. 
\item\label{stab basic 2} $\upnu_{\boldb}\upnu_{\boldb}\upbeta=\upbeta$. 
\item\label{stab basic 3} $\upvartheta\cdot\upnu_{\boldb}\upbeta=\upnu_{\boldb}\upvartheta\cdot\upbeta$.
\end{enumerate}
\end{lemma}
\begin{proof}
This is easily verified by direct calculation.
\end{proof}

\begin{lemma}\label{stab dim vec}
With the setup \ref{stab setup} of this subsection, let $x\in\mod \Lambda$ and $y\in\mod\Gamma$.
\begin{enumerate}
\item\label{stab dim vec 1} If $\Ext^1_\Lambda(T_I,x)=0$, then $\vdim\Hom_\Lambda(T_I,x)=\upnu_{\boldb}\vdim x$.
\item\label{stab dim vec 2} If $\Tor_1^{\Gamma}(T_I,y)=0$, then $\vdim (T_I\otimes_{\Gamma} y)=\upnu_{\boldb}\vdim y$.
\end{enumerate}
\end{lemma}
\begin{proof}
(1) By definition $T_I=(\oplus_{j\notin I}P_j)\oplus(\oplus_{i\in I} C_i)$.  Set $\upbeta:=\vdim x$.   It is clear that 
\[
e_t \Hom_\Lambda(T_I,x)=\left\{\begin{array}{cc} \Hom_\Lambda(P_t,x) & \mbox{if }t\notin I\\ \Hom_\Lambda(C_t,x) & \mbox{if }t\in I \end{array}\right.
\cong
\left\{\begin{array}{cc} e_tx & \mbox{if }t\notin I\\ \Hom_\Lambda(C_t,x) & \mbox{if }t\in I \end{array}\right. 
\]
and thus $(\vdim \Hom_\Lambda(T_I,x))_t=\upbeta_t$ when $t\notin I$, and hence we just need to verify that $\dim_{\mathbb{C}}\Hom_\Lambda(C_t,x)=\left(\sum_{j\notin I}b_{t,j}\upbeta_j\right)-\upbeta_t$. But by the assumptions, applying $\Hom_\Lambda(-,x)$ to the exact sequence \eqref{defin of Ci} gives an exact sequence
\[
0\to \Hom_\Lambda(C_t,x)\to\Hom_\Lambda(\oplus_{j\notin I}P_j^{\oplus b_{t,j}},x)\to \Hom_\Lambda(P_t,x)\to 0.
\]
Counting dimensions, using $\Hom_\Lambda(P_j,x)\cong e_jx$,  yields the result.\\
(2) By assumption and \ref{pd 1 for TI both sides}, $T_I\otimes^{\bf L}_{\Gamma}y=T_I\otimes_{\Gamma}y$.  Mutation gives a derived equivalence, so $\RHom_\Lambda(T_I,T_I\otimes_{\Gamma}y)\cong y$, which implies that $T_I\otimes_{\Gamma}y$ satisfies the conditions in (1), and further $\Hom_\Lambda(T_I,T_I\otimes_{\Gamma}y)\cong y$.  Consequently 
\[
\vdim y=\vdim\Hom_\Lambda(T_I,T_I\otimes_{\Gamma}y)
\stackrel{\scriptsize\mbox{\eqref{stab dim vec 1}}}{=}\upnu_{\boldb}\vdim (T_I\otimes_{\Gamma}y),
\]
which by \ref{stab basic}\eqref{stab basic 2} implies that $\vdim (T_I\otimes_{\Gamma}y)=\upnu_{\boldb}\upnu_{\boldb}\vdim (T_I\otimes_{\Gamma}y)=\upnu_{\boldb}\vdim y$.
\end{proof}

When tracking stability under mutation, as in \ref{not true multiple} the fact that $\Lambda_I$ need not be Cohen-Macaulay and need not be perfect causes problems.  The following two technical results allows us to overcome the first.  To avoid cases in the statement and proof, as a convention $\frac{M}{(a_1,\hdots,a_t)M}:=M$ when $t=0$.
\begin{lemma}\label{technical comm Ext vanishing}
With the setup \ref{stab setup} of this subsection, let $M,N\in\mod\Lambda$ with $\depth_R M:=t$, and choose a regular sequence $\{ a_1,\hdots,a_t\}$ for $M$.   If
\begin{enumerate}
\item $\pd_\Lambda M<\infty$,
\item $N\in\fl\Lambda$,
\item $\Hom_\Lambda\left(N,\frac{M}{(a_1,\hdots,a_t)M}\right)=0$,
\end{enumerate}  
then $\Ext^{d-t}_\Lambda(M,N)=0$.
\end{lemma}
\begin{proof}
When $t=0$,  $\Ext^{d}_\Lambda(M,N)\cong D\Hom_\Lambda(N,M)=0$ since $\Lambda$ is $d$-sCY, $M$ has finite projective dimension and $N$ has finite length.  This establishes the result in the case $t=0$, so we can assume that $t>0$.  Hence $a_1$ exists, and applying $\Hom_\Lambda(-,N)$ to the exact sequence
\[
0\to M\xrightarrow{a_1}M\to\tfrac{M}{a_1M}\to 0
\]
gives an exact sequence
\[
\hdots\to\Ext^{d-t}_\Lambda(M,N)\xrightarrow{\cdot a_1}\Ext^{d-t}_\Lambda(M,N)\to\Ext^{(d-t)+1}_\Lambda\left(\tfrac{M}{a_1M},N\right)\to\hdots
\]
If $\Ext^{d-t}_\Lambda(M,N)\neq 0$, then by Nakayama's Lemma the image of $(\cdot a_1)$ is a proper submodule of $\Ext^{d-t}_\Lambda(M,N)$, which implies that $\Ext^{(d-t)+1}_\Lambda(\tfrac{M}{a_1M},N)\neq 0$.  Inducting along the regular sequence gives $\Ext^d_\Lambda(\frac{M}{(a_1,\hdots,a_t)M},N)\neq 0$.  But again
\[
\Ext^d_\Lambda\left(\tfrac{M}{(a_1,\hdots,a_t)M},N\right)\cong  D\Hom_\Lambda\left(N,\tfrac{M}{(a_1,\hdots,a_t)M}\right)=0
\]
since $\Lambda$ is $d$-sCY, $\frac{M}{(a_1,\hdots,a_t)M}$ has finite projective dimension \cite[4.3.14]{Weibel}, and $N$ has finite length. This is a contradiction, and so $\Ext^{d-t}_\Lambda(M,N)= 0$, as claimed.
\end{proof}

\begin{cor}\label{stab Ext van key}
With the setup \ref{stab setup} of this subsection, suppose that either 
\begin{enumerate}
\item[(a)] $\upnu_IM\cong M$, or
\item[(b)] $d= 3$, $\upnu_I\upnu_IM\cong M$ and $\dim_{\mathbb{C}}\Lambda_I<\infty$.
\end{enumerate}
Set $t:=\depth_R\Lambda_I$, then $\Ext^{d-t}_\Lambda(\Lambda_I,x)=0$ for all $x\in \fl\Lambda$ provided $\Hom_\Lambda(x,S_i)=0$ for all $i\in I$. 
\end{cor}
\begin{proof}
By either \ref{Ext 2 thm}\eqref{Ext 2 thm 2} or \ref{Ext 3 thm}\eqref{Ext 3 thm 2}, $\pd_\Lambda\Lambda_I<\infty$. Thus by \ref{technical comm Ext vanishing} applied with $M=\Lambda_I$ and $N=x$, we only need to verify that $\Hom_\Lambda(x,\tfrac{\Lambda_I}{(a_1,\hdots,a_t)\Lambda_I})$ is zero.  Consider an element $f$, then since $x$ is finite dimensional, so is $\Im f$.  Thus being a submodule of a factor of $\Lambda_I$, $\Im f$ must have a finite filtration with factors from the set $\{ S_i\mid i\in I\}$. Since $\Hom_\Lambda(x,S_i)=0$ for all $i\in I$, inducting along the finite filtration gives $\Hom_\Lambda(x,\Im f)=0$, and hence $\Hom_\Lambda(x,\tfrac{\Lambda_I}{(a_1,\hdots,a_t)\Lambda_I})=0$.\end{proof}

The following, which is a consequence of \ref{surfaces hack} and \ref{stab Ext van key}, will be needed in \ref{main stab track}.

\begin{cor}\label{new stab}
With the setup \ref{stab setup} of this subsection, assume that either
\begin{enumerate}
\item[(a)] $\upnu_IM\cong M$, or
\item[(b)] $\upnu_I\upnu_IM\cong M$ and $\dim_{\mathbb{C}}\Lambda_I<\infty$.
\end{enumerate}
Then for all $x\in\fl\Lambda$ and $y\in\fl\Gamma$, 
\begin{enumerate}
\item\label{new stab 1} $\Ext^{1}_\Lambda(T_I,x)=0$ provided $\Hom_\Lambda(x,S_i)=0$ for all $i\in I$.
\item\label{new stab 2}  $\Tor_{1}^{\Gamma}(T_I,y)=0$ provided $\Hom_\Gamma(S'_i,y)=0$ for all $i\in I$.
\end{enumerate}
\end{cor}
\begin{proof}
Denote $t=\depth\Lambda_I$.\\
(1) In situation (a), by \ref{Ext 2 thm}\eqref{Ext 2 thm 2} $\pd_\Lambda\Lambda_I=2$ and $\Ext_{\Lambda}^1(T_I,x)\cong\Ext_\Lambda^{2}(\Lambda_I,x)$, which is $\Ext_\Lambda^{d-t}(\Lambda_I,x)$ by Auslander--Buchsbaum.  This is zero by \ref{stab Ext van key}.   In situation (b), by \ref{proj res thm}\eqref{proj res thm part 4} the assumptions in fact force $d\leq 3$.  If $d=2$ then the result is precisely \ref{surfaces hack}, so we can assume that $d= 3$.  In this case, by \ref{Ext 3 thm}\eqref{Ext 3 thm 2} $\pd_\Lambda\Lambda_I=3$ and $\Ext_{\Lambda}^1(T_I,x)\cong\Ext_\Lambda^{3}(\Lambda_I,x)=\Ext_\Lambda^{d-t}(\Lambda_I,x)$, which again is zero by \ref{stab Ext van key}.   \\
(2) By \cite[VI.5.1]{CE} There is an isomorphism 
\[
\Tor_{1}^\Gamma(T_I,y)\cong D\Ext^{1}_{\Gamma^{\op}}(T_I,Dy),
\]
where $D$ is the $\mathbb{C}$-dual.  Note that $0=
\Hom_\Gamma(S'_i,y)\cong\Hom_{\Gamma^{\op}}(Dy,DS'_i)$
for all $i\in I$.  Now the simple left $\Gamma^{\op}$-modules are the $DS'_j$, and by \ref{iterate and dual} either the assumptions (a) or (b) hold for $(\upnu_IM)^*$.  Hence by \ref{pd 1 for TI both sides}\eqref{pd 1 for TI both sides 2} we can apply \eqref{new stab 1} to $\Gamma^{\op}\cong\End_R((\upnu_IM)^*)$ to obtain $\Ext^{1}_{\Gamma^{\op}}(T_I,Dy)=0$, and so $\Tor_{1}^\Gamma(T_I,y)=0$.   
\end{proof}

The following lemma is elementary.
\begin{lemma}\label{very easy}
Assume the setup \ref{stab setup} of this subsection.  Suppose that $\upvartheta\in\Uptheta(\Lambda)$ and $\upphi\in\Uptheta(\Gamma)$ are stability parameters, with $x\in \cS_{\upvartheta}(\Lambda)$ and $y\in \cS_{\upphi}(\Gamma)$.
\begin{enumerate}
\item\label{very easy 1} If $\upvartheta_i>0$ for all $i\in I$, then $\Hom_\Lambda(x,S_i)=0$ for all $i\in I$.
\item\label{very easy 2} If $\upphi_i<0$ for all $i\in I$, then $\Hom_\Gamma(S'_i,y)=0$ for all $i\in I$.
\end{enumerate}
\end{lemma}
\begin{proof}
(1) If there exists a non-zero morphism $x\to S_i$, then necessarily it has to be surjective, so there is a short exact sequence
\[
0\to K_i\to x\to S_i\to 0.
\]
This implies that $\upvartheta_i=\upvartheta\cdot\vdim S_i=\upvartheta\cdot\vdim x-\upvartheta\cdot\vdim K_i\leq 0$, since $x\in \cS_{\upvartheta}(\Lambda)$, contradicting the assumption $\upvartheta_i>0$.\\
(2) Any non-zero morphism $S'_i\to y$ is necessarily injective, so $\upphi_i=\upphi\cdot \vdim S'_i\geq 0$ since $y\in \cS_{\upphi}(\Gamma)$. Since $\upphi_i<0$, the morphism must be zero.
\end{proof}

Given the technical preparation above, the following is now very similar to \cite[3.5]{SY}.

\begin{thm}\label{main stab track}
With the setup \ref{stab setup} of this subsection, assume that either 
\begin{enumerate}
\item[(a)] $\upnu_IM\cong M$, or
\item[(b)] $\upnu_I\upnu_IM\cong M$ and $\dim_{\mathbb{C}}\Lambda_I<\infty$.
\end{enumerate}
Then for every dimension vector $\upbeta$, and for every $\upvartheta\in\Uptheta$ with $\upvartheta_i>0$ for all $i\in I$, 
\begin{enumerate}
\item\label{main stab track 1} $\Hom_{\Lambda}(T_I,-)\colon\mathcal{S}_{\upvartheta}(\Lambda)\to\mathcal{S}_{\upnu_{\boldb}\upvartheta}(\Gamma)$ is an exact functor.
\item\label{main stab track 2} $T_I\otimes_\Gamma-\colon\mathcal{S}_{\upnu_{\boldb}\upvartheta}(\Gamma)\to\mathcal{S}_{\upvartheta}(\Lambda)$ is an exact functor.
\item\label{main stab track 3} There is a categorical equivalence
\begin{eqnarray}
\begin{array}{c}
\begin{tikzpicture}
\node (A) at (0,0) {$\mathcal{S}_{\upbeta,\upvartheta}(\Lambda)$};
\node (B) at (4.25,0) {$\mathcal{S}_{\upnu_{\boldb}\upbeta,\upnu_{\boldb}\upvartheta}(\Gamma)$};
\draw[->] (0.75,0.1) -- node [above] {$\scriptstyle \Hom_\Lambda(T_I,-)$} (3,0.1);
\draw [<-] (0.75,-0.1) -- node [below] {$\scriptstyle T_I\otimes_{\Gamma}-$} (3,-0.1);
\end{tikzpicture}
\end{array}\label{proved}
\end{eqnarray}
preserving $S$-equivalence classes, under which $\upvartheta$-stable modules correspond to $\upnu_{\boldb}\upvartheta$-stable modules.
\item\label{main stab track 4} $\upvartheta$ is generic if and only if $\upnu_{\boldb}\upvartheta$ is generic.
\end{enumerate}  
\end{thm}
\begin{proof}
(1) By \ref{new stab}\eqref{new stab 1} and \ref{very easy}\eqref{very easy 1}, $\Hom_\Lambda(T_I,-)$ is exact out of $\cS_{\upvartheta}(\Lambda)$. To see that $\Hom_{\Lambda}(T_I,-)$ maps $\mathcal{S}_{\upvartheta}(\Lambda)$ to $\mathcal{S}_{\upnu_{\boldb}\upvartheta}(\Gamma)$, suppose that $x\in\mathcal{S}_{\upvartheta}(\Lambda)$, let $y:=\Hom_\Lambda(T_I,x)\cong\RHom_\Lambda(T_I,x)$ and consider a $\Gamma$-submodule $y'\subseteq y$.  Since $\upnu_{\boldb}\upvartheta\cdot\vdim y=\upnu_{\boldb}\upvartheta\cdot\upnu_{\boldb}\vdim x=\upvartheta\cdot\vdim x=0$ by \ref{stab dim vec}\eqref{stab dim vec 1} and \ref{stab basic}\eqref{stab basic 1},  it suffices to show that $\upnu_{\boldb}\upvartheta\cdot \vdim y'\geq 0$.  

The inclusion $y'\subseteq y$ induces an exact sequence
\[
0\to y'\to y\to c\to 0.
\]
Since mutation is a derived equivalence, $T_I\otimes_{\Gamma}^{\bf L}y\cong x$, so $\Tor_1^{\Gamma}(T_I,y)=0$.  Thus applying $T_I\otimes_{\Gamma}-$ to the above sequence and using \ref{pd 1 for TI both sides} gives an exact sequence
\begin{eqnarray}
0\to \Tor_1^{\Gamma}(T_I,y')\to 0 \to \Tor_1^{\Gamma}(T_I,c)\to  T_I\otimes_\Gamma y'\to T_I\otimes_\Gamma y\to T_I\otimes_\Gamma c\to 0\label{for splicing 1} 
\end{eqnarray}
of $\Lambda$-modules.  Now $e_j \Tor_1^\Gamma(T_I,c)=\Tor_1^\Gamma(e_jT_I,c)=0$ for all $j\notin I$, since $e_jT_I=e_j\Hom_\Lambda(\Lambda,T_I)\cong\Hom_\Lambda(P_j,T_I)$ is a projective $\Gamma^{\op}$-module if $j\notin I$.  Consequently
\[
(\vdim \Tor_1^\Gamma(T_I,c))_j=
\left\{\begin{array}{cc}
0 &\mbox{if }j\notin I\\
n_j &\mbox{if }j\in I
\end{array}
\right.
\]
for some collection $n_j\in\mathbb{Z}_{\geq 0}$.  Splicing \eqref{for splicing 1} gives an exact sequence
\begin{eqnarray}
0 \to \Tor_1^{\Gamma}(T_I,c)\to  T_I\otimes_\Gamma y'\to a\to 0\label{after splicing 1}
\end{eqnarray}
where $a$ is a submodule of $T_I\otimes_\Gamma y\cong x$, so $\upvartheta\cdot\vdim a\geq 0$ since $x$ is $\upvartheta$-semistable.  Thus applying $\upvartheta\cdot$ to \eqref{after splicing 1} we obtain
\begin{eqnarray}
\upvartheta\cdot\vdim (T_I\otimes_\Gamma y')=\upvartheta\cdot\vdim a+\upvartheta\cdot\vdim \Tor_1^{\Gamma}(T_I,c)=
\upvartheta\cdot\vdim a+\sum_{i\in I}\upvartheta_in_i \geq 0.\label{geq Hom}
\end{eqnarray}
It follows that 
\begin{align*}
\upnu_{\boldb}\upvartheta\cdot\vdim y'
&=\upvartheta\cdot\upnu_{\boldb}\vdim y' \tag{by \ref{stab basic}\eqref{stab basic 3}}\\
&=\upvartheta\cdot \vdim (T_I\otimes_\Gamma y')\tag{by \eqref{for splicing 1} and  \ref{stab dim vec}\eqref{stab dim vec 2}}\\
&\geq 0, \tag{by \eqref{geq Hom}}
\end{align*}
and so $y$ is $\upnu_{\boldb}\upvartheta$-semistable, proving the claim.\\
(2) This is similar to \eqref{main stab track 1}, but we give the proof for completeness. By \ref{new stab}\eqref{new stab 2}  and \ref{very easy}\eqref{very easy 2}, $T_I\otimes_\Gamma-$ is exact out of $\cS_{\upnu_{\boldb}\upvartheta}(\Gamma)$.  To see that $T_I\otimes_\Gamma-$ maps $\mathcal{S}_{\upnu_{\boldb}\upvartheta}(\Gamma)$ to $\mathcal{S}_{\upvartheta}(\Lambda)$, let $y\in\mathcal{S}_{\upnu_{\boldb}\upvartheta}(\Gamma)$ and consider $x:=T_I\otimes_\Gamma y\cong T_I\otimes^{\bf L}_\Gamma y$ where the last isomorphism holds by \ref{pd 1 for TI both sides} and $\Tor_1$ vanishing.  Consider a $\Lambda$-submodule $x'\subseteq x$, then this induces an exact sequence
\[
0\to x'\to x\to d\to 0.
\]
Since $\upvartheta\cdot\vdim x=\upvartheta\cdot\upnu_{\boldb}\vdim y=\upnu_{\boldb}\upvartheta\cdot\vdim y=0$ by \ref{stab dim vec}\eqref{stab dim vec 2} and \ref{stab basic}\eqref{stab basic 3},  it suffices to show that $\upvartheta\cdot\vdim x'\geq 0$, or equivalently $\upvartheta\cdot\vdim d\leq 0$.  

The above exact sequence induces an exact sequence
\begin{eqnarray}
0\to\Hom_\Lambda(T_I,x')\to y\to\Hom_\Lambda(T_I,d)\to\Ext^1_\Lambda(T_I,x')\to 0\to\Ext^1_\Lambda(T_I,d)\to 0,\label{for splicing 2} 
\end{eqnarray}
again using \ref{pd 1 for TI both sides}. Splicing this sequence gives an exact sequence
\begin{eqnarray}
0\to b\to\Hom_\Lambda(T_I,d)\to\Ext^1_\Lambda(T_I,x')\to 0\label{after splicing 2}
\end{eqnarray}
where $\upnu_{\boldb}\upvartheta\cdot \vdim b\leq 0$ since $b$ is a factor of the $\upnu_{\boldb}\upvartheta$-stable module $y$.  But now $e_j \Ext^1_\Lambda(T_I,x')=\Ext^1_\Lambda(P_j,x')=0$ for all $j\notin I$, so
\[
(\vdim \Ext^1_\Lambda(T_I,x'))_j=
\left\{\begin{array}{cc}
0 &\mbox{if }j\notin I\\
m_j &\mbox{if }j\in I
\end{array}
\right.
\]
for some collection $m_j\in\mathbb{Z}_{\geq 0}$.  Hence
\begin{align*}
\upvartheta\cdot\vdim d
&=\upvartheta\cdot\upnu_{\boldb}\vdim \Hom_\Lambda(T_I,d) \tag{by \eqref{for splicing 2} and \ref{stab dim vec}\eqref{stab dim vec 1}}\\
&=\upnu_{\boldb}\upvartheta\cdot\vdim \Hom_\Lambda(T_I,d) \tag{by \ref{stab basic}\eqref{stab basic 3}}\\
&=\upnu_{\boldb}\upvartheta\cdot \vdim b+ \upnu_{\boldb}\upvartheta\cdot\vdim\Ext^1_\Lambda(T_I,x')\tag{by \eqref{after splicing 2}}\\
&=\upnu_{\boldb}\upvartheta\cdot \vdim b+ \sum_{i\in I}(\upnu_{\boldb}\upvartheta)_im_i,
\end{align*}
which is less than or equal to zero.\\
(3) If $x\in\mathcal{S}_{\upbeta, \upvartheta}(\Lambda)$ then by \ref{new stab}\eqref{new stab 1} and \ref{very easy}\eqref{very easy 1} $\Ext^1_\Lambda(T_I,x)=0$.  Thus $\vdim\Hom_\Lambda(T_I,x)=\upnu_{\boldb}\upbeta$ by \ref{stab dim vec}\eqref{stab dim vec 1}.   Similarly if $y\in\mathcal{S}_{\upnu_{\boldb}\upbeta,\upnu_{\boldb}\upvartheta}(\Gamma)$ then by \ref{new stab}\eqref{new stab 2} and and \ref{very easy}\eqref{very easy 2} $\Tor_1^\Gamma(T_I,y)=0$ and so $\vdim(T_I\otimes_\Gamma y)=\upnu_{\boldb}\vdim y=\upnu_{\boldb}\upnu_{\boldb}\upbeta=\upbeta$ by \ref{stab dim vec}\eqref{stab dim vec 2} and \ref{stab basic}\eqref{stab basic 2}.  Thus the functors in \eqref{proved} are well defined, and further since $T_I$ has projective dimension one (on both sides) by \ref{pd 1 for TI both sides}, they are isomorphic to their derived versions.  Since the derived versions are an equivalence, we deduce the underived versions are. They are exact by \eqref{main stab track 1} and \eqref{main stab track 2}, so it follows that they preserve the $S$-equivalence classes.  It is also clear in the above proof that replacing $\geq 0$ by $>0$ throughout shows that under the equivalence, stable modules correspond to stable modules.\\
(4) Follows immediately from \eqref{main stab track 3}.
\end{proof}

Leading up to the next corollary, recall that $\boldb$ is defined in \eqref{defin bs} by decomposing first $U_I=\oplus_{i\in I}U_i$, then further decomposing each $U_i$.  We may play the same trick to the $V_I$'s, namely for each indecomposable summand $M_i$ of $M_I$ consider its minimal right $\add M_{I^c}$-approximation
\[
\bigoplus_{j\notin I}M_j^{\oplus c_{i,j}}\to M_i
\]
for some collection $c_{i,j}\in\mathbb{Z}_{\geq 0}$.  These give $\boldc:=(c_{i,j})_{i\in I,j\notin I}$.  In general, $\boldb\neq \boldc$.

\begin{cor}\label{stab change b and c}
With notation and assumptions as in \ref{main stab track}, 
 for every dimension vector $\upbeta$, and for every $\upvartheta\in\Uptheta$ with $\upvartheta_i>0$ for all $i\in I$,
\begin{enumerate}  
\item\label{stab change b and c 1} There is an isomorphism  $\cM_{\upbeta,\upvartheta}(\Lambda)\cong \cM_{\upnu_{\boldb}\upbeta,\upnu_{\boldb}\upvartheta}(\Gamma)$.
\item\label{stab change b and c 2} There is an isomorphism  $\cM_{\upnu_{\boldc}\upbeta,\upnu_{\boldc}\upvartheta}(\Lambda)\cong \cM_{\upbeta,\upvartheta}(\Gamma)$.
\end{enumerate}
\end{cor}
\begin{proof}
(1) It follows immediately from \ref{main stab track}\eqref{main stab track 3} that there is a bijection on closed points. The fact that \ref{main stab track}\eqref{main stab track 3} holds after base change, and so there is an isomorphism of schemes, is dealt with in \cite[4.20]{SY}, noting the small correction in \cite[Appendix A]{Joe}.\\
(2) By \ref{iterate and dual} either the assumption (a) or (b) holds for $\upnu_IM$. Hence we can apply \ref{main stab track}\eqref{main stab track 3} to the mutation $\Gamma\mapsto\upnu_I\Gamma\cong \Lambda$.  For this mutation, the $b$'s are given by $\boldc$, using \eqref{8O} (and the fact that $W_I\cong V_I$ there).
\end{proof}

Recall from the introduction \S\ref{GIT intro} the definition of the dimension vector $\rk$.
\begin{cor}\label{stab change for b=c}
With the notation and assumptions as in \ref{stab change b and c}, suppose further that $\boldb=\boldc$ (equivalently, $U_i\cong V_i$ for all $i\in I$).  Then for every dimension vector $\upbeta$, and for every $\upvartheta\in\Uptheta$ with $\upvartheta_i\neq 0$ for all $i\in I$, there is an isomorphism
\[
\cM_{\upbeta,\upvartheta}(\Lambda)\cong \cM_{\upnu_{\boldb}\upbeta,\upnu_{\boldb}\upvartheta}(\Gamma).
\]
In particular, if $\boldb=\boldc$, then $\cM_{\rk,\upvartheta}(\Lambda)\cong \cM_{\rk,\upnu_{\boldb}\upvartheta}(\Gamma)$ for every $\upvartheta\in\Uptheta$ with $\upvartheta_i\neq 0$ for all $i\in I$.
\end{cor}

\begin{remark}
We will prove later in \ref{Ui=Vi for cDV} that $U_i\cong V_i$ for all $i\in I$ in the case of cDV singularities, so $\boldb= \boldc$ in this case.  However, even for NCCRs in dimension three with $I=\{i\}$, $\boldb\neq \boldc$ in general.
\end{remark}

\subsection{Chamber Structure: Reduction to Surfaces}\label{chamber red to surface subsection}
In this subsection we revert back to the crepant one dimensional fibre setting of \ref{crepant setup}.  Throughout, we restrict to the dimension vector $\rk$, and show that (for this dimension vector) the chamber structure on the stability parameters can be calculated by passing to a Kleinian singularity.  

\begin{remark}\label{stab abuse}
As the moduli space $\cM_{\rk,\upvartheta}(\Lambda)$ parameterises only semistable $\Lambda$-modules of dimension vector $\rk$, and such modules $x$ necessarily satisfy $\upvartheta\cdot \rk=\upvartheta\cdot \vdim x=0$ by definition of semistability, henceforth we are only concerned with those stability parameters for which $\upvartheta\cdot\rk=0$.  This subspace of $\Uptheta$, which we will temporarily denote by $\Uptheta_{\rk}$, has a wall and chamber structure.  The non-generic parameters cut out walls, dividing the generic parameters of $\Uptheta_{\rk}$ into chambers.  
\end{remark}

Recall in the general setup of \ref{general setup} that $\Lambda=\End_X(\cV_X)$, where $\cV_X$ has a summand $\cO_X$, which has rank one.  Write $\star$ (or sometimes $0$) for the vertex in $\Lambda$ corresponding to $\cO_X$, and consider the dimension vector $\rk$.  Since by definition all elements $\upvartheta\in\Uptheta_{\rk}$ satisfy $\upvartheta\cdot\rk=0$, it follows that
\[
\upvartheta_{\star}=-\sum_{i\in Q_0\backslash \star}(\rank \cN_i)\,\upvartheta_i
\]
and so $\Uptheta_{\rk}$ can be viewed as $\mathds{Q}^{|Q_0|-1}$, with co-ordinates $\upvartheta_i$ for $i\neq 0$.  Later, this means that to calculate the wall and chamber structure in $\Uptheta_{\rk}$, we do not need to mutate at the summand $R$.

Each $\Lambda$ in the general setup of \ref{general setup} has an associated $\Uptheta_{\rk}$, and as the chamber structure of $\Uptheta_{\rk}$ depends on $\Lambda$, later care will be required.  When it is necessary to emphasise which ring is being considered, we will use the notation $\Uptheta_{\rk}(\Lambda)$. 

\begin{notation}
Henceforth, until the end of the paper, we will write $\Uptheta_{\rk}$ as simply $\Uptheta$, and $\Uptheta_{\rk}(\Lambda)$ as simply $\Uptheta(\Lambda)$, with it being implicit that everywhere walls and chambers are discussed, this involves only working with those stability parameters $\upvartheta$ such that $\upvartheta\cdot\rk=0$.  This is an abuse of notation, but it is required to maintain readability later.  
\end{notation}

\begin{lemma}\label{Cplus is chamber}
In the general setup of \ref{general setup}, consider $\Lambda:=\End_X(\cV_X)$, with dimension vector $\rk$.  As above, consider $\Uptheta$ with co-ordinates $\upvartheta_i$ for $i\neq 0$.   Then
\[
C_+:=\{ (\upvartheta_i)\in\Uptheta\mid \upvartheta_i>0\mbox{ for all }i \}
\]
is a chamber in $\Uptheta$.
\end{lemma}
\begin{proof}
It is clear that every element of $C_+$ is generic, and further if $\upvartheta, \upvartheta'\in C_+$, then $x$ is $\upvartheta$-stable if and only if it is $\upvartheta'$-stable.  Hence $C_+$ is contained in some GIT chamber.  It suffices to show that for each $i$,  there exists some $x_i\in\cM_{\rk,C_+}(\Lambda)$ and an injection $S_i\hookrightarrow x_i$, since this implies that $x_i$ is not stable in the limit $\upvartheta_i\rightarrow 0$ and so $\upvartheta_i=0$ then defines a wall.  

Consider the curve $C_i$ and pick a point $y\in C_i$.  There is a short exact sequence
\[
0\to \cO_{C_i}(-1)\to \cO_{C_i}\to\cO_y\to 0
\] 
and thus after tensoring by $\cL_i^{*}$ and rotating gives a triangle
\[
\cO_{C_i}(-1)\to\cO_y\to \cO_{C_i}(-2)[1]\to
\]
in $\Db(\coh X)$.  Applying $\RHom_X(\cV_X,-)$, using \ref{simple across Db}  gives a triangle
\[
S_i\to x_i\to C\to 
\]
in $\Db(\mod\End_X(\cV_X))$, where $x_i$ is a $C_+$-stable module of dimension vector $\rk$ by \cite[\S5.2]{Joe}.  The first morphism is non-zero, and so since $S_i=\mathbb{C}$ it is necessarily injective. 
\end{proof}

The strategy to describe the chambers of $\Uptheta(\Lambda)$ is to track $C_+$ through mutation, and calculate the combinatorics by passing to surfaces.  This requires a special case of the following general result.
\begin{prop}\label{mut preserved by generic}
Let $R$ be a complete local $3$-sCY normal domain, suppose that $M$ is modifying, and $M_I$ is a summand of $M$ with $R\in\add M_{I^c}$.  Consider the exchange sequence \eqref{K0prime}
\[
0\to K_i\xrightarrow{c_i} V_i\xrightarrow{a_i} M_i.
\]
Then $a_i$ is surjective.  Further, for any $x\in\m$ which is an $\Ext^1_R(M,M)$-regular element, denoting $F:=(R/xR)\otimes_R-$, then
\begin{enumerate}
\item\label{mut preserved by generic 1}  $(R/xR)\otimes_R\End_R(M)\cong\End_{FR}(FM)$, and $FM$ is indecomposable.
\item\label{mut preserved by generic 2} If further $x$ is $\Ext^1_R(K_i,K_i)$-regular, the sequence
\begin{eqnarray}
0\to F{K}_i\to F{V}_i\xrightarrow{F{a_i}} F{M}_i\to 0\label{after cut}
\end{eqnarray}
is exact, and further $F{a_i}$ is a minimal $\add F{M_{I^c}}$-approximation.
\end{enumerate}
\end{prop}
\begin{proof}
(1) The first statement is well-known; see for example the argument in \cite[5.24]{IW6}, which uses the fact that $x$ is $\Ext^1_R(M,M)$-regular.  The second follows from the first, since if $FM$ decomposes, since $R$ is complete local we can lift idempotents to obtain a contradiction.\\
(2) Since $M_{I^c}$ is a generator and $\Hom_R(M_{I^c},-)$ applied to \eqref{K0} is exact, it follows that $a_i$ is surjective.  Also, since $M_{I^c}$ is a generator and $\End_R(M)\in\CM R$, necessarily $M\in\CM R$ and since CM modules are closed under kernels of epimorphisms, $K_I\in\CM R$.  

Now since $x\in\m$ and CM modules are submodules of free modules, $x$ is not a zero divisor on any of the modules in \eqref{K0}, thus
\[
\begin{tikzpicture}
\node (A0) at (0,0) {$0$};
\node (A1) at (1,0) {$K_i$};
\node (A2) at (2.25,0) {$V_i$};
\node (A3) at (3.5,0) {$M_i$};
\node (A4) at (4.5,0) {$0$};
\node (B0) at (0,-1.25) {$0$};
\node (B1) at (1,-1.25) {$K_i$};
\node (B2) at (2.25,-1.25) {$V_i$};
\node (B3) at (3.5,-1.25) {$M_i$};
\node (B4) at (4.5,-1.25) {$0$};
\draw[->] (A0)--(A1);
\draw[->] (A1)--(A2);
\draw[->] (A2)--(A3);
\draw[->] (A3)--(A4);
\draw[->] (B0)--(B1);
\draw[->] (B1)--(B2);
\draw[->] (B2)--(B3);
\draw[->] (B3)--(B4);
\draw[right hook->] (A1) -- node[left] {$\scriptstyle x$} (B1);
\draw[right hook->] (A2) -- node[left] {$\scriptstyle x$} (B2);
\draw[right hook->] (A3) -- node[left] {$\scriptstyle x$} (B3);
\end{tikzpicture}
\]
and so by the snake lemma \eqref{after cut} is exact.  Now since $a_i$ is an $\add M_{I^c}$-approximation, there is a commutative diagram
\[
\begin{array}{c}
\begin{tikzpicture}
\node (a1) at (0,0) {$\Hom_R(M_{I^c},V_i)$};
\node (a2) at (5.5,0) {$\Hom_R(M_{I^c},M_i)$};
\node (b1) at (0,-1.25) {$(R/xR)\otimes_R\Hom_R(M_{I^c},V_i)$};
\node (b2) at (5.5,-1.25) {$(R/xR)\otimes_R\Hom_R(M_{I^c},M_i)$};
\node (c1) at (0,-2.5) {$\Hom_{F{R}}(F{M}_{I^c},F{V}_i)$};
\node (c2) at (5.5,-2.5) {$\Hom_{F{R}}(F{M}_{I^c},F{M}_i)$};
\draw[->>] (a1) -- node[above] {$\scriptstyle \cdot a_i$} (a2);
\draw[->] (b1) -- node[above] {$\scriptstyle F(\cdot a_i)$}  (b2);
\draw[->] (c1) --  node [above] {$\scriptstyle \cdot Fa_i$} (c2);
\draw[->>] (a1) -- (b1);
\draw[->] (b1) -- node[left] {$\scriptstyle \sim$} (c1);
\draw[->>] (a2) to (b2);
\draw[->] (b2) to node[right] {$\scriptstyle \sim$} (c2);
\end{tikzpicture}
\end{array}
\]
where the bottom two vertical maps are isomorphisms by \eqref{mut preserved by generic 1}.  It follows that the bottom horizontal map is surjective, so $Fa_i$ is indeed an $\add FM_{I^c}$-approximation.  For minimality, by \eqref{mut preserved by generic 1} applied to $K_i$, $\End_R(K_i)/x\End_R(K_i)\cong \End_{FR}(FK_i)$.  Thus if $Fa_i$ were not minimal then $FK_i$ would decompose into more summands than $K_i$, which since $\End_R(K_i)/x\End_R(K_i)\cong \End_{FR}(FK_i)$  is impossible since $R$ is complete local so we can lift idempotents.
\end{proof}

In the crepant setup of \ref{crepant setup} with $d=3$, by Reid's general elephant principle \cite[1.1, 1.14]{Pagoda}, cutting by a generic hyperplane section yields
\[
\begin{array}{c}
\begin{tikzpicture}[yscale=1.25]
\node (Xp) at (-1,0) {$X_2$}; 
\node (X) at (1,0) {$X$};
\node (Rp) at (-1,-1) {$\Spec (R/g)$}; 
\node (R) at (1,-1) {$\Spec R$};
\draw[->] (Xp) to node[above] {$\scriptstyle $} (X);
\draw[->] (Rp) to node[above] {$\scriptstyle $} (R);
\draw[->] (Xp) --  node[left] {$\scriptstyle \upvarphi$}  (Rp);
\draw[->] (X) --  node[right] {$\scriptstyle f$}  (R);
\end{tikzpicture}
\end{array}
\]
where $R/g$ is an ADE singularity and $\upvarphi$ is a partial crepant resolution. Since $N\in\CM R$ and $g$ is not a zero-divisor on $N$, necessarily $N/gN\in\CM R/g$, and so any indecomposable summand $N_i$ of $N$ cuts to $N_i/gN_i$, which must correspond to a vertex in an ADE Dynkin diagram via the original Auslander--McKay correspondence. 

Following the notation from \cite{Katz}, we encode $X_2$ pictorially by simply describing which curves are blown down from the minimal resolution.  The diagrams 

\[
\begin{array}{cccc}
\begin{array}{c}
\begin{tikzpicture}
\node (0) at (0,0) [DW] {};
\node (1) at (0.75,0) [DW] {};
\node (1b) at (0.75,0.75) [DW] {};
\node (2) at (1.5,0) [DW] {};
 \node (3) at (2.25,0) [DW] {};
\draw [-] (0) -- (1);
\draw [-] (1) -- (2);
\draw [-] (2) -- (3);
\draw [-] (1) -- (1b);
\end{tikzpicture}
\end{array}
&&
\begin{array}{c}
\begin{tikzpicture}
\node (0) at (0,0) [DW] {};
\node (1) at (0.75,0) [DW] {};
\node (1b) at (0.75,0.75) [DB] {};
\node (2) at (1.5,0) [DB] {};
 \node (3) at (2.25,0) [DB] {};
\draw [-] (0) -- (1);
\draw [-] (1) -- (2);
\draw [-] (2) -- (3);
\draw [-] (1) -- (1b);
\end{tikzpicture}
\end{array}
\end{array}
\]
represent, respectively, the minimal resolution of the $D_5$ surface singularity, and the partial resolution obtained from it by contracting the curves corresponding to the black vertices.

\begin{cor}\label{chamber 3fold=chamber surface prep}
With the crepant setup \ref{crepant setup} with $d=3$, if $g$ is a sufficiently generic hyperplane section, then
\begin{enumerate}
\item\label{chamber 3fold=chamber surface prep 1} $\Lambda/g\Lambda\cong \End_{R/gR}(N/gN)$.
\item\label{chamber 3fold=chamber surface prep 2} Let $N_I$ be a summand as in \ref{N notation}, and consider the exchange sequences
\begin{align*}
0\to K_i\xrightarrow{c_i} V_i\xrightarrow{a_i} N_i\\
0\to J_i\xrightarrow{d_i} U_i^*\xrightarrow{b_i} N_i^*
\end{align*}
Then $a_i$ and $b_i$ are surjective, and 
\begin{align*}
0\to FK_i\xrightarrow{Fc_i} FV_i\xrightarrow{Fa_i} FN_i\to0\\
0\to FJ_i\xrightarrow{Fd_i} FU^*_i\xrightarrow{Fb_i} FN_i^*\to0
\end{align*}
are exact, with $Fa_i$ and $Fb_i$ being minimal right approximations.
\item\label{chamber 3fold=chamber surface prep 3} We have that $0\to J_i\xrightarrow{d_i} U_i^*\xrightarrow{b_i} N_i^*\to 0$ is exact, inducing an exact sequence
\[
0\to FN_i\xrightarrow{F(b_i^*)} FU_i\xrightarrow{F(d_i^*)} FJ_i^*\to0
\]
where $F(b_i^*)$ is a minimal left $\add FM_{I^c}$-approximation.
\end{enumerate}
\end{cor}
\begin{proof}
(1)(2) Since $\End_R(N)\in\CM R$, $\depth\Ext^1_R(N,N)>0$ and so if an element $g$ acts on $E_N:=\Ext^1_R(N,N)$ as a zero divisor, then it is contained in one of the finitely many associated primes of $E_N$, which are all non-maximal.  We can apply the same logic to both $K_I$ and $J_I^*$, and thus the finite number of associated primes of $E_N\oplus E_{K_I}\oplus E_{J_I^*}$ are non-maximal.  Hence  we can find $g$ sufficiently generic to be $E_N\oplus E_{K_I}\oplus E_{J_I^*}$-regular, so the first two parts follow from \ref{mut preserved by generic}.\\
(3) This is just the dual of \eqref{chamber 3fold=chamber surface prep 2}, and follows by \ref{dualityofapprox},  part \eqref{chamber 3fold=chamber surface prep 2} and isomorphisms such as $ \Hom_{FR}(FN_i^*,FR)\cong F\Hom_{R}(N_i^*,R)\cong FN_i$.
\end{proof}

The proof of the next lemma, \ref{Ui=Vi for cDV}, requires a little knowledge regarding knitting on AR quivers, which we briefly review in an example.  We refer the reader to \cite[\S4]{IW} for full details.
\begin{example}\label{knitting demonstration}
Consider the $D_5$ ADE surface singularity $R$.  The AR quiver, which coincides with the McKay quiver \cite{Aus3}, is 
\[
\begin{array}{c}
\begin{tikzpicture}[>=stealth,xscale=1.6,yscale=1.5]
\node (A0) at (0,0)  {$\scriptstyle R$};
\node (A1) at (1,0)  {$\scriptstyle R$};
\node (B0) at (0,-0.5) {$\scriptstyle A_1$};
\node (B1) at (0.5,-0.5) {$\scriptstyle B_1$};
\node (B2) at (1,-0.5) {$\scriptstyle A_1$};
\node (C0) at (0,-1)  {$\scriptstyle B_2$};
\node (C1) at (0.5,-1)  {$\scriptstyle A_2$};
\node (C2) at (1,-1)   {$\scriptstyle B_2$};
\node (D0) at (0.5,-1.5)  {$\scriptstyle A_3$};
\draw[->] (A0)--(B1);
\draw[->] (B0)--(B1);
\draw[->] (B1)--(B2);
\draw[->] (B1)--(A1);
\draw[->] (C0)--(B1);
\draw[->] (B1)--(C2);
\draw[->] (C0)--(C1);
\draw[->] (C1)--(C2);
\draw[->] (C0)--(D0);
\draw[->] (D0)--(C2);
\end{tikzpicture}
\end{array}
\]
where the left and right hand sides are identified.  In this example, suppose that $N:=R\oplus A_1\oplus B_2$ and $N_I:=B_2$ (so that $N_{I^c}:=R\oplus A_1$). We calculate the minimal $\add N_{I^c}$-approximation of $N_I$.

Consider the cover of the AR quiver, since this is easiest to draw, and drop labels and the directions of arrows.  The calculation begins by placing a $1$ in the position of $B_2$ (boxed below), and circling all the vertices corresponding to indecomposable summands of $N_{I^c}$
\[
\begin{array}{c}
\begin{tikzpicture}[>=stealth,scale=1.35]
\node (A-2) at (-2,0)  [vertex] {};
\node (A-1) at (-1,0)  [vertex] {};
\node (A0) at (0,0)  [vertex] {};
\node (A1) at (1,0)  [vertex] {};
\node (B-4) at (-2,-0.5)  [vertex] {};
\node (B-3) at (-1.5,-0.5)  [vertex] {};
\node (B-2) at (-1,-0.5)  [vertex] {};
\node (B-1) at (-0.5,-0.5)  [vertex] {};
\node (B0) at (0,-0.5)  [vertex] {};
\node (B1) at (0.5,-0.5)  [vertex] {};
\node (B2) at (1,-0.5)  [vertex] {};
\node (C-4) at (-2,-1)   [vertex] {};
\node (C-3) at (-1.5,-1)   [vertex] {};
\node (C-2) at (-1,-1)   [vertex] {};
\node (C-1) at (-0.5,-1)   [vertex] {};
\node (C0) at (0,-1)   [vertex] {};
\node (C1) at (0.5,-1)   [vertex] {};
\node (C2) at (1,-1)   {$1$};
\node (D-2) at (-1.5,-1.5)   [vertex] {};
\node (D-1) at (-0.5,-1.5)   [vertex] {};
\node (D0) at (0.5,-1.5)   [vertex] {};
\node at (-2.5,-0.75) {$\hdots$};
\draw[-] (A0)--(B1);
\draw[-] (B0)--(B1);
\draw[-] (B1)--(B2);
\draw[-] (B1)--(A1);
\draw[-] (C0)--(B1);
\draw[-] (B1)--(C2);
\draw[-] (C0)--(C1);
\draw[-] (C1)--(C2);
\draw[-] (C0)--(D0);
\draw[-] (D0)--(C2);

\draw[-] (A-1)--(B-1);
\draw[-] (B-2)--(B-1);
\draw[-] (B-1)--(B0);
\draw[-] (B-1)--(A0);
\draw[-] (C-2)--(B-1);
\draw[-] (B-1)--(C0);
\draw[-] (C-2)--(C-1);
\draw[-] (C-1)--(C0);
\draw[-] (C-2)--(D-1);
\draw[-] (D-1)--(C0);

\draw[-] (A-2)--(B-3);
\draw[-] (B-4)--(B-3);
\draw[-] (B-3)--(B-2);
\draw[-] (B-3)--(A-1);
\draw[-] (C-4)--(B-3);
\draw[-] (B-3)--(C-2);
\draw[-] (C-4)--(C-3);
\draw[-] (C-3)--(C-2);
\draw[-] (C-4)--(D-2);
\draw[-] (D-2)--(C-2);
\draw (1,0) circle (3.3pt);
\draw (0,0) circle (3.3pt);
\draw (-1,0) circle (3.3pt);
\draw (-2,0) circle (3.3pt);
\draw (1,-0.5) circle (3.3pt);
\draw (0,-0.5) circle (3.3pt);
\draw (-1,-0.5) circle (3.3pt);
\draw (-2,-0.5) circle (3.3pt);
\node[draw,regular polygon, regular polygon sides=4, minimum size=0.6cm] at (1,-1) {};
\end{tikzpicture}
\end{array}
\]
The calculation continues by counting backwards, using the usual knitting rule that for any given AR sequence, the left-hand value is the sum of the middle values, minus the right-hand value.   Doing this, we obtain
\[
\begin{array}{cccccccc}
\begin{array}{c}
\begin{tikzpicture}[>=stealth,scale=1.35]
\node (A1) at (1,0)  [vertex] {};
\node (B1) at (0.5,-0.5)  {$ 1$};
\node (B2) at (1,-0.5)  [vertex] {};
\node (C1) at (0.5,-1)  {$ 1$};
\node (C2) at (1,-1)   {$ 1$};
\node (D0) at (0.5,-1.5)   {$ 1$};
\draw[-] (B1)--(B2);
\draw[-] (B1)--(A1);
\draw[-] (B1)--(C2);
\draw[-] (C1)--(C2);
\draw[-] (D0)--(C2);
\draw (1,0) circle (3.3pt);
\draw (1,-0.5) circle (3.3pt);
\node[draw,regular polygon, regular polygon sides=4, minimum size=0.6cm] at (1,-1) {};
\end{tikzpicture}
\end{array}
&&
\begin{array}{c}
\begin{tikzpicture}[>=stealth,scale=1.35]
\node (A0) at (0,0)  {$ 1$};
\node (A1) at (1,0)  [vertex] {};
\node (B0) at (0,-0.5)  {$ 1$};
\node (B1) at (0.5,-0.5)  {$ 1$};
\node (B2) at (1,-0.5)  [vertex] {};
\node (C0) at (0,-1)   {$ 2$};
\node (C1) at (0.5,-1)   {$ 1$};
\node (C2) at (1,-1)   {$ 1$};
\node (D0) at (0.5,-1.5) {$ 1$};
\draw[-] (A0)--(B1);
\draw[-] (B0)--(B1);
\draw[-] (B1)--(B2);
\draw[-] (B1)--(A1);
\draw[-] (C0)--(B1);
\draw[-] (B1)--(C2);
\draw[-] (C0)--(C1);
\draw[-] (C1)--(C2);
\draw[-] (C0)--(D0);
\draw[-] (D0)--(C2);
\draw (1,0) circle (3.3pt);
\draw (0,0) circle (4.5pt);
\draw (1,-0.5) circle (3.3pt);
\draw (0,-0.5) circle (4.5pt);
\node[draw,regular polygon, regular polygon sides=4, minimum size=0.6cm] at (1,-1) {};
\end{tikzpicture}
\end{array}
&&
\begin{array}{c}
\begin{tikzpicture}[>=stealth,scale=1.35]
\node (A0) at (0,0)   {$ 1$};
\node (A1) at (1,0)  [vertex] {};
\node (B-1) at (-0.5,-0.5)   {$ 1$};
\node (B0) at (0,-0.5)   {$ 1$};
\node (B1) at (0.5,-0.5)  {$ 1$};
\node (B2) at (1,-0.5)  [vertex] {};
\node (C-1) at (-0.5,-1)   {$ 1$};
\node (C0) at (0,-1)    {$ 2$};
\node (C1) at (0.5,-1)  {$ 1$};
\node (C2) at (1,-1)   {$ 1$};
\node (D-1) at (-0.5,-1.5)   {$ 1$};
\node (D0) at (0.5,-1.5)   {$ 1$};
\draw[-] (A0)--(B1);
\draw[-] (B0)--(B1);
\draw[-] (B1)--(B2);
\draw[-] (B1)--(A1);
\draw[-] (C0)--(B1);
\draw[-] (B1)--(C2);
\draw[-] (C0)--(C1);
\draw[-] (C1)--(C2);
\draw[-] (C0)--(D0);
\draw[-] (D0)--(C2);
\draw[-] (D-1)--(C0);
\draw[-] (C-1)--(C0);
\draw[-] (B-1)--(C0);
\draw[-] (B-1)--(B0);
\draw[-] (B-1)--(A0);
\draw (1,0) circle (3.3pt);
\draw (0,0) circle (4.5pt);
\draw (0,-0.5) circle (4.5pt);
\node[draw,regular polygon, regular polygon sides=4, minimum size=0.6cm] at (1,-1) {};
\end{tikzpicture}
\end{array}\\
\mbox{Step 1}
&&
\mbox{Step 2}
&&
\mbox{Step 3}
\end{array}
\]
where in Step 3 the values in the circled vertices act like zero.  Continuing, the process stops when the value $-1$ appears
\[
\begin{array}{c}
\begin{tikzpicture}[>=stealth,scale=1.35]
\node (A-2) at (-2,0)  {$ 0$};
\node (A-1) at (-1,0)  {$ 1$};
\node (A0) at (0,0)  {$ 1$};
\node (A1) at (1,0)  [vertex] {};
\node (B-4) at (-2,-0.5)  {$ 0$};
\node (B-3) at (-1.5,-0.5)  {$ 0$};
\node (B-2) at (-1,-0.5)  {$ 1$};
\node (B-1) at (-0.5,-0.5)  {$ 1$};
\node (B0) at (0,-0.5)  {$ 1$};
\node (B1) at (0.5,-0.5)  {$ 1$};
\node (B2) at (1,-0.5) [vertex] {};
\node (C-4) at (-2,-1)   {$-1$};
\node (C-3) at (-1.5,-1)   {$ 0$};
\node (C-2) at (-1,-1)   {$ 1$};
\node (C-1) at (-0.5,-1)   {$ 1$};
\node (C0) at (0,-1)   {$ 2$};
\node (C1) at (0.5,-1)   {$ 1$};
\node (C2) at (1,-1)   {$ 1$};
\node (D-2) at (-1.5,-1.5)   {$ 0$};
\node (D-1) at (-0.5,-1.5)   {$ 1$};
\node (D0) at (0.5,-1.5)   {$ 1$};
\draw[-] (A0)--(B1);
\draw[-] (B0)--(B1);
\draw[-] (B1)--(B2);
\draw[-] (B1)--(A1);
\draw[-] (C0)--(B1);
\draw[-] (B1)--(C2);
\draw[-] (C0)--(C1);
\draw[-] (C1)--(C2);
\draw[-] (C0)--(D0);
\draw[-] (D0)--(C2);

\draw[-] (A-1)--(B-1);
\draw[-] (B-2)--(B-1);
\draw[-] (B-1)--(B0);
\draw[-] (B-1)--(A0);
\draw[-] (C-2)--(B-1);
\draw[-] (B-1)--(C0);
\draw[-] (C-2)--(C-1);
\draw[-] (C-1)--(C0);
\draw[-] (C-2)--(D-1);
\draw[-] (D-1)--(C0);

\draw[-] (A-2)--(B-3);
\draw[-] (B-4)--(B-3);
\draw[-] (B-3)--(B-2);
\draw[-] (B-3)--(A-1);
\draw[-] (C-4)--(B-3);
\draw[-] (B-3)--(C-2);
\draw[-] (C-4)--(C-3);
\draw[-] (C-3)--(C-2);
\draw[-] (C-4)--(D-2);
\draw[-] (D-2)--(C-2);
\draw (1,0) circle (3.3pt);
\draw (0,0) circle (4.5pt);
\draw (-1,0) circle (4.5pt);
\draw (-2,0) circle (4.5pt);
\draw (1,-0.5) circle (3.3pt);
\draw (0,-0.5) circle (4.5pt);
\draw (-1,-0.5) circle (4.5pt);
\draw (-2,-0.5) circle (4.5pt);
\node[draw,regular polygon, regular polygon sides=4, minimum size=0.6cm] at (1,-1) {};
\end{tikzpicture}
\end{array}
\]
Summing up the values on the circled vertices gives $R^{\oplus 2}\oplus A_1^{\oplus 2}$, and the vertex with $-1$ corresponds to $B_2$.  From this, we read off that
\[
0\to B_2\to R^{\oplus 2}\oplus A_1^{\oplus 2}\to B_2\to 0
\]
is an exact sequence, with the first map a minimal right $\add N_{I^c}$-approximation.  The calculation for the minimal left approximation is similar, by placing a $1$ and counting forwards to give
\[
\begin{array}{c}
\begin{tikzpicture}[>=stealth,scale=1.35]
\node (A-2) at (-2,0)  [vertex] {};
\node (A-1) at (-1,0)  {$ 1$};
\node (A0) at (0,0)  {$ 1$};
\node (A1) at (1,0)  {$ 0$};
\node (B-4) at (-2,-0.5)  [vertex] {};
\node (B-3) at (-1.5,-0.5)  {$ 1$};
\node (B-2) at (-1,-0.5)  {$ 1$};
\node (B-1) at (-0.5,-0.5)  {$ 1$};
\node (B0) at (0,-0.5)  {$ 1$};
\node (B1) at (0.5,-0.5)  {$ 0$};
\node (B2) at (1,-0.5) {$ 0$};
\node (C-4) at (-2,-1)   {$ 1$};
\node (C-3) at (-1.5,-1)   {$ 1$};
\node (C-2) at (-1,-1)   {$ 2$};
\node (C-1) at (-0.5,-1)   {$ 1$};
\node (C0) at (0,-1)   {$ 1$};
\node (C1) at (0.5,-1)   {$ 0$};
\node (C2) at (1,-1)   {$- 1$};
\node (D-2) at (-1.5,-1.5)   {$ 1$};
\node (D-1) at (-0.5,-1.5)   {$ 1$};
\node (D0) at (0.5,-1.5)   {$ 0$};
\draw[-] (A0)--(B1);
\draw[-] (B0)--(B1);
\draw[-] (B1)--(B2);
\draw[-] (B1)--(A1);
\draw[-] (C0)--(B1);
\draw[-] (B1)--(C2);
\draw[-] (C0)--(C1);
\draw[-] (C1)--(C2);
\draw[-] (C0)--(D0);
\draw[-] (D0)--(C2);

\draw[-] (A-1)--(B-1);
\draw[-] (B-2)--(B-1);
\draw[-] (B-1)--(B0);
\draw[-] (B-1)--(A0);
\draw[-] (C-2)--(B-1);
\draw[-] (B-1)--(C0);
\draw[-] (C-2)--(C-1);
\draw[-] (C-1)--(C0);
\draw[-] (C-2)--(D-1);
\draw[-] (D-1)--(C0);

\draw[-] (A-2)--(B-3);
\draw[-] (B-4)--(B-3);
\draw[-] (B-3)--(B-2);
\draw[-] (B-3)--(A-1);
\draw[-] (C-4)--(B-3);
\draw[-] (B-3)--(C-2);
\draw[-] (C-4)--(C-3);
\draw[-] (C-3)--(C-2);
\draw[-] (C-4)--(D-2);
\draw[-] (D-2)--(C-2);
\draw (1,0) circle (4.5pt);
\draw (0,0) circle (4.5pt);
\draw (-1,0) circle (4.5pt);
\draw (-2,0) circle (3.3pt);
\draw (1,-0.5) circle (4.5pt);
\draw (0,-0.5) circle (4.5pt);
\draw (-1,-0.5) circle (4.5pt);
\draw (-2,-0.5) circle (3.3pt);
\node[draw,regular polygon, regular polygon sides=4, minimum size=0.6cm] at (-2,-1) {};
\end{tikzpicture}
\end{array}
\]
This also gives the exact sequence $0\to B_2\to R^{\oplus 2}\oplus A_1^{\oplus 2}\to B_2\to 0$.  Observe that the second calculation can be obtained from the first simply by reflecting in the vertical line through the boxed vertex.
\end{example}

The following is an extension of the above observation.

\begin{lemma}\label{Ui=Vi for cDV}  
With the crepant setup \ref{crepant setup}, consider the modifying $R$-module $N:=H^0(\cV_X)$ and choose a summand $N_I$ as in \ref{N notation}.  With the notation in \ref{chamber 3fold=chamber surface prep},
\begin{enumerate}
\item $FU_i\cong FV_i$.
\item $U_i\cong V_i$.
\end{enumerate}
\end{lemma}
\begin{proof}
As in \ref{chamber 3fold=chamber surface prep}, let $g$ be a sufficiently generic hyperplane section and let $F:=(R/gR)\otimes_R(-)$.  As before, decompose $U_i\cong\bigoplus_{j\notin I}N_j^{\oplus b_{i,j}}$ and $V_i\cong \bigoplus_{j\notin I}N_j^{\oplus c_{i,j}}$, from which it is clear that $FU_i\cong\bigoplus_{j\notin I}(FN_j)^{\oplus b_{i,j}}$ and $FV_i\cong \bigoplus_{j\notin I}(FN_j)^{\oplus c_{i,j}}$.  By \ref{mut preserved by generic}\eqref{mut preserved by generic 1}  the $FN_j$ are indecomposable, so by Krull--Schmidt to prove both parts it suffices to show that $b_{i,j}=c_{i,j}$ for all $i\in I$, $j\notin I$.

Both $FU_i$ and $FV_i$ can be calculated by knitting on the AR quiver of the ADE singularity $FR$.  As in \ref{knitting demonstration}, the calculation for $FU_i$ begins by placing a $1$ in the place of $FN_i$, and proceeds by counting to the left, using the usual knitting rules, and records the numbers in the circles whilst treating them as zero for the next step.  At the end of the calculation, we read off the $b_{i,j}$ by adding the numbers in the circled vertices.

On the other hand, the calculation for $FV_i$ begins by placing a $1$ in the place of $FN_i$, then proceeds by counting to the right.  In exactly the same way, we read off the $c_{i,j}$ by adding the numbers in the circled vertices.  Since the AR quiver of ADE surface singularities coincides with the McKay quiver \cite{Aus3}, which is symmetric, we can obtain one calculation from the other by reflecting in the line through the original boxed vertex.  Thus both calculations return the same numbers, so $c_{i,j}=b_{i,j}$ for all $i\in I$ and $j\notin I$.
\end{proof}

\begin{cor}\label{see the flop}
With the $d=3$ crepant setup $f\colon X\to\Spec R$ of \ref{crepant setup}, set $N:=H^0(\cV_X)$ and $\Lambda:=\End_X(\cV_X)\cong\End_R(N)$.  Suppose further that either
\begin{enumerate}
\item[(A)] $f\colon X\to\Spec R$ is a minimal model, or
\item[(B)] $f\colon X\to\Spec R$ is a flopping contraction.
\end{enumerate}
Then, for any $i$, the region
\[
\upvartheta_i<0,\quad\upvartheta_j+b_j\upvartheta_i>0\mbox{ for all }j\neq i
\]
defines a chamber in $\Uptheta(\Lambda)$, and for any parameter $\upvartheta$ inside this chamber,
\[
\cM_{\rk, \upvartheta}(\Lambda)\cong\left\{  \begin{array}{ll}
X_i^+&\mbox{if $C_i$ flops}\\
X&\mbox{else,}
\end{array}\right.
\]
where $X_i^+$ denotes the flop of $X$ at $C_i$.  Thus the flop of $C_i$, if it exists, is obtained by crashing through the single wall $\upvartheta_i=0$ in $\Uptheta(\Lambda)$.
\end{cor}
\begin{proof}
Pick a curve $C_i$ (i.e.\ consider $I=\{i\}$), and mutate at the indecomposable summand $N_i$ of $N$.  By \ref{Cplus is chamber}, $\upvartheta_i=0$ is a wall.  Since we are mutating only at indecomposable summands, in situation (A) \ref{basic2} shows that the assumptions of \ref{main stab track} are satisfied.  In situation (B), \ref{mutmutI general} together with \ref{contract on f} shows that the assumptions of \ref{main stab track} are satisfied.  Thus, in either case, since $U_i\cong V_i$ by \ref{Ui=Vi for cDV},  provided that $\upvartheta_i\neq 0$ it is possible to track moduli using \ref{stab change for b=c}.

In either (A) or (B), if $\dim_{\mathbb{C}}\Lambda_i<\infty$ then $C_i$ flops, in which case $\upnu_iN\cong H^0(\cV_{X_i^+})$ by \ref{MMmodule under flop cor 2}.  Thus $\cM_{\rk,\upphi}(\upnu_i\Lambda)\cong X_i^+$ for all $\upphi\in C_+(\upnu_i\Lambda)$ by \ref{mut moduli gives flop text}\eqref{mut moduli gives flop text 2}, so the result then follows by moduli tracking \ref{stab change for b=c}. The only remaining case is when $\dim_{\mathbb{C}}\Lambda_i=\infty$ in situation (A), but then $\upnu_iN\cong N$ by \ref{basic2} and so the result is obvious. 
\end{proof}

The main result of this subsection needs the following result, which may be of independent interest. The case when $Y$ is the minimal resolution is well known \cite{CS, Kr}.

\begin{thm}\label{surface intersection root}
Consider an ADE singularity $\Spec R$, let $\Lambda$ be the corresponding NCCR, and let $Y\to\Spec R$ be a partial crepant resolution.  Set $N:=H^0(\cV_Y)$ and $\Gamma:=\End_R(N)$.  Suppose that the minimal resolution $X\to\Spec R$ has curves $C_1,\hdots,C_n$, and after re-indexing if necessary $Y$ is obtained from $X$ by contracting the curves $C_{r+1},\hdots,C_n$. Then
\begin{enumerate}
\item\label{surface intersection root 1} The walls of $\Uptheta(\Gamma)$ are obtained by intersecting the subspace $L$ of $\Uptheta(\Lambda)$ spanned by $\upvartheta_1,\hdots,\upvartheta_r$ with the walls of $\Uptheta(\Lambda)$ that do not contain $L$. 
\item\label{surface intersection root 2} $\Uptheta(\Gamma)$ has a finite number of chambers, and the walls are given by a finite collection of hyperplanes containing the origin.  The co-ordinate hyperplanes $\upvartheta_i=0$ are included in this collection.
\item\label{surface intersection root 3} Considering iterated mutations at indecomposable summands, tracking the chamber $C_+$ on $\upnu_{i_1}\hdots\upnu_{i_t}(\Gamma)$ back to $\Uptheta(\Gamma)$ gives all the chambers of $\Uptheta(\Gamma)$.
\end{enumerate}
\end{thm}
\begin{proof}
With the ordering of the curves as in the statement, we first contract $C_n$, then $C_{n-1}$, and continue to obtain a chain of crepant morphisms
\[
X\xrightarrow{f_1} X_{n-1}\xrightarrow{f_2}\hdots\to Y.
\]
The intersection in \eqref{surface intersection root 1} can be calculated inductively, so we first establish the result is true for  $\Lambda_{n-1}:=\End_{X_{n-1}}(\cV_{X_{n-1}})\cong(1-e_n)\Lambda(1-e_n)$.  

As notation, $\Uptheta(\Lambda)$ has coordinates $\upvartheta_1,\hdots,\upvartheta_{n}$, and we let $S$ be the subspace spanned by $\upvartheta_1,\hdots,\upvartheta_{n-1}$.  By abuse of notation, we let $\upvartheta_1,\hdots,\upvartheta_{n-1}$ also denote the coordinates of $\Uptheta(\Lambda_{n-1})$, so that we identify $\Uptheta(\Lambda_{n-1})$ with $S$.  We let $W_S$ be the set of walls of $\Uptheta(\Lambda)$ not containing $S$, then the intersection $S\cap W_S$ partitions $S$ into a finite number of regions.  We claim that these are precisely the chambers of $\Uptheta(\Lambda_{n-1})$.

First, since by \ref{Cplus is chamber} $\{\upvartheta\in\Uptheta(\Lambda)\mid \upvartheta_i>0 \mbox{ for all }1\leq i\leq n\}$ is a chamber of $\Uptheta(\Lambda)$, certainly no walls of $ \Uptheta(\Lambda)$ intersect $\{\upvartheta\in S\mid \upvartheta_i>0 \mbox{ for all }1\leq i\leq n-1\}$.  Thus we may identify this region of $S\backslash(S\cap W_S)$ with $C_+$ in $\Uptheta(\Lambda_{n-1})$.

Next, on $\Lambda$ we mutate the summand $N_1\oplus N_n$.  By \ref{stab change for b=c}, tracking $C_+$ from $\upnu_{\{1,n\}}\Lambda$ to $\Lambda$ using the formula in \ref{bi defin post ref} gives the chamber 
\begin{eqnarray}
\upvartheta_1<0,\quad\upvartheta_n<0,\quad\upvartheta_i+b_i\upvartheta_1+a_i\upvartheta_n>0\mbox{ for all }i\notin\{1,n\}\label{cham 1}
\end{eqnarray}
of $\Uptheta(\Lambda)$, where the $b_i$ are obtained from an $\add(R\oplus N_2\oplus\hdots\oplus N_{n-1})$-approximation of $N_1$, and the  $a_i$ are obtained from an $\add(R\oplus N_2\oplus\hdots\oplus N_{n-1})$-approximation of $N_n$.   On the other hand, using the approximation of $N_1$ above, by \ref{stab change for b=c} tracking $C_+$ from $\upnu_1\Lambda_{n-1}$ to $\Lambda_{n-1}$ using the formula in \ref{bi defin post ref} gives the chamber 
\begin{eqnarray}
\upvartheta_1<0,\quad\upvartheta_i+b_i\upvartheta_1>0\mbox{ for all }i\neq 1\label{cham 2}
\end{eqnarray}
of $\Uptheta(\Lambda_{n-1})$.  We already know the $\upvartheta_1=0$ edge of \eqref{cham 2} is a wall, and since the other edge walls of $C_+$ on $\upnu_1\Lambda_{n-1}$ can also be tracked by \ref{main stab track}\eqref{main stab track 4} to give strictly semi-stable points, the walls bounding \eqref{cham 2} are precisely the intersection of the walls bounding \eqref{cham 1} with $S$ (just set $\upvartheta_n=0$).  Since we know that the walls of $\Uptheta(\Lambda)$ are a hyperplane arrangement of planes through the origin \cite{CS, Kr}, this implies that the walls of the chamber \eqref{cham 2} of $\Uptheta(\Lambda_{n-1})$ are given by intersecting $S$ with all members of $W_S$.  There is nothing special about $\upvartheta_1$, so by symmetry all the walls of all the chambers bordering $C_+$ in $\Uptheta(\Lambda_{n-1})$ are given by intersecting $S$ with the elements of  $W_S$.

The proof then proceeds by induction.  By applying the argument above, tracking $C_+$ from $\upnu_{\{2,n\}}\upnu_{\{1,n\}}\Lambda$ to $\upnu_{\{1,n\}}\Lambda$, implies that tracking $C_+$ from $\upnu_2\upnu_1\Lambda_{n-1}$ to $\upnu_1\Lambda_{n-1}$ gives a chamber in $\Uptheta(\upnu_1\Lambda_{n-1})$, adjacent to $C_+$, cut out by intersecting walls from $\Uptheta(\Lambda)$.  In particular, the plane $x_1=0$ does not cut through this chamber, so by  \ref{stab change for b=c} we can track the full chamber all the way back to $\Uptheta(\Lambda_{n-1})$ to obtain a chamber adjacent to \eqref{cham 2}.  Again, the same argument shows that its walls are given by intersecting $S$ with the elements of  $W_S$.  By symmetry,  all the walls of all the chambers bordering all the chambers that border $C_+$ in $\Uptheta(\Lambda_{n-1})$ are given by intersecting $S$ with the elements of  $W_S$. 

Since $\Uptheta(\Lambda)$ has finitely many walls \cite{CS, Kr}, so does $S\cap W_S$, so continuing the above process all the walls of $\Uptheta(\Lambda_{n-1})$ are given by intersecting $S$ with the elements of  $W_S$, and each region is the tracking of $C_+$ under iterated mutation.  This proves \eqref{surface intersection root 1}, \eqref{surface intersection root 2} and \eqref{surface intersection root 3} for $X_{n-1}$. 

Next, consider $f_2\colon X_{n-1}\to X_{n-2}$.  Since by above $\Uptheta(\Lambda_{n-1})$, and all other $\Lambda'_{n-1}$ obtained from $X$ by contracting only a single curve, have walls given by a finite collection of hyperplanes passing through the origin, the above argument can be repeated to $\Lambda_{n-2}\cong(1-e_{n-1})\Lambda_{n-1}(1-e_{n-1})$ to show that $\Uptheta(\Lambda_{n-2})$ can be obtained from $\Uptheta(\Lambda_{n-1})$ by intersecting (and thus from $\Uptheta(\Lambda)$ by intersecting), and each region is the tracking of $C_+$ under iterated mutation.  By induction, parts \eqref{surface intersection root 1}, \eqref{surface intersection root 2} and \eqref{surface intersection root 3} follow.    
\end{proof}

The following is the main result of this subsection.

\begin{cor}\label{chamber 3fold=chamber surface}
With the $d=3$ crepant setup $f\colon X\to\Spec R$ of \ref{crepant setup}, set $N:=H^0(\cV_X)$ and $\Lambda:=\End_X(\cV_X)\cong\End_R(N)$.  Suppose further that either
\begin{enumerate}
\item[(A)] $f\colon X\to\Spec R$ is a minimal model, or
\item[(B)] $f\colon X\to\Spec R$ is a flopping contraction.
\end{enumerate}
Then for sufficiently generic $g$,
\begin{enumerate}
\item\label{chamber 3fold=chamber surface 1} The chamber structure of $\Uptheta(\Lambda)$ is the same as the chamber structure of $\Uptheta(\Lambda/g\Lambda)$.
\item\label{chamber 3fold=chamber surface 2}  $\Uptheta(\Lambda)$ has a finite number of chambers, and the walls are given by a finite collection of hyperplanes containing the origin.  The co-ordinate hyperplanes $\upvartheta_i=0$ are included in this collection.
\item\label{chamber 3fold=chamber surface 3}
Considering iterated mutations at indecomposable summands, tracking the chamber $C_+$ on $\upnu_{i_1}\hdots\upnu_{i_t}\Lambda$ back to $\Uptheta(\Lambda)$ gives all the chambers of $\Uptheta(\Lambda)$.
\end{enumerate}
\end{cor}
\begin{proof}
By \ref{chamber 3fold=chamber surface prep} the combinatorics of tracking $C_+$ are the same for the surface $R/gR$ as they are for the $3$-fold $R$, so all parts follow immediately from \ref{surface intersection root}.
\end{proof}

\subsection{Surfaces Chamber Structure via AR theory}\label{calculate chamber subsection} Having in \ref{surface intersection root} and \ref{chamber 3fold=chamber surface} reduced the problem to tracking the chamber $C_+$ under iterated mutation for partial crepant resolutions of Kleinian singularities, in this subsection we illustrate the combinatorics in two examples, summarising others in \S\ref{GIT examples section}, and give some applications.  

The intersection in \ref{surface intersection root}\eqref{surface intersection root 1} is in practice very cumbersome to calculate, since the full root systems are very large and contain much redundant  information.  In addition to giving an easy, direct way of calculating the chamber structure, the benefit of working with mutation is that we also obtain, in \ref{minimal bound},  a lower bound for the number of minimal models on the $3$-fold.

\begin{example}\label{knitting example}
Let $S$ be the $E_7$ surface singularity, and consider the partial resolution $Y\to\Spec S$ depicted by
\begin{eqnarray}
\begin{array}{c}
\begin{tikzpicture}
\node (-1) at (-0.75,0) [DB] {};
\node (0) at (0,0) [DB] {};
\node (1) at (0.75,0) [DW] {};
\node (1b) at (0.75,0.75) [DB] {};
\node (2) at (1.5,0) [DB] {};
\node (3) at (2.25,0) [DW] {};
\node (4) at (3,0) [DB] {};
\draw [-] (-1) -- (0);
\draw [-] (0) -- (1);
\draw [-] (1) -- (2);
\draw [-] (2) -- (3);
\draw [-] (3) -- (4);
\draw [-] (1) -- (1b);
\node at (-0.75,-0.3) {$\scriptstyle B_1$};
\node at (0,-0.3) {$\scriptstyle C_1$};
\node at (0.75,-0.3) {$\scriptstyle D$};
\node at (1.1,0.75) {$\scriptstyle B_3$};
\node at (1.5,-0.3) {$\scriptstyle C_2$};
\node at (2.25,-0.3) {$\scriptstyle B_2$};
\node at (3,-0.3) {$\scriptstyle A_2$};
\end{tikzpicture}
\end{array}\label{E7 config example}
\end{eqnarray}
where the vertices have been labelled by their corresponding CM $S$-modules.  We calculate the chamber structure of $\End_S(S\oplus B_2\oplus D)$.  The AR quiver for $\CM S$ is
\[
\begin{array}{c}
\begin{tikzpicture}[>=stealth,xscale=1.6,yscale=1.5]
\node (R0) at (0.5,0.5) {$\scriptstyle S$};
\node (A0) at (0,0)  {$\scriptstyle B_1$};
\node (A1) at (1,0)  {$\scriptstyle B_1$};
\node (B0) at (0.5,-0.5) {$\scriptstyle C_1$};
\node (C0) at (0,-1)  {$\scriptstyle D$};
\node (C1) at (0.5,-1)  {$\scriptstyle B_3$};
\node (C2) at (1,-1)   {$\scriptstyle D$};
\node (D0) at (0.5,-1.5)  {$\scriptstyle C_2$};
\node (E0) at (0,-2)  {$\scriptstyle B_2$};
\node (E1) at (1,-2)  {$\scriptstyle B_2$};
\node (F0) at (0.5,-2.5)  {$\scriptstyle A_2$};
\draw[->] (A0)--(R0);
\draw[->] (R0)--(A1);
\draw[->] (A0)--(B0);
\draw[->] (B0)--(A1);
\draw[->] (C0)--(B0);
\draw[->] (B0)--(C2);
\draw[->] (C0)--(C1);
\draw[->] (C1)--(C2);
\draw[->] (C0)--(D0);
\draw[->] (D0)--(C2);
\draw[->] (E0)--(D0);
\draw[->] (D0)--(E1);
\draw[->] (E0)--(F0);
\draw[->] (F0)--(E1);
\end{tikzpicture}
\end{array}
\]
where the left and right hand sides are identified.  Via knitting, as in \ref{knitting demonstration},
\[
\begin{array}{ccc}
\begin{array}{c}
\begin{tikzpicture}[xscale=0.85,yscale=0.85]
\node (R0) at (0.5,0.5) {$0$};
\node (R1) at (1.5,0.5) {$0$};
\node (A0) at (0,0) [vertex] {};
\node (A1) at (1,0) {$0$};
\node (A2) at (2,0) [vertex] {};
\node (B0) at (0.5,-0.5) {$0$};
\node (B1) at (1.5,-0.5) {$0$};
\node (C0) at (0,-1)  [vertex] {};
\node (C1) at (0.5,-1) {$0$};
\node (C2) at (1,-1)  {$1$};
\node (C3) at (1.5,-1) {$0$};
\node (C4) at (2,-1)  [vertex] {};
\node (D0) at (0.5,-1.5) {$0$};
\node (D1) at (1.5,-1.5) {$1$};
\node (E0) at (0,-2)  {$-1$};
\node (E1) at (1,-2)  {$1$};
\node (E2) at (2,-2)  {$1$};
\node (F0) at (0.5,-2.5)  {$0$};
\node (F1) at (1.5,-2.5)  {$1$};
\draw (1,-1) circle (7.5pt);
\draw (0.5,0.5) circle (7.5pt);
\draw (1.5,0.5) circle (7.5pt);
\node[draw,regular polygon, regular polygon sides=4, minimum size=0.6cm] at (2,-2) {};
\end{tikzpicture}
\end{array}
&
&
\begin{array}{c}
\begin{tikzpicture}[xscale=0.85,yscale=0.85]
\node (R8) at (8.5,0.5) {$0$};
\node (R9) at (9.5,0.5) {$1$};
\node (R10) at (10.5,0.5) {$1$};
\node (R11) at (11.5,0.5) {$0$};
\node (A8) at (8,0) [vertex] {};
\node (A9) at (9,0) {$1$};
\node (A10) at (10,0) {$1$};
\node (A11) at (11,0) {$0$};
\node (A12) at (12,0) [vertex] {};
\node (B8) at (8.5,-0.5) {$0$};
\node (B9) at (9.5,-0.5) {$1$};
\node (B10) at (10.5,-0.5) {$2$};
\node (B11) at (11.5,-0.5) {$1$};
\node (C16) at (8,-1)  {$-1$};
\node (C17) at (8.5,-1)  {$0$};
\node (C18) at (9,-1)  {$1$};
\node (C19) at (9.5,-1)  {$1$};
\node (C20) at (10,-1)  {$2$};
\node (C21) at (10.5,-1)  {$1$};
\node (C22) at (11,-1)  {$2$};
\node (C23) at (11.5,-1)  {$1$};
\node (C24) at (12,-1)   {$1$};
\node (D8) at (8.5,-1.5)  {$0$};
\node (D9) at (9.5,-1.5)  {$1$};
\node (D10) at (10.5,-1.5)  {$1$};
\node (D11) at (11.5,-1.5)  {$1$};
\node (E8) at (8,-2)  [vertex] {};
\node (E9) at (9,-2) {$1$};
\node (E10) at (10,-2)  {$1$};
\node (E11) at (11,-2)  {$1$};
\node (E12) at (12,-2)  [vertex] {};
\node (F8) at (8.5,-2.5)  {$0$};
\node (F9) at (9.5,-2.5)  {$0$};
\node (F10) at (10.5,-2.5)  {$0$};
\node (F11) at (11.5,-2.5)  {$0$};
\draw (11,-2) circle (7.5pt);
\draw (10,-2) circle (7.5pt);
\draw (9,-2) circle (7.5pt);
\draw (8.5,0.5) circle (7.5pt);
\draw (9.5,0.5) circle (7.5pt);
\draw (10.5,0.5) circle (7.5pt);
\draw (11.5,0.5) circle (7.5pt);
\node[draw,regular polygon, regular polygon sides=4, minimum size=0.6cm] at (12,-1) {};
\end{tikzpicture}
\end{array}
\end{array}
\]
we obtain the exchange sequences
\begin{eqnarray}
0\to B_2\to D\to B_2\to 0\label{mut 1 stab}\\
0\to D\to R^{\oplus 2}\oplus B_2^{\oplus 3}\to D\to 0\label{mut 2 stab}
\end{eqnarray}
Thus in this example, the dual graph does not change under mutation.  Now fix the ordering of the curves  
\[
\begin{array}{c}
\begin{tikzpicture}
\node (-1) at (-0.75,0) [DB] {};
\node (0) at (0,0) [DB] {};
\node (1) at (0.75,0) [DW] {};
\node (1b) at (0.75,0.75) [DB] {};
\node (2) at (1.5,0) [DB] {};
\node (3) at (2.25,0) [DW] {};
\node (4) at (3,0) [DB] {};
\draw [-] (-1) -- (0);
\draw [-] (0) -- (1);
\draw [-] (1) -- (2);
\draw [-] (2) -- (3);
\draw [-] (3) -- (4);
\draw [-] (1) -- (1b);
\node at (0.75,-0.3) {$\scriptstyle 2$};
\node at (2.25,-0.3) {$\scriptstyle 1$};
\end{tikzpicture}
\end{array}
\]
First, we track the $C_+$ chamber from $\upnu_1\Lambda$ to $\Lambda$.  By \ref{stab change for b=c}, 
\[
\begin{array}{c}
\upphi_1\\
\upphi_2
\end{array}
\stackrel{\scriptsize\mbox{\eqref{mut 1 stab}}}{\mapsto} 
\begin{array}{c}
-\upphi_1\\
\upphi_1+\upphi_2
\end{array}
\]
and so the $C_+$ chamber from $\upnu_1\Lambda$ maps to  the region $\upvartheta_1<0$ and $\upvartheta_1+\upvartheta_2>0$ of $\Lambda$, and thus this is a chamber for $\Lambda$.  Next, we track the $C_+$ chamber from $\upnu_2\upnu_1\Lambda$ to $\upnu_1\Lambda$ to $\Lambda$.  By the same logic   
\[
\begin{array}{c}
\upphi_1\\
\upphi_2
\end{array}
\stackrel{\scriptsize\mbox{\eqref{mut 2 stab}}}{\mapsto} 
\begin{array}{c}
\upphi_1+3\upphi_2\\
-\upphi_2
\end{array}
\stackrel{\scriptsize\mbox{\eqref{mut 1 stab}}}{\mapsto} 
\begin{array}{c}
-(\upphi_1+3\upphi_2)\\
-\upphi_2+(\upphi_1+3\upphi_2)
\end{array}
=
\begin{array}{r}
-\upphi_1-3\upphi_2\\
\upphi_1+2\upphi_2
\end{array}
\]
which is precisely the region $\upvartheta_1+\upvartheta_2<0$
and $2\upvartheta_1+3\upvartheta_2>0$ of $\Lambda$, and so this too is a chamber.  Continuing in this fashion, we obtain the chamber structure illustrated in Figure~\ref{E7 config example 2}, where there are 12 chambers in total.
\begin{figure}[H]
\[
\begin{array}{ccc}
\begin{array}{c}
\begin{tikzpicture}[scale=0.75]
\coordinate (A1) at (135:2cm);
\coordinate (A2) at (-45:2cm);
\coordinate (B1) at (-33.690:2cm);
\coordinate (B2) at (146.31:2cm);
\coordinate (C1) at (153.435:2cm);
\coordinate (C2) at (-26.565:2cm);
\coordinate (D1) at (161.565:2cm);
\coordinate (D2) at (-18.435:2cm);
\draw[red] (A1) -- (A2);
\draw[blue] (B1) -- (B2);
\draw[green] (C1) -- (C2);
\draw[purple] (D1) -- (D2);
\draw[->] (-2,0)--(2,0);
\node at (2.5,0) {$\upvartheta_1$};
\draw[->] (0,-2)--(0,2);
\node at (0.5,2.25) {$\upvartheta_2$};
\draw[densely dotted,gray] (0,0) circle (2cm);
\end{tikzpicture}
\end{array}
&&
\begin{array}{cl}
&\upvartheta_1=0\\
&\upvartheta_2=0\\
\begin{array}{c}\begin{tikzpicture}\node at (0,0){}; \draw[red] (0,0)--(0.5,0);\end{tikzpicture}\end{array}&\upvartheta_1+\upvartheta_2=0\\
\begin{array}{c}\begin{tikzpicture}\node at (0,0){}; \draw[blue] (0,0)--(0.5,0);\end{tikzpicture}\end{array}&2\upvartheta_1+3\upvartheta_2=0\\
\begin{array}{c}\begin{tikzpicture}\node at (0,0){}; \draw[green] (0,0)--(0.5,0);\end{tikzpicture}\end{array}&\upvartheta_1+2\upvartheta_2=0\\
\begin{array}{c}\begin{tikzpicture}\node at (0,0){}; \draw[purple] (0,0)--(0.5,0);\end{tikzpicture}\end{array}&\upvartheta_1+3\upvartheta_2=0
\end{array}
\end{array}
\]
\caption{Chamber structure for $E_7$ with configuration \eqref{E7 config example}.}\label{E7 config example 2}
\end{figure}
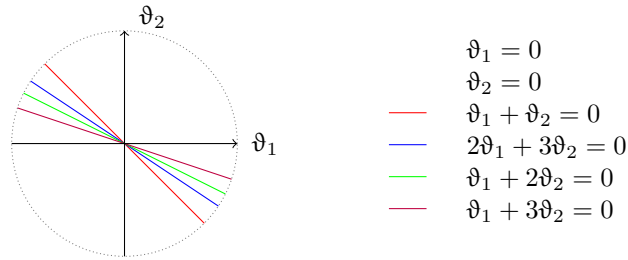
\end{example}

However, often the dual graph does change under mutation.
\begin{example}\label{knitting example 2}
Let $S$ be the $D_4$ surface singularity, and consider the partial resolution depicted by
\begin{eqnarray}
\begin{array}{c}
\begin{tikzpicture}
\node (0) at (0,0) [DW] {};
\node (1) at (0.75,0) [DB] {};
\node (1b) at (0.75,0.75) [DW] {};
\node (2) at (1.5,0) [DW] {};
\draw [-] (0) -- (1);
\draw [-] (1) -- (2);
\draw [-] (1) -- (1b);
\node at (0,-0.3) {$\scriptstyle A_1$};
\node at (0.75,-0.3) {$\scriptstyle M$};
\node at (1.1,0.75) {$\scriptstyle A_3$};
\node at (1.5,-0.3) {$\scriptstyle A_2$};
\end{tikzpicture}
\end{array}\label{D4 3 curve text figure}
\end{eqnarray}
Via knitting, to mutate at $A_3$ the relevant exchange sequence is
\begin{eqnarray}
0\to M\to R\oplus A_1\oplus A_2\to A_3\to 0\label{mut stab 3}
\end{eqnarray}
Hence $\upnu_3\Lambda=\End_R(R\oplus A_1\oplus A_2\oplus M)$, and this corresponds to the partial resolution 
\begin{eqnarray}
\begin{array}{c}
\begin{tikzpicture}
\node (0) at (0,0) [DW] {};
\node (1) at (0.75,0) [DW] {};
\node (1b) at (0.75,0.75) [DB] {};
\node (2) at (1.5,0) [DW] {};
\draw [-] (0) -- (1);
\draw [-] (1) -- (2);
\draw [-] (1) -- (1b);
\end{tikzpicture}
\end{array}\label{D4 3 curve text figure 2}
\end{eqnarray}
Hence under mutation, the dual graph changes from \eqref{D4 3 curve text figure} to \eqref{D4 3 curve text figure 2}.  Tracking the chamber $C_+$ from $\upnu_3\Lambda$ to $\Lambda$, by \ref{stab change for b=c}
\[
\begin{array}{c}
\upphi_1\\
\upphi_2\\
\upphi_3
\end{array}
\stackrel{\scriptsize\mbox{\eqref{mut stab 3}}}{\mapsto} 
\begin{array}{c}
\upphi_1+\upphi_3\\
\upphi_2+\upphi_3\\
-\upphi_3
\end{array}
\]
which is the region $\upvartheta_3<0$, $\upvartheta_1+\upvartheta_3>0$, $\upvartheta_2+\upvartheta_3>0$.  Hence this is a chamber for $\Lambda$, and it corresponds to a difference curve configuration.  Note that since $\upnu_3N\ncong N$ on the surface, by combining \ref{MMmodule under flop cor} and \ref{chamber 3fold=chamber surface prep} any cDV singularity with minimal model that cuts under generic hyperplane section to \eqref{D4 3 curve text figure} must flop when crossing the wall $\upvartheta_3=0$.  By symmetry, it must also flop when crossing the walls $\upvartheta_1=0$ and $\upvartheta_2=0$. This shows that any such $3$-fold must have at least four minimal models, since we can flop three different curves.  In this example,  the full chamber structure coincides (after a change in parameters) with the chamber structure in Figure~\ref{cD4 121}, so there are 32 chambers in total.
\end{example}

The following is an extension of the above observation.  In each chamber of $\Uptheta(\Lambda/g\Lambda)$, we draw the curve configuration appearing on the surface mutation calculation, as calculated in the above example.  We refer to this as the \emph{enhanced chamber structure} of $\Uptheta(\Lambda/g\Lambda)$.  See \ref{enriched E7 example} for an example.

\begin{lemma}\label{minimal bound}
Suppose that $R$ is a $cDV$ singularity, with a minimal model $X\to\Spec R$.  Set $\Lambda:=\End_R(N)$, where $N:=H^0(\cV_X)$.  Calculating $\Uptheta(\Lambda)$ by passing to a general hyperplane section $g$, the number of minimal models of $\Spec R$ is bounded below by the number of different curve configurations obtained in the enhanced chamber structure of $\Uptheta(\Lambda/g\Lambda)$.  
\end{lemma}
\begin{proof}
Certainly if two minimal models $X$ and $Y$ cut under generic hyperplane section to two different curve configurations, then $X$ and $Y$ must be different minimal models. Thus it suffices to show that every curve configuration in the enhanced chamber structure of $\Uptheta(\Lambda/g\Lambda)$ does actually appear as the cut of some minimal model.  By \ref{surface intersection root}\eqref{surface intersection root 3} and \ref{chamber 3fold=chamber surface}\eqref{chamber 3fold=chamber surface 3} it is possible to reach any such configuration starting at $C_+$ by mutating at indecomposable summands.  Since by \ref{MMmodule under flop cor} at each wall crossing  either the moduli stays the same, or some curve flops, each chamber in $\Uptheta(\Lambda)$ gives a minimal model of $\Spec R$.  Hence if a curve configuration is in a chamber $D$ of $\Uptheta(\Lambda/g\Lambda)$, consider the minimal model given by the chamber $D$ of $\Uptheta(\Lambda)$.  This minimal model cuts to the desired curve configuration, by \ref{chamber 3fold=chamber surface prep}.
\end{proof}

\section{First Applications}

\subsection{The Craw--Ishii Conjecture}\label{CI subsection}
Combining moduli tracking from \S\ref{stab and mut section} with the Homological MMP in Figure~\ref{Fig2} leads immediately to a proof of the Craw--Ishii conjecture for cDV singularities. To prove a slightly more precise version,  the following terminology will be convenient.

\begin{defin}\label{skeleton definition}
The \emph{skeleton} of the GIT chamber decomposition of $\Uptheta$ is defined to be the graph obtained by placing a vertex in every chamber, and two vertices are connected by an edge if and only if the associated chambers share a codimension one wall.
\end{defin}
The following is the main result of this subsection.

\begin{thm}\label{Craw Ishii true} With the $d=3$ crepant setup of \ref{crepant setup}, assume further that $X$ is $\mathds{Q}$-factorial.  Set $N=H^0(\cV_X)$ and $\Lambda=\End_R(N)$, then
\begin{enumerate}
\item\label{Craw Ishii true 1} In the skeleton of $\Uptheta(\Lambda)$, there exists a connected path, containing the chamber $C_+$, where every minimal model can be found.  Furthermore, each wall crossing in this path corresponds to a flop. 
\item\label{Craw Ishii true 2} The Craw--Ishii Conjecture \ref{CI conj} is true for cDV singularities, namely for any fixed MMA $\Gamma:=\End_R(M)$ with $R\in\add M$, every projective minimal model can be obtained as the moduli space of $\Gamma$ for some stability parameter $\upvartheta$.
\end{enumerate}
\end{thm}
\begin{proof}
(1) We run Figure~\ref{Fig2} whilst picking only single curves satisfying $\dim_{\mathbb{C}}\Lambda_i<\infty$.  As in \ref{NCvsC min model}, this produces all minimal models.  By \ref{mut moduli gives flop text} we can view all these minimal models as the $C_+ $ moduli on their corresponding algebra $\upnu_{i_1}\hdots\upnu_{i_t}\Lambda$.  By \ref{chamber 3fold=chamber surface}\eqref{chamber 3fold=chamber surface 3} it is possible to track all these back to give chambers in $\Uptheta(\Lambda)$, and the proof of \ref{surface intersection root} shows that this combinatorial tracking gives a connected path. The fact that each wall gives a flop is identical to \ref{see the flop}, since at each stage $\dim_{\mathbb{C}}\Lambda_i<\infty$.\\
(2) Consider an MMA $\Gamma:=\End_R(M)$ as in the statement. By the Auslander--McKay correspondence \ref{NCvsC min model} $M\cong H^0(\cV_Y)$ for some minimal model $Y\to\Spec R$.  The result then follows by applying \eqref{Craw Ishii true 1} to $Y\to\Spec R$.
\end{proof}

We remark that flops of multiple curves can also be easily described.  The following is the multi-curve version of \ref{see the flop}.
\begin{lemma}
With the $d=3$ crepant setup of \ref{crepant setup}, set $N=H^0(\cV_X)$, $\Lambda=\End_R(N)$, and pick a subset of curves $I$ above the origin.  If  $\dim_{\mathbb{C}}\Lambda_I<\infty$, then $\cM_{\rk,\,\upnu_{\boldb}\upvartheta}(\Lambda)$ is the flop of $\bigcup_{i\in I}C_i$, for all $\upvartheta\in C_+(\Lambda)$
\end{lemma}
\begin{proof}
This follows by combining \ref{mut moduli gives flop text} with moduli tracking \ref{stab change for b=c}, using \ref{Ui=Vi for cDV}   to establish that $\boldb=\boldc$.
\end{proof}

\subsection{Auslander--McKay Revisited}\label{AM section 2}
In the $d=3$ crepant setup of \ref{crepant setup}, in this subsection we use the extra information of the GIT chamber decomposition of $\Uptheta(\Lambda)$ from \S\ref{stab and mut section}  to extend aspects of the Auslander--McKay Correspondence from \S\ref{AM subsection 1}.

\begin{thm}\label{NCvsC min model full}
The correspondence in \ref{NCvsC min model} and \ref{new thm in email} further satisfies
\begin{enumerate}
\item[(5)] The simple mutation graph of MM generators can be viewed as a subgraph of the skeleton of the GIT chamber decomposition of $\Uptheta(\Lambda)$.
\item[(6)] The number of MM generators is bounded above by the number of chambers in the GIT chamber decomposition of any of the MMAs, and is bounded below by the number of different curve configurations obtained in the enhanced chamber structure of $\Uptheta(\Lambda/g\Lambda)$.
\end{enumerate}
\end{thm}
\begin{proof}
Part (5) follows from \ref{Craw Ishii true}\eqref{Craw Ishii true 1} together with \ref{NCvsC min model}\eqref{NCvsC min model 3}.  The first part of (6) is then obvious, and the second half is \ref{minimal bound}.
\end{proof}

\subsection{Root Systems}
We observed in \ref{knitting example} that the chamber structure of partial resolutions of Kleinian singularities, and thus by \ref{chamber 3fold=chamber surface} also the corresponding cDV singularities, cannot in general be identified with the root system of a semisimple Lie algebra.  In special cases, however, they can. 

\begin{lemma}\label{chamber general A surfaces}
With the crepant setup $f\colon X\to \Spec R$ of \ref{crepant setup}, suppose that $d=2$ and $R$ is a type $A$ Kleinian singularity. Set $\Lambda:=\End_X(\cV_X)\cong\End_R(N)$.  If there are $t$ curves above the unique closed point,  then the chamber structure for $\Uptheta(\Lambda)$ can be identified with the root system of $\mathfrak{sl}_t$.
\end{lemma}
\begin{proof}
Label the CM $R$-modules corresponding to the curves in the minimal resolution by $N_1,\hdots, N_n$, from left to right.  Since $X$ is dominated by the minimal resolution, it is obtained by contracting curves, to leave CM modules $N_{j_1},\hdots,N_{j_t}$ say, so that $N=H^0(\cV_X)=N_{j_1}\oplus\hdots\oplus N_{j_t}$.

By \ref{surface intersection root}\eqref{surface intersection root 3},  the chamber structure can be calculated by tracking $C_+$ back through iterated mutation at indecomposable summands.  The AR quiver is 
\[
\begin{tikzpicture}[scale=0.75,>=stealth]
\node (A1) at (0,0) {$\scriptstyle i$};
\node (A2) at (1,0) {$\scriptstyle i+1$};
\node (A3) at (2,0) {$\scriptstyle \phantom i$};
\node (B1) at (0,-1) {$\scriptstyle i-1$};
\node (B2) at (1,-1) {$\scriptstyle i$};
\node (B3) at (2,-1) {$\scriptstyle \phantom i$};
\node (C1) at (0,-2) {$\scriptstyle \phantom i$};
\node (C2) at (1,-2) {$\scriptstyle \phantom i$};
\node (C3) at (2,-2) {$\scriptstyle \phantom i$};
\draw[->] (A1)--(A2);
\draw[->] (A2)--(A3);
\draw[->] (B1)--(B2);
\draw[->] (B2)--(B3);
\draw[->] (A1)--(B1);
\draw[->] (A2)--(B2);
\draw[->] (B1)--(C1);
\draw[->] (B2)--(C2);
\end{tikzpicture}
\]
Consider an indecomposable summand $N_{j_i}$ (i.e.\ $I=\{ j_i\}$), then to calculate its $\add N_{I^c}$-approximation, by knitting it is clear that 
\[
\begin{array}{ccc}
\begin{array}{c}
\begin{tikzpicture}[scale=1,>=stealth]
\node (A1) at (0,0) {$\scriptstyle j_i$};
\node (A2) at (1,0) {$\scriptstyle j_i+1$};
\node (A3) at (2,0) {$\scriptstyle j_i+2$};
\node (A4) at (3,0) {$\scriptstyle \phantom i$};
\node (A5) at (4,0) {$\scriptstyle j_{i+1}$};
\node (B1) at (0,-1) {$\scriptstyle j_i-1$};
\node (B2) at (1,-1) {$\scriptstyle j_i$};
\node (B3) at (2,-1) {$\scriptstyle\phantom i$};
\node (C1) at (0,-2) {$\scriptstyle \phantom i$};
\node (C2) at (1,-2) {$\scriptstyle \phantom i$};
\node (D1) at (0,-3) {$\scriptstyle j_{i-1}$};
\draw[->] (A1)--(A2);
\draw[->] (A2)--(A3);
\draw[->] (A3)--(A4);
\draw[->] (A4)--(A5);
\draw[->] (B1)--(B2);
\draw[->] (B2)--(B3);
\draw[->] (B2)--(C2);
\draw[->] (A1)--(B1);
\draw[->] (A2)--(B2);
\draw[->] (B1)--(C1);
\draw[->] (C1)--(D1);
\draw[->] (A3)--(A4);
\end{tikzpicture}
\end{array}
&&
\begin{array}{c}
\begin{tikzpicture}[scale=1,>=stealth]
\node (A1) at (0,0) {$\scriptstyle 1$};
\node (A2) at (1,0) {$\scriptstyle 1$};
\node (A3) at (2,0) {$\scriptstyle 1$};
\node (A4) at (3,0) {$\scriptstyle 1$};
\node (A5) at (4,0) {$\scriptstyle 1$};
\node (B1) at (0,-1) {$\scriptstyle 1$};
\node (B2) at (1,-1) {$\scriptstyle 1$};
\node (B3) at (2,-1) {$\scriptstyle 1$};
\node (B4) at (3,-1) {$\scriptstyle 1$};
\node (B5) at (4,-1) {$\scriptstyle 0$};
\node (C1) at (0,-2) {$\scriptstyle 1$};
\node (C2) at (1,-2) {$\scriptstyle 1$};
\node (C3) at (2,-2) {$\scriptstyle 1$};
\node (C4) at (3,-2) {$\scriptstyle 1$};
\node (C5) at (4,-2) {$\scriptstyle 0$};
\node (D1) at (0,-3) {$\scriptstyle 1$};
\node (D2) at (1,-3) {$\scriptstyle 0$};
\node (D3) at (2,-3) {$\scriptstyle 0$};
\node (D4) at (3,-3) {$\scriptstyle 0$};
\node (D5) at (4,-3) {$-\scriptstyle 1$};
\draw[->] (A1)--(A2);
\draw[->] (A2)--(A3);
\draw[->] (A3)--(A4);
\draw[->] (A4)--(A5);
\draw[->] (B1)--(B2);
\draw[->] (B2)--(B3);
\draw[->] (B2)--(C2);
\draw[->] (A1)--(B1);
\draw[->] (A2)--(B2);
\draw[->] (B1)--(C1);
\draw[->] (C1)--(D1);
\draw[->] (A3)--(A4);
\draw (0,-3) circle (4.5pt);
\draw (4,0) circle (4.5pt);
\node[draw,regular polygon, regular polygon sides=4, minimum size=0.57cm] at (0,0) {};
\end{tikzpicture}
\end{array}
\end{array}
\]
and so the combinatorics that determine the tracking negates $\upvartheta_{j_i}$ and adds $\upvartheta_{j_i}$ to each of its neighbours.   

On the other hand, if we consider the minimal resolution of the Type $A$ singularity  $\frac{1}{t+1}(1,-1)$, which also has $t$ curves above the origin, the combinatorics that governs tracking $C_+$ in this case is also the rule that negates $\upvartheta_j$ and adds $\upvartheta_j$ to its neighbours.  Hence since the chamber structure for the minimal resolution of $\frac{1}{t+1}(1,-1)$  can be identified with the root system of $\mathfrak{sl}_t$ \cite{CS, Kr}, so too can the chamber structure of $\Uptheta(\Lambda)$.
\end{proof}

By combining \ref{chamber 3fold=chamber surface} and \ref{chamber general A surfaces}, the following is immediate.
\begin{cor}\label{chamber general A}
With the $d=3$ crepant setup $f\colon X\to\Spec R$ of \ref{crepant setup}, set $N:=H^0(\cV_X)$ and $\Lambda:=\End_X(\cV_X)\cong\End_R(N)$.  Suppose further that either
\begin{enumerate}
\item[(A)] $f\colon X\to\Spec R$ is a minimal model, or
\item[(B)] $f\colon X\to\Spec R$ is a flopping contraction.
\end{enumerate}
where $R$ is a complete local $cA_n$ singularity.  Set $\Lambda:=\End_X(\cV_X)\cong\End_R(N)$.  If there are $t$ curves above the unique closed point,  then the chamber structure for $\Uptheta(\Lambda)$ can be identified with the root system of $\mathfrak{sl}_t$.
\end{cor}

There are also other cases in which root systems appear.  Consider the following assumption, made throughout in \cite{TodaResPub}.
\begin{setup}\label{Toda assumption}
In the $d=3$ crepant setup of \ref{crepant setup}, suppose that $R$ is isolated and there is a hyperplane section which cuts $X$ to give the minimal resolution.
\end{setup}

The setup is restrictive, for example in the case of a minimal model of $\Spec R$ with only one curve above the origin, it forces $R$ to be Type $A$.  Nevertheless, under the setup \ref{Toda assumption}, associated to $R$ is some ADE Dynkin diagram.  The following recovers \cite[\S5.1]{Toda}.

\begin{lemma}\label{Toda for AM}
With the assumption in \ref{Toda assumption}, 
\begin{enumerate}
\item The chamber structure of $\Uptheta(\Lambda)$ can be identified with the root system of the corresponding Dynkin diagram.
\item There are precisely $|W|$ chambers, where $W$ is the corresponding Weyl group.
\end{enumerate}
\end{lemma}
\begin{proof}
(1) Since $R$ is isolated $\Ext^1_R(N,N)=0$.  Thus the $\Ext^1_R(N,N)$-regular condition in \ref{mut preserved by generic} is redundant, so \ref{chamber 3fold=chamber surface prep} holds for the particular $g$ in \ref{Toda assumption}.  Appealing to this directly in the proof of \ref{chamber 3fold=chamber surface}  shows that the chamber structure of $\Uptheta(\Lambda)$ and $\Uptheta(\Lambda/g\Lambda)$ coincide.  Since by \ref{Toda assumption} the pullback of the hyperplane section is the full minimal resolution, it follows that $\Lambda/g\Lambda$ is the preprojective algebra of the corresponding Dynkin diagram, and its chamber structure is well-known \cite{CS, Kr}. Part (2) is immediate. 
\end{proof}

\subsection{Auslander--McKay for Isolated Singularities}\label{AM isolated subsection} With the crepant setup of \ref{crepant setup},  the case when $R$ is in addition an isolated singularity is particularly important for two reasons.  First, it aligns well with cluster theory, since in this setting $\uCM R$ is a Hom-finite $2$-CY category, with maximal rigid objects the MM generators, and cluster tilting objects (if they exist) the CT modules.  Second, the minimal models are easier to count, thus we have finer control over the mutation graph.

Recall that if $\cC$ is an exact category, then $M\in\cC$ is called \emph{rigid} if $\Ext^1_{\cC}(M,M)=0$, and  $M\in \cC$ is called \emph{maximal rigid} if $M$ is rigid and further it is maximal with respect to this property, namely if there exists $Y\in \cC$ such that $M\oplus Y$ is rigid, then $Y\in \add M$.  Equivalently, $M$ is a maximal rigid object of $\cC$ if
\[
\add M=\{  Y\in\cC\mid \Ext^1_{\cC}(M\oplus Y,M\oplus Y)=0\}.
\]
Also, recall that $M\in\cC$ is called a \emph{cluster tilting} object in $\cC$ if  
\[
\add M=\{ Y\in\cC \mid \Ext^{1}_{\cC}(M,Y)=0 \} =\{ X\in\cC \mid \Ext^{1}_{\cC}(X,M)=0 \}.
\]

\begin{cor}\label{isolated AM}
Let $R$ be a complete local isolated cDV singularity.  Then \ref{NCvsC min model} reduces to a one-to-one correspondence
\[
\begin{array}{c}
\begin{tikzpicture}
\node (A) at (-1,0) {$\{ \mbox{basic maximal rigid objects in $\CM R$}\}$};
\node (B) at (6.25,0) {$
\{\mbox{minimal models $f_i\colon X_i\to\Spec R$} \}.
$};
\draw[<->] (2.25,0) -- node [above] {$\scriptstyle $} (3.25,0);
\end{tikzpicture}
\end{array}
\]
If further the minimal models of $\Spec R$ are smooth (equivalently, $\CM R$ admits a cluster tilting object), then this reduces to
\[
\begin{array}{c}
\begin{tikzpicture}
\node (A) at (-1,0) {$\{ \mbox{basic cluster tilting objects in $\CM R$}\}$};
\node (B) at (6.25,0) {$
\{\mbox{crepant resolutions $f_i\colon X_i\to\Spec R$} \}.
$};
\draw[<->] (2.15,0) -- node [above] {$\scriptstyle $} (3.05,0);
\end{tikzpicture}
\end{array}
\]
In either case, under this correspondence properties \t{(1)--(4)} in \ref{NCvsC min model} and \ref{new thm in email} still hold, but further we now also have
\begin{enumerate}
\item[(5)]\label{NCvsC min model isolated 6} The simple mutation graph of the maximal rigid (respectively, cluster tilting) objects in $\uCM R$ is precisely the skeleton of the GIT chamber structure.
\item[(6)]\label{NCvsC min model isolated 5} The number of basic maximal rigid (respectively, cluster tilting) objects in $\uCM R$ is precisely the number of chambers in the GIT chamber decomposition. 
\end{enumerate}
If furthermore \ref{Toda assumption} is satisfied, then 
\begin{enumerate}
\item[(7)]\label{NCvsC min model isolated 7}  There are precisely $|W|$ maximal rigid objects in $\cC=\uCM R$, where $W$ is the corresponding Weyl group. 
\end{enumerate}
\end{cor}
\begin{proof}
Since $R$ is isolated, $M$ is a maximal rigid object in the category $\CM R$ if and only if $M$ is an MM generator \cite[5.12]{IW4}, so the first bijection is a special case of the bijection in \ref{NCvsC min model}.  
The second bijection is similar, using \cite[5.11]{IW4}.  Further, since $R$ is isolated, it follows that always $\dim_{\mathbb{C}}\Lambda_i<\infty$, so all curves flop, and all summands non-trivially mutate.  Thus (5) follows from \ref{NCvsC min model}\eqref{NCvsC min model 3}, using the argument of \ref{Craw Ishii true}\eqref{Craw Ishii true 1}.  Part (6) follows immediately from (5), and part (7) follows from (6) together with \ref{Toda for AM}.
\end{proof}

We refer the reader to \S\ref{GIT examples section} for examples of chamber structures and mutation graphs.  The following is a non-explicit proof of \cite[4.15]{BIKR}, extended from crepant resolutions to also cover minimal models.
\begin{cor}
Consider an isolated cDV singularity $R:=\mathbb{C}[[u,v,x,y]]/(uv-f_1\hdots f_n)$ where each $f_i\in \m:=(x,y)\subseteq \mathbb{C}[[x,y]]$.  Then there are precisely $n!$ maximal rigid objects in $\uCM R$, and all are connected by mutation.
\end{cor}
\begin{proof}
As in \cite[6.1(e)]{BIKR}, $R$ is a $cA_m$ singularity, where $m=\ord(f_1\hdots f_n)-1$, and it is well known (see e.g.\ the calculation in \cite[\S5.1]{IW5}) that the minimal models of $\Spec R$ have $n$ curves above the origin.   But by \ref{chamber general A} the GIT chamber decomposition of any of the MMAs $\End_R(M)$ with $R\in\add M$ has precisely $n!$ chambers, so the result follows from \ref{isolated AM}.
\end{proof}

\subsection{Partial Converse}
Let $R$ be a complete local Gorenstein $3$-fold.  By the Auslander--McKay correspondence, if $R$ is cDV then there are only finitely many basic MM modules up to isomorphism.  Recall from \ref{1d finite conjecture} that we conjecture the converse to be true.  Since such $R$ are known to be  hypersurfaces, the corollary of the following result, although it does not prove the conjecture, does give it some credibility.

As preparation, recall that the \emph{complexity} of $M\in\mod R$ measures the rate of growth of the ranks of the free modules in the minimal projective resolution of $M$, and is defined
\[
\cx_R(M):=\mathrm{inf} \{t\in\mathbb{Z}_{\geq 0}\mid \exists\, a\in\mathbb{R}\mbox{ with }\dim_{\mathbb{C}}\Ext^n_R(M,k)\leq an^{t-1}\mbox{ for }n\gg 0  \}.
\]
Since $\cx_R(M)$ measures the asymptotic behaviour, it is clear that $\cx_R(M)=\cx_R(\Omega^iM)$ for all $i\geq 0$. 

The following extends \cite[\S3]{Bergh} to cover not-necessarily-isolated singularities.  
\begin{prop}\label{CT PB thm}
Suppose that $R$ is a $d$-dimensional complete local Gorenstein algebra.  If $R$ admits only finitely many basic CT modules up to isomorphism, then $R$ is a hypersurface.
\end{prop}
\begin{proof}
Let $M$ be such a CT module, which is necessarily a generator, and consider $X:=\Omega^d k\in\CM R$.  Since $R$ is Gorenstein, by taking a projective cover of $X^*$ then dualizing, we have an exact sequence $0\to X\to P_0\to P_1$ with each $P_i\in \add R$.   Applying $\Hom_R(M,-)$ gives an exact sequence
\[
0\to\Hom_R(M,X)\to\Hom_R(M,P_0)\stackrel{g}{\to}\Hom_R(M,P_1)\to\Cok g\to 0.
\]
Since both $\Hom_R(M,P_i)$ are projective $\End_R(M)$-modules, and $\gl \End_R(M)=d$ by \cite[5.4]{IW4}, it follows that  $\pd_{\End_R(M)}\Hom_R(M,X)\leq d-2$.  

Since $M$ is a generator, $\Hom_R(M,-)\colon\mod R\to\mod\End_R(M)$ is fully faithful, and restricts to an equivalence $\add M\to\proj\End_R(M)$, see e.g.\ \cite[2.5(1)]{IW4}.  Thus we may take a projective resolution 
\[
0\to\Hom_R(M,M_{d-2})\to\hdots\to\Hom_R(M,M_{0})\to\Hom_R(M,X)\to 0
\]
which necessarily comes from a complex 
\begin{eqnarray}
0\to M_{d-2}\to M_{d-3}\hdots\to M_0\to X\to 0.\label{M gen complex}
\end{eqnarray}
This complex \eqref{M gen complex} is exact, since $M$ is a generator.

Now, by general mutation theory, $\Omega^iM$ are CT modules for all $i\in\mathbb{N}$ \cite[6.11]{IW4}, and since by assumption there are only finitely many basic CT modules, $\Omega^iM\cong \Omega^jM$ for some $i\neq j$, which by taking cosyzygies implies that $\Omega^tM\cong M$ for some $t\geq 1$.  Consequently, $\cx_RM\leq 1$.

But splicing the sequence \eqref{M gen complex} gives an exact sequence
\[
0\to M_{d-2}\to M_{d-3}\to C_{d-3}\to 0
\] 
and thus an exact sequence
\begin{eqnarray}
\hdots\to\Ext^n_R(M_{d-2},k)\to\Ext^{n+1}_R(C_{d-3},k)\to\Ext^{n+1}_R(M_{d-3},k)\to\hdots.\label{induct cx Ext}
\end{eqnarray}
Since $\cx_RM\leq 1$ there exist $a_{i}\in\mathbb{R}$ such that $\dim_{\mathbb{C}}\Ext_R^n(M_{i},k)\leq a_i$ for $n\gg0$.  Thus inspecting \eqref{induct cx Ext} it follows that $\dim_{\mathbb{C}}\Ext_R^n(C_{d-3},k)\leq a_{d-2}+a_{d-3}$ for $n\gg0$, so $\cx_R C_{d-3}\leq 1$.  Inducting along the splicing of \eqref{M gen complex} gives $\cx_R X\leq 1$, which implies that $\cx_Rk\leq 1$.  This is well-known to imply that $R$ is a hypersurface \cite{Gulliksen}.
\end{proof}

\begin{cor}\label{CT PB cor}
Suppose that $R$ is a $3$-dimensional complete local normal Gorenstein algebra, and suppose that $R$ admits an NCCR.  If there are only finitely many basic MM generators up to isomorphism,  then $R$ is a hypersurface. 
\end{cor}
\begin{proof}
Since $R$ admits an NCCR, by \cite[5.9]{IW4} $\CM R$ has a CT module.  As a consequence, by \cite[5.11(2)]{IW4} CT modules are precisely the MM generators, so the assumptions now imply that there are only finitely many CT modules.  Thus the result follows from \ref{CT PB thm}.
\end{proof}

\section{Examples}\label{examples section}
In this section we summarise the GIT chamber decompositions of some crepant partial resolutions of ADE surface singularities, and give the corresponding applications to cDV singularities.  We also illustrate how to run Figure~\ref{Fig2} in some explicit cases.

\subsection{GIT Chamber structures}\label{GIT examples section} Throughout this subsection, $Y\to \Spec S$ denotes a crepant partial resolution, where $S$ is a complete local ADE surface singularity, and $X\to \Spec R$ denotes a crepant partial resolution, where $R$ is cDV.

\begin{example}
Suppose that $Y\to \Spec S$ has only one curve above the origin.  Then $\Uptheta$ is parametrised by $\upvartheta_1$, and there is a single wall at $\upvartheta_1=0$
\[
\begin{array}{c}
\begin{tikzpicture}[scale=0.75]
\draw[->] (-2,0)--(2,0);
\node at (2.5,0) {$\upvartheta_1$};
\filldraw[densely dotted,gray] (0,0) circle (2pt);
\end{tikzpicture}
\end{array}
\]
In the $d=3$ crepant setting \ref{crepant setup}, if $X\to\Spec R$ has only one curve above the origin and does not contract a divisor, then it has the above chamber structure.  This includes, as a special case, all simple flops.
\end{example}

\begin{example}\label{2 curves examples}
With notation as in \S\ref{chamber red to surface subsection}, using a similar argument as in \ref{knitting example}, the following are examples of chamber structures for some $2$-curve configurations.
\vspace{-7mm}
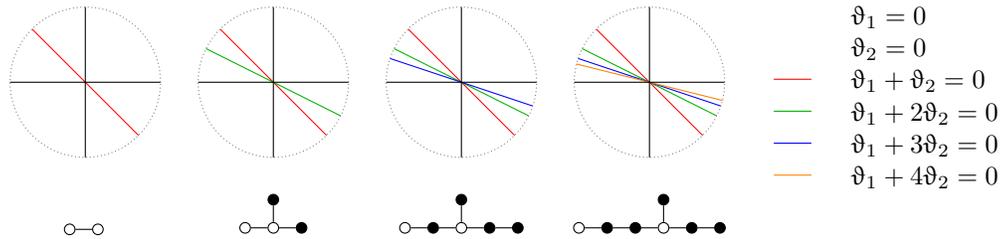
\begin{figure}[H]
\[
\begin{array}{c}
\begin{tikzpicture}
\node at (0,0) {$
\begin{tikzpicture}[scale=0.5]
\coordinate (A1) at (135:2cm);
\coordinate (A2) at (-45:2cm);
\draw[red] (A1) -- (A2);
\draw (-2,0)--(2,0);
\draw (0,-2)--(0,2);
\draw[densely dotted,gray] (0,0) circle (2cm);
\end{tikzpicture}$};
\node at (0,-1.75) {$
\begin{tikzpicture}[scale=0.5]
\node (0) at (0,0) [DW] {};
\node (1) at (0.75,0) [DW] {};
\node (1b) at (0.75,0.75) {};
\draw [-] (0) -- (1);
\end{tikzpicture}$};
\node at (2.5,0) {$
\begin{tikzpicture}[scale=0.5]
\coordinate (A1) at (135:2cm);
\coordinate (A2) at (-45:2cm);
\coordinate (B1) at (153.435:2cm);
\coordinate (B2) at (-26.565:2cm);
\draw[red] (A1) -- (A2);
\draw[green!70!black] (B1) -- (B2);
\draw (-2,0)--(2,0);
\draw (0,-2)--(0,2);
\draw[densely dotted,gray] (0,0) circle (2cm);
\end{tikzpicture}$};
\node at (2.5,-1.75) {$
\begin{tikzpicture}[scale=0.5]
\node (0) at (0,0) [DW] {};
\node (1) at (0.75,0) [DW] {};
\node (1b) at (0.75,0.75) [DB] {};
\node (2) at (1.5,0) [DB] {};
\draw [-] (0) -- (1);
\draw [-] (1) -- (2);
\draw [-] (1) -- (1b);
\end{tikzpicture}$};
\node at (5,0) {$\begin{tikzpicture}[scale=0.5]
\coordinate (A1) at (135:2cm);
\coordinate (A2) at (-45:2cm);
\coordinate (B1) at (153.435:2cm);
\coordinate (B2) at (-26.565:2cm);
\coordinate (C1) at (161.565:2cm);
\coordinate (C2) at (-18.435:2cm);
\draw[red] (A1) -- (A2);
\draw[green!70!black] (B1) -- (B2);
\draw[blue] (C1) -- (C2);
\draw (-2,0)--(2,0);
\draw (0,-2)--(0,2);
\draw[densely dotted,gray] (0,0) circle (2cm);
\end{tikzpicture}$};
\node at (5,-1.75) {$\begin{tikzpicture}[scale=0.5]
\node (-1) at (-0.75,0) [DW] {};
\node (0) at (0,0) [DB] {};
\node (1) at (0.75,0) [DW] {};
\node (1b) at (0.75,0.75) [DB] {};
\node (2) at (1.5,0) [DB] {};
\node (3) at (2.25,0) [DB] {};
\draw [-] (-1) -- (0);
\draw [-] (0) -- (1);
\draw [-] (1) -- (2);
\draw [-] (2) -- (3);
\draw [-] (1) -- (1b);
\end{tikzpicture}$};
\node at (7.5,0) {$\begin{tikzpicture}[scale=0.5]
\coordinate (A1) at (135:2cm);
\coordinate (A2) at (-45:2cm);
\coordinate (B1) at (153.435:2cm);
\coordinate (B2) at (-26.565:2cm);
\coordinate (C1) at (161.565:2cm);
\coordinate (C2) at (-18.435:2cm);
\coordinate (D1) at (165.964:2cm);
\coordinate (D2) at (-14.036:2cm);
\draw[red] (A1) -- (A2);
\draw[green!70!black] (B1) -- (B2);
\draw[blue] (C1) -- (C2);
\draw[orange] (D1) -- (D2);
\draw (-2,0)--(2,0);
\draw (0,-2)--(0,2);
\draw[densely dotted,gray] (0,0) circle (2cm);
\end{tikzpicture}$};
\node at (7.5,-1.75) {$\begin{tikzpicture}[scale=0.5]
\node (-2) at (-1.5,0) [DW] {};
\node (-1) at (-0.75,0) [DB] {};
\node (0) at (0,0) [DB] {};
\node (1) at (0.75,0) [DW] {};
\node (1b) at (0.75,0.75) [DB] {};
\node (2) at (1.5,0) [DB] {};
\node (3) at (2.25,0) [DB] {};
\draw [-] (-2) -- (-1);
\draw [-] (-1) -- (0);
\draw [-] (0) -- (1);
\draw [-] (1) -- (2);
\draw [-] (2) -- (3);
\draw [-] (1) -- (1b);
\end{tikzpicture}$};
\node at (10.5,0) {
$
\begin{array}{cl}
\\
&\upvartheta_1=0\\
&\upvartheta_2=0\\
\begin{array}{c}\begin{tikzpicture}\node at (0,0){}; \draw[red] (0,0)--(0.5,0);\end{tikzpicture}\end{array}&\upvartheta_1+\upvartheta_2=0\\
\begin{array}{c}\begin{tikzpicture}\node at (0,0){}; \draw[green!70!black] (0,0)--(0.5,0);\end{tikzpicture}\end{array}&\upvartheta_1+2\upvartheta_2=0\\
\begin{array}{c}\begin{tikzpicture}\node at (0,0){}; \draw[blue] (0,0)--(0.5,0);\end{tikzpicture}\end{array}&\upvartheta_1+3\upvartheta_2=0\\
\begin{array}{c}\begin{tikzpicture}\node at (0,0){}; \draw[orange] (0,0)--(0.5,0);\end{tikzpicture}\end{array}&\upvartheta_1+4\upvartheta_2=0
\end{array}
$};
\end{tikzpicture}
\end{array}
\]
\caption{Chamber structures for some two curve configurations.}
\end{figure}
\noindent
\end{example}

\begin{example}\label{enriched E7 example}
If further we enhance each chamber with the curve configuration for that chamber (calculated as a byproduct of mutation, as in \ref{knitting example 2}), for $E_6$ with configuration 
\begin{eqnarray}
\begin{array}{c}
\begin{tikzpicture}[scale=1]
\node (-1) at (-0.75,0) [DW] {};
\node (0) at (0,0) [DB] {};
\node (1) at (0.75,0) [DW] {};
\node (1b) at (0.75,0.75) [DB] {};
\node (2) at (1.5,0) [DB] {};
\node (3) at (2.25,0) [DB] {};
\draw [-] (-1) -- (0);
\draw [-] (0) -- (1);
\draw [-] (1) -- (2);
\draw [-] (2) -- (3);
\draw [-] (1) -- (1b);
\end{tikzpicture}
\end{array}\label{E613config}
\end{eqnarray}
after rescaling we obtain the enhanced GIT chamber decomposition
\begin{figure}[H]
\[
\begin{tikzpicture}[scale=1]
\node[draw,densely dotted,gray,name=s,regular polygon, regular polygon sides=10, minimum size=5.5cm] at (0,0) {}; 
\draw (s.corner 1)--  (s.corner 6);
\draw[red] (s.corner 2)--  (s.corner 7);
\draw[green!70!black] (s.corner 3)--  (s.corner 8);
\draw[blue] (s.corner 4)--  (s.corner 9);
\draw (s.corner 5)--  (s.corner 10);
\node[name=t,regular polygon, regular polygon sides=10, minimum size=4cm] at (0,0) {}; 
\node at (t.side 10) {\begin{tikzpicture}[scale=0.3]
\node (-1) at (-0.75,0) [DW] {};
\node (0) at (0,0) [DB] {};
\node (1) at (0.75,0) [DW] {};
\node (1b) at (0.75,0.75) [DB] {};
\node (2) at (1.5,0) [DB] {};
\node (3) at (2.25,0) [DB] {};
\draw [-] (-1) -- (0);
\draw [-] (0) -- (1);
\draw [-] (1) -- (2);
\draw [-] (2) -- (3);
\draw [-] (1) -- (1b);
\end{tikzpicture}
};
\node at (t.side 1) {\begin{tikzpicture}[scale=0.3]
\node (-1) at (-0.75,0) [DB] {};
\node (0) at (0,0) [DW] {};
\node (1) at (0.75,0) [DW] {};
\node (1b) at (0.75,0.75) [DB] {};
\node (2) at (1.5,0) [DB] {};
\node (3) at (2.25,0) [DB] {};
\draw [-] (-1) -- (0);
\draw [-] (0) -- (1);
\draw [-] (1) -- (2);
\draw [-] (2) -- (3);
\draw [-] (1) -- (1b);
\end{tikzpicture}
};
\node at (t.side 2) {\begin{tikzpicture}[scale=0.3]
\node (-1) at (-0.75,0) [DB] {};
\node (0) at (0,0) [DW] {};
\node (1) at (0.75,0) [DB] {};
\node (1b) at (0.75,0.75) [DB] {};
\node (2) at (1.5,0) [DW] {};
\node (3) at (2.25,0) [DB] {};
\draw [-] (-1) -- (0);
\draw [-] (0) -- (1);
\draw [-] (1) -- (2);
\draw [-] (2) -- (3);
\draw [-] (1) -- (1b);
\end{tikzpicture}
};
\node at (t.side 3) {\begin{tikzpicture}[scale=0.3]
\node (-1) at (-0.75,0) [DB] {};
\node (0) at (0,0) [DB] {};
\node (1) at (0.75,0) [DW] {};
\node (1b) at (0.75,0.75) [DB] {};
\node (2) at (1.5,0) [DW] {};
\node (3) at (2.25,0) [DB] {};
\draw [-] (-1) -- (0);
\draw [-] (0) -- (1);
\draw [-] (1) -- (2);
\draw [-] (2) -- (3);
\draw [-] (1) -- (1b);
\end{tikzpicture}
};
\node at (t.side 4) {\begin{tikzpicture}[scale=0.3]
\node (-1) at (-0.75,0) [DB] {};
\node (0) at (0,0) [DB] {};
\node (1) at (0.75,0) [DW] {};
\node (1b) at (0.75,0.75) [DB] {};
\node (2) at (1.5,0) [DB] {};
\node (3) at (2.25,0) [DW] {};
\draw [-] (-1) -- (0);
\draw [-] (0) -- (1);
\draw [-] (1) -- (2);
\draw [-] (2) -- (3);
\draw [-] (1) -- (1b);
\end{tikzpicture}
};
\node at (t.side 5) {\begin{tikzpicture}[scale=0.3]
\node (-1) at (-0.75,0) [DB] {};
\node (0) at (0,0) [DB] {};
\node (1) at (0.75,0) [DW] {};
\node (1b) at (0.75,0.75) [DB] {};
\node (2) at (1.5,0) [DB] {};
\node (3) at (2.25,0) [DW] {};
\draw [-] (-1) -- (0);
\draw [-] (0) -- (1);
\draw [-] (1) -- (2);
\draw [-] (2) -- (3);
\draw [-] (1) -- (1b);
\end{tikzpicture}
};
\node at (t.side 6) {\begin{tikzpicture}[scale=0.3]
\node (-1) at (-0.75,0) [DB] {};
\node (0) at (0,0) [DB] {};
\node (1) at (0.75,0) [DW] {};
\node (1b) at (0.75,0.75) [DB] {};
\node (2) at (1.5,0) [DW] {};
\node (3) at (2.25,0) [DB] {};
\draw [-] (-1) -- (0);
\draw [-] (0) -- (1);
\draw [-] (1) -- (2);
\draw [-] (2) -- (3);
\draw [-] (1) -- (1b);
\end{tikzpicture}
};
\node at (t.side 7) {\begin{tikzpicture}[scale=0.3]
\node (-1) at (-0.75,0) [DB] {};
\node (0) at (0,0) [DW] {};
\node (1) at (0.75,0) [DB] {};
\node (1b) at (0.75,0.75) [DB] {};
\node (2) at (1.5,0) [DW] {};
\node (3) at (2.25,0) [DB] {};
\draw [-] (-1) -- (0);
\draw [-] (0) -- (1);
\draw [-] (1) -- (2);
\draw [-] (2) -- (3);
\draw [-] (1) -- (1b);
\end{tikzpicture}
};
\node at (t.side 8) {\begin{tikzpicture}[scale=0.3]
\node (-1) at (-0.75,0) [DB] {};
\node (0) at (0,0) [DW] {};
\node (1) at (0.75,0) [DW] {};
\node (1b) at (0.75,0.75) [DB] {};
\node (2) at (1.5,0) [DB] {};
\node (3) at (2.25,0) [DB] {};
\draw [-] (-1) -- (0);
\draw [-] (0) -- (1);
\draw [-] (1) -- (2);
\draw [-] (2) -- (3);
\draw [-] (1) -- (1b);
\end{tikzpicture}
};
\node at (t.side 9) {\begin{tikzpicture}[scale=0.3]
\node (-1) at (-0.75,0) [DW] {};
\node (0) at (0,0) [DB] {};
\node (1) at (0.75,0) [DW] {};
\node (1b) at (0.75,0.75) [DB] {};
\node (2) at (1.5,0) [DB] {};
\node (3) at (2.25,0) [DB] {};
\draw [-] (-1) -- (0);
\draw [-] (0) -- (1);
\draw [-] (1) -- (2);
\draw [-] (2) -- (3);
\draw [-] (1) -- (1b);
\end{tikzpicture}
};
\end{tikzpicture}
\]
\caption{Enhanced chamber decomposition for $E_6$ with  configuration \eqref{E613config}.}
\end{figure}
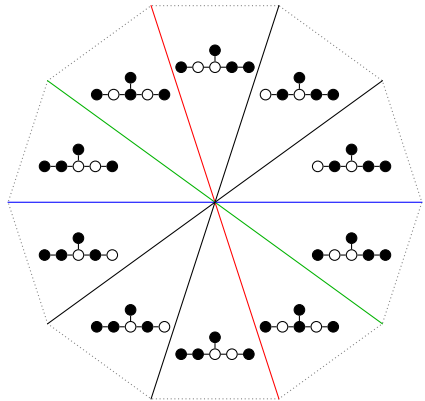
\noindent
By \ref{minimal bound}, it follows that any cDV singularity with a minimal model that cuts under generic hyperplane section to \eqref{E613config} has at least $5$, and at most $10$, minimal models. 
\end{example}

\begin{example}
For the 3-curve configuration $\begin{array}{c}
\begin{tikzpicture}[scale=0.5]
\node (0) at (0,0) [DW] {};
\node (1) at (0.75,0) [DW] {};
\node (1b) at (0.75,0.75) [DB] {};
\node (2) at (1.5,0) [DW] {};
\draw [-] (0) -- (1);
\draw [-] (1) -- (2);
\draw [-] (1) -- (1b);
\end{tikzpicture}
\end{array}
$, the chamber structure is

\begin{figure}[H]
\[
\begin{array}{ccccc}
\begin{array}{c}
\includegraphics[angle=0,scale = 0.2]{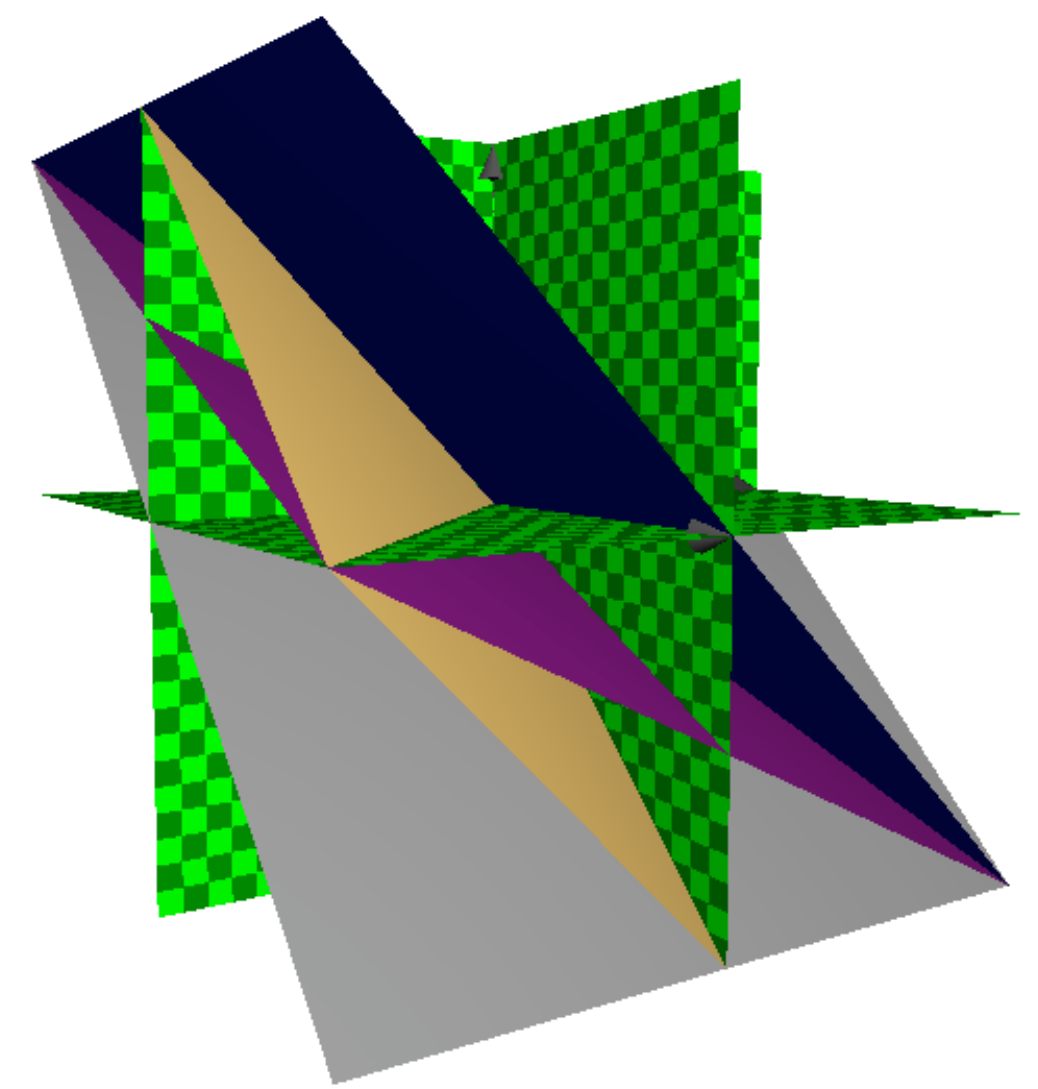}
\end{array}
&&
\begin{array}{c}
\includegraphics[angle=0,scale = 0.2]{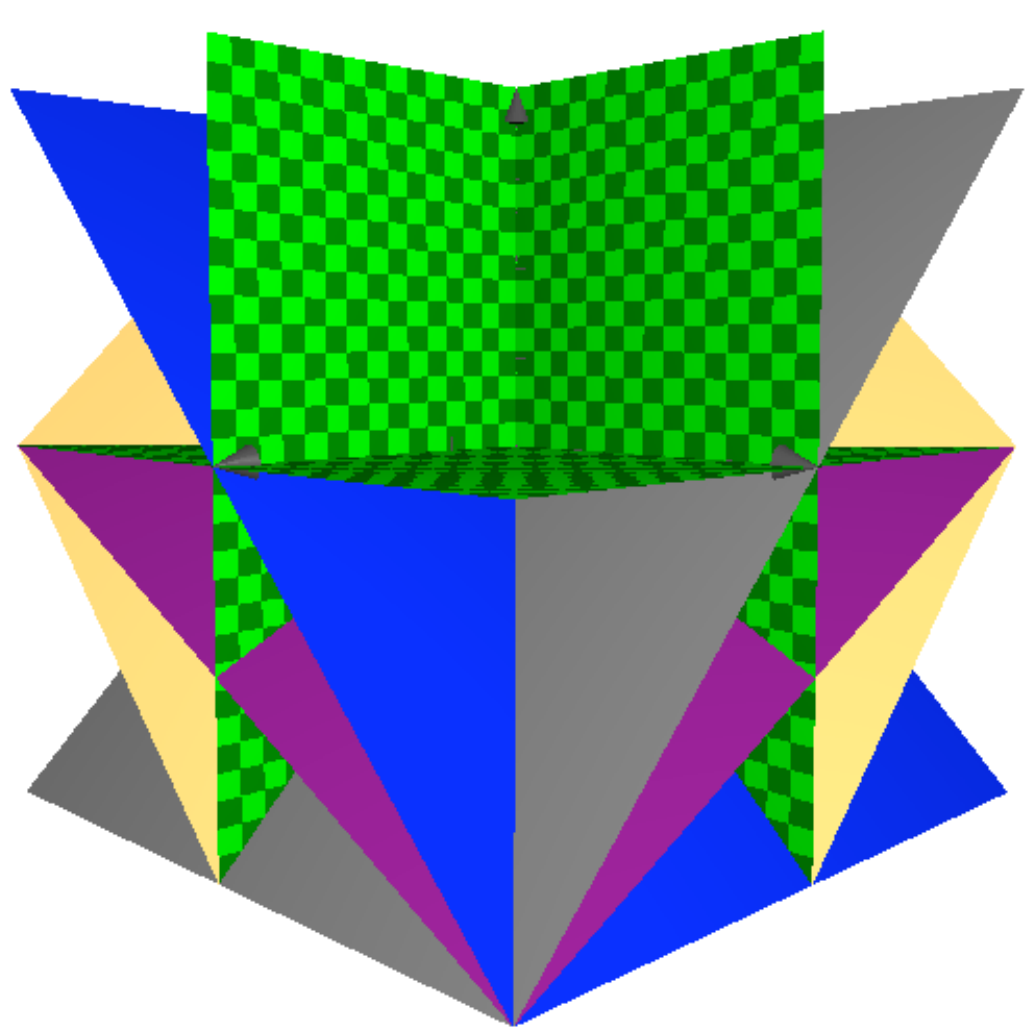}
\end{array}
&&
\begin{array}{l}
\upvartheta_1=0\\
\upvartheta_2=0\\
\upvartheta_3=0\\
\upvartheta_1+\upvartheta_3=0\\
\upvartheta_2+\upvartheta_3=0\\
\upvartheta_1+\upvartheta_2+\upvartheta_3=0\\
\upvartheta_1+\upvartheta_2+2\upvartheta_3=0
\end{array}
\end{array}
\]
\caption{The 32 chambers for $(1,2,1)$}\label{cD4 121}
\end{figure}
\noindent
whereas for the 3-curve configuration $\begin{array}{c}
\begin{tikzpicture}[scale=0.5]
\node (-1) at (-0.75,0) [DW] {};
\node (0) at (0,0) [DB] {};
\node (1) at (0.75,0) [DW] {};
\node (1b) at (0.75,0.75) [DB] {};
\node (2) at (1.5,0) [DB] {};
\node (3) at (2.25,0) [DW] {};
\draw [-] (-1) -- (0);
\draw [-] (0) -- (1);
\draw [-] (1) -- (2);
\draw [-] (2) -- (3);
\draw [-] (1) -- (1b);
\end{tikzpicture}
\end{array}
$, the chamber structure is
\begin{figure}[H]
\[
\begin{array}{ccccc}
\begin{array}{c}
\includegraphics[angle=0,scale = 0.2]{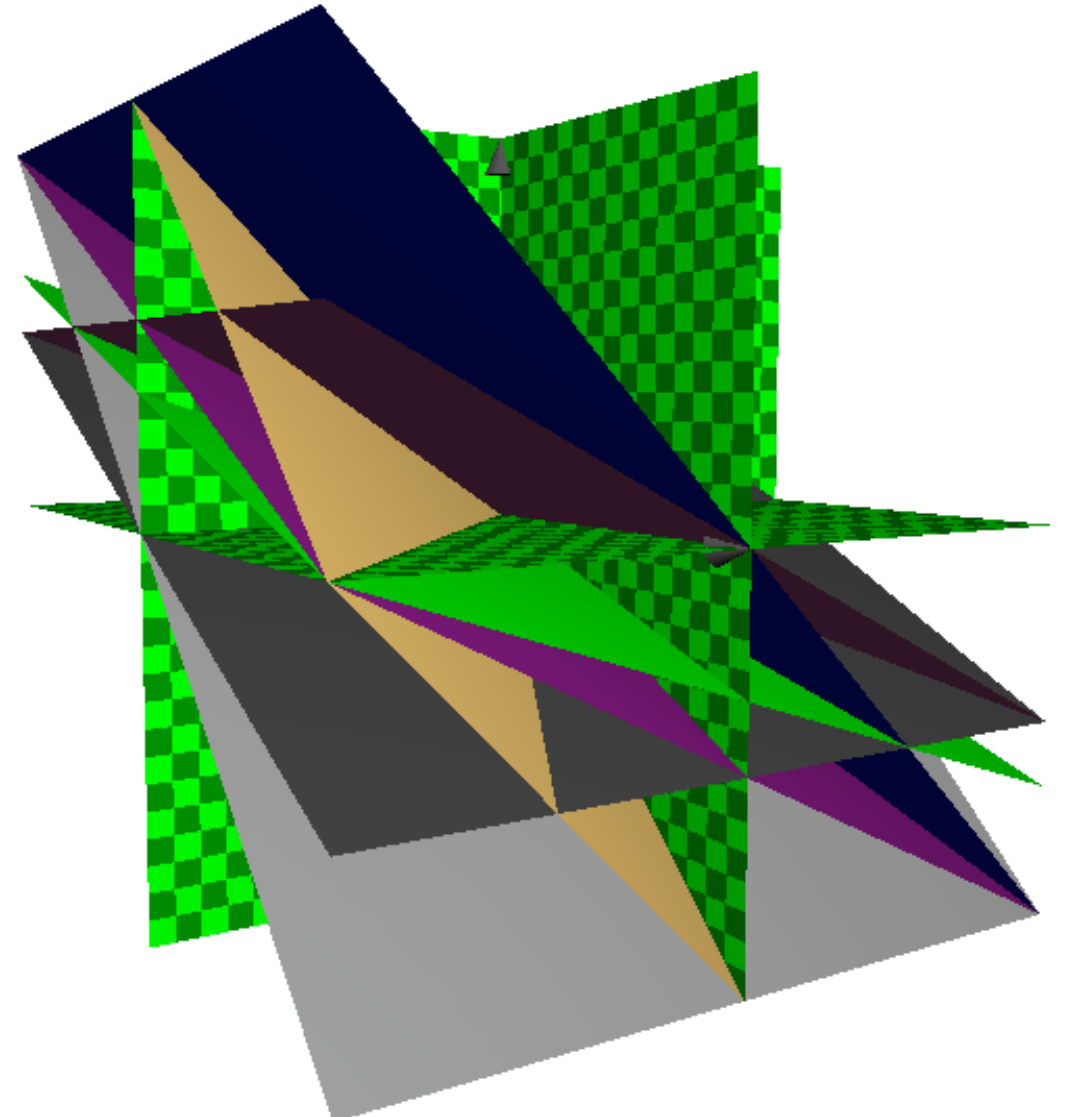}
\end{array}
&&
\begin{array}{c}
\includegraphics[angle=0,scale = 0.2]{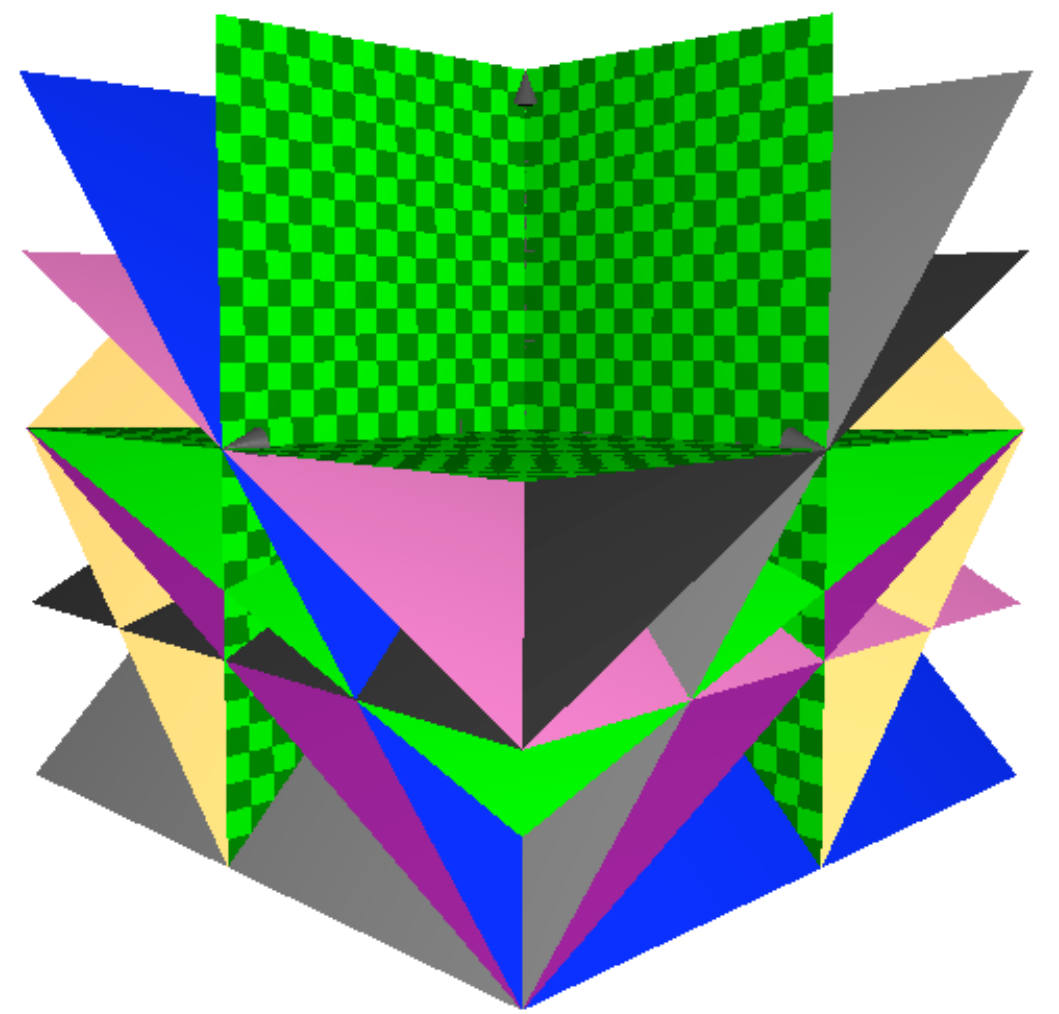}
\end{array}
&&
{\small
\begin{array}{l}
\upvartheta_1=0\\
\upvartheta_2=0\\
\upvartheta_3=0\\
\upvartheta_1+\upvartheta_3=0\\
\upvartheta_1+2\upvartheta_3=0\\
\upvartheta_2+\upvartheta_3=0\\
\upvartheta_2+2\upvartheta_3=0\\
\upvartheta_1+\upvartheta_2+\upvartheta_3=0\\
\upvartheta_1+\upvartheta_2+2\upvartheta_3=0\\
\upvartheta_1+\upvartheta_2+3\upvartheta_3=0
\end{array}}
\end{array}
\]
\caption{The 60 chambers for $(1,3,1)$}
\end{figure}
\noindent
Tracking the dual graph through mutation, as in \ref{knitting example 2} and \ref{minimal bound}, any cDV singularity with a minimal model that cuts to the above $D_4$ configuration has at least $4$ and at most $32$ minimal models.  Any cDV singularity with a minimal model that cuts to the above $E_6$ configuration has at least $5$ and at most $60$ minimal models. 
\end{example}

\begin{remark}
The singularity $R:=\mathbb{C}[[u,x,y,z]]/(u^2-xyz)$ in \ref{m not n} is in fact $cD_4$ with a three curve configuration, so the chamber structure is precisely Figure~\ref{cD4 121}.  The chamber structure for the particular example $u^2=xyz$ was computed independently, using entirely different methods, by Craw and King in 2000.  Indeed, \cite[5.31, footnote 5 p117]{AliThesis} computes the first four chambers.  See also \cite{MT}.
\end{remark}

\subsection{Running the algorithm}
This subsection illustrates how to run the Homological MMP in two examples.  For the aid of the reader, we begin with the toric example in \ref{m not n}, since the geometry will already be familiar. 
\begin{example}\label{Z2Z2example}
Consider again the $cD_4$ singularity $R:=\mathbb{C}[[u,x,y,z]]/(u^2=xyz)$.  As in \cite[6.26]{IW4}, $N:=R\oplus (u,x)\oplus (u,y)\oplus (u,z)$ is an MM (in fact, CT) $R$-module, and 

\begin{eqnarray}
\Lambda:=\End_R(N)\cong
\begin{array}{ccc}
\begin{array}{c}
\begin{tikzpicture}[bend angle=10, looseness=1,>=stealth]
\node[name=s,regular polygon, regular polygon sides=4, minimum size=2.5cm] at (0,0) {}; 
\node (1) at (s.corner 1)  {$\scriptstyle N_{2}$};
\node (2) at (s.corner 2)  {$\scriptstyle N_{1}$};
\node (3) at (s.corner 3)  {$\scriptstyle R$};
\node (4) at (s.corner 4)  {$\scriptstyle N_{3}$};
\draw[->,black,bend left] (4) to node[gap]{$\scriptstyle x$} (1);
\draw[->,black,bend left] (1) to node[gap]{$\scriptstyle x$} (4);
\draw[->,black,bend left] (3) to node[gap]{$\scriptstyle x$} (2);
\draw[->,black,bend left] (2) to node[gap]{$\scriptstyle x$} (3);
\draw[->,black,bend left] (2) to node[gap]{$\scriptstyle y$} (1);
\draw[->,black,bend left] (1) to node[gap]{$\scriptstyle y$} (2);
\draw[->,black,bend left] (3) to node[gap]{$\scriptstyle y$} (4);
\draw[->,black,bend left] (4) to node[gap]{$\scriptstyle y$} (3);
\draw[->,black,bend right,looseness=0.75]  (3) to
node[inner sep=1pt,fill=white,pos=0.56] {} 
node[inner sep=1pt,fill=white,pos=0.44] {}
node[inner sep=0.5pt,fill=white,pos=0.7] {$\scriptstyle z$} (1);
\draw[->,black,bend right,looseness=0.75]  (1) to 
node[inner sep=1pt,fill=white,pos=0.53] {}
node[inner sep=1pt,fill=white,pos=0.4] {} 
node[inner sep=0.5pt,fill=white,pos=0.7] {$\scriptstyle z$} (3);
\draw[->,black,bend left,looseness=0.75]  (4) to node[inner sep=0.5pt,fill=white,pos=0.7]  {$\scriptstyle z$} (2);
\draw[->,black,bend left,looseness=0.75]  (2) to node[inner sep=0.5pt,fill=white,pos=0.7]  {$\scriptstyle z$}  (4);
\end{tikzpicture}
\end{array}
&&
{\small \begin{array}{c}
xy=yx\\
xz=zx\\
yz=zy
\end{array}}
\end{array} \label{start NCCR}
\end{eqnarray}
with the relations being interpreted as $x,y$ and $z$ commute wherever that makes sense.  From the quiver, by \ref{reconstruction new} we read off that the fibre above the origin has three curves meeting at a point, and all are $(-1,-1)$-curves.

{\bf Step 1: Contractions.}  We inspect the contraction algebras to determine which sets of curves are floppable. It is clear that $\Lambda_{\{1\}}=\Lambda_{\{2\}}=\Lambda_{\{3\}}=\mathbb{C}$, and so each of the three curves is individually floppable.    Furthermore, 

\[
\Lambda_{\{1,2\}}\cong
\begin{array}{ccc}
\begin{array}{c}
\begin{tikzpicture}[bend angle=10, looseness=1,>=stealth]
\node[name=s,regular polygon, regular polygon sides=4, minimum size=2.5cm] at (0,0) {}; 
\node (1) at (s.corner 1)  {$\scriptstyle N_{2}$};
\node (2) at (s.corner 2)  {$\scriptstyle N_{1}$};
\node[black!20] (3) at (s.corner 3)  {$\scriptstyle R$};
\node[black!20] (4) at (s.corner 4)  {$\scriptstyle N_{3}$};
\draw[->,black!20,bend left] (4) to node[gap]{$\scriptstyle x$} (1);
\draw[->,black!20,bend left] (1) to node[gap]{$\scriptstyle x$} (4);
\draw[->,black!20,bend left] (3) to node[gap]{$\scriptstyle x$} (2);
\draw[->,black!20,bend left] (2) to node[gap]{$\scriptstyle x$} (3);
\draw[->,black,bend left] (2) to node[gap]{$\scriptstyle y$} (1);
\draw[->,black,bend left] (1) to node[gap]{$\scriptstyle y$} (2);
\draw[->,black!20,bend left] (3) to node[gap]{$\scriptstyle y$} (4);
\draw[->,black!20,bend left] (4) to node[gap]{$\scriptstyle y$} (3);
\draw[->,black!20,bend right,looseness=0.75]  (3) to
node[inner sep=1pt,fill=white,pos=0.56] {} 
node[inner sep=1pt,fill=white,pos=0.44] {}
node[inner sep=0.5pt,fill=white,pos=0.7] {$\scriptstyle z$} (1);
\draw[->,black!20,bend right,looseness=0.75]  (1) to 
node[inner sep=1pt,fill=white,pos=0.53] {}
node[inner sep=1pt,fill=white,pos=0.4] {} 
node[inner sep=0.5pt,fill=white,pos=0.7] {$\scriptstyle z$} (3);
\draw[->,black!20,bend left,looseness=0.75]  (4) to node[inner sep=0.5pt,fill=white,pos=0.7]  {$\scriptstyle z$} (2);
\draw[->,black!20,bend left,looseness=0.75]  (2) to node[inner sep=0.5pt,fill=white,pos=0.7]  {$\scriptstyle z$}  (4);
\end{tikzpicture}
\end{array}
&\cong&
\begin{array}{c}
\begin{tikzpicture}[bend angle=15, looseness=1,>=stealth]
\node (a) at (-1,0) [vertex] {};
\node (b) at (0,0) [vertex] {};
\node at (1.5,0) {\small{(no relations)}};
\draw[->,bend right] (a) to (b);
\draw[<-,bend left] (a) to (b);
\end{tikzpicture}
\end{array}
\end{array} 
\]
since all relations in \eqref{start NCCR} involve $x$'s and $z$'s, and these are zero in the quotient.  Thus $\dim_{\mathbb{C}}\Lambda_{\{1,2\}}=\infty$ and so curves $1$ and $2$ do not flop together.  By symmetry in this example, the same can be said of all pairs.  Finally $\dim_{\mathbb{C}}\Lambda_{\{1,2,3\}}=\infty$.  Hence each individual curve flops, but no other combinations do.

{\bf Step 2: Flops.} By symmetry, we only need mutate at $N_2$ (i.e.\ flop curve two), since the other cases are identical.  In this example, it is clear that the relevant approximation is
\[
0\to N_2\xrightarrow{(z\, y\, x)}R\oplus N_1\oplus N_3
\]
Thus the mutation at vertex $N_2$ changes $N=R\oplus N_{1}\oplus N_{2}\oplus N_{3}$ into $\upnu_2N:=R\oplus N_{1}\oplus K_{2}\oplus N_{3}$ where $K_2$ is the cokernel of the above map which (by counting ranks) has rank 2.   On the level of quivers of the endomorphism rings, this induces the mutation

\[
\End_R(\upnu_2N)\cong
\begin{array}{ccc}
\begin{array}{c}
\begin{tikzpicture}[bend angle=10, looseness=1,>=stealth]
\node[name=s,regular polygon, regular polygon sides=4, minimum size=2.5cm] at (0,0) {}; 
\node (1) at (s.corner 1)  {$\scriptstyle K_{2}$};
\node (2) at (s.corner 2)  {$\scriptstyle N_{1}$};
\node (2a) at ($(s.corner 2)+(180:4pt)$)  {};
\node (3) at (s.corner 3)  {$\scriptstyle R$};
\node (3a) at ($(s.corner 3)+(-135:3pt)$)  {};
\node (4) at (s.corner 4)  {$\scriptstyle N_{3}$};
\node (4a) at ($(s.corner 4)+(-90:2pt)$)  {};
\draw[->,black,bend left] (4) to node[gap]{$\scriptstyle c$} (1);
\draw[->,black,bend left] (1) to node[gap]{$\scriptstyle C$} (4);
\draw[->,black,bend left] (2) to node[gap]{$\scriptstyle b$} (1);
\draw[->,black,bend left] (1) to node[gap]{$\scriptstyle B$} (2);
\draw[->,black,bend right,looseness=0.75]  (3) to
node[gap] {$\scriptstyle a$} (1);
\draw[->,black,bend right,looseness=0.75]  (1) to 
node[gap] {$\scriptstyle A$} (3);
\draw[->]  (2a) edge [in=-150,out=150,loop,looseness=8] node[left]{$\scriptstyle v$} (2a);
\draw[->]  (3a) edge [in=-102.5,out=-167.5,loop,looseness=8] node[below left]{$\scriptstyle u$} (3a);
\draw[<-]  (4a) edge [in=-120,out=-60,loop,looseness=8] node[below]{$\scriptstyle w$} (4a);
\end{tikzpicture}
\end{array}
&&
{\scriptsize
\begin{array}{c}
\begin{array}{l}
aA=0\\
bB=0\\
cC=0
\end{array}\\
\begin{array}{l}
ua=aCcBb+aBbCc\\
vb=bAaCc+bCcAa\\
wc=cAaBb+cBbAa\\
Au=BbCcA+CcBbA\\
Bv=AaCcB+CcAaB\\
Cw=AaBbC+BbAaC
\end{array}
\end{array}
}
\end{array}
\]
By \ref{reconstruction new} we read off that the new dual graph has three curves intersecting in a type $A$ configuration, with the outer two curves being $(-2,0)$-curves, and the inner curve being a $(-1,-1)$-curve.   By the symmetry of the situation, we obtain the beginning of the simple mutation graph:
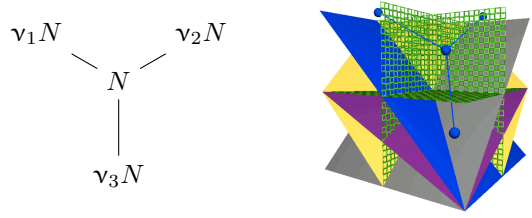
\begin{figure}[H]
\[
\begin{array}{ccc}
\begin{array}{c}
\begin{tikzpicture}
\node[name=s,regular polygon, regular polygon sides=3, minimum size=2.5cm,rotate=60] at (0,0) {}; 
\node (M) at (s.center) {$N$};
\node (L) at (s.corner 1) {$\upnu_1 N$}; 
\node (R) at (s.corner 2) {$\upnu_3 N$};  
\node (T) at (s.corner 3) {$\upnu_2 N$};  
\draw[-] (M) --  (L);
\draw[-] (M) --  (R); 
\draw[-] (M) --  (T);
\end{tikzpicture}
\end{array}
&&
\begin{array}{c}
\includegraphics[angle=0,scale = 0.15]{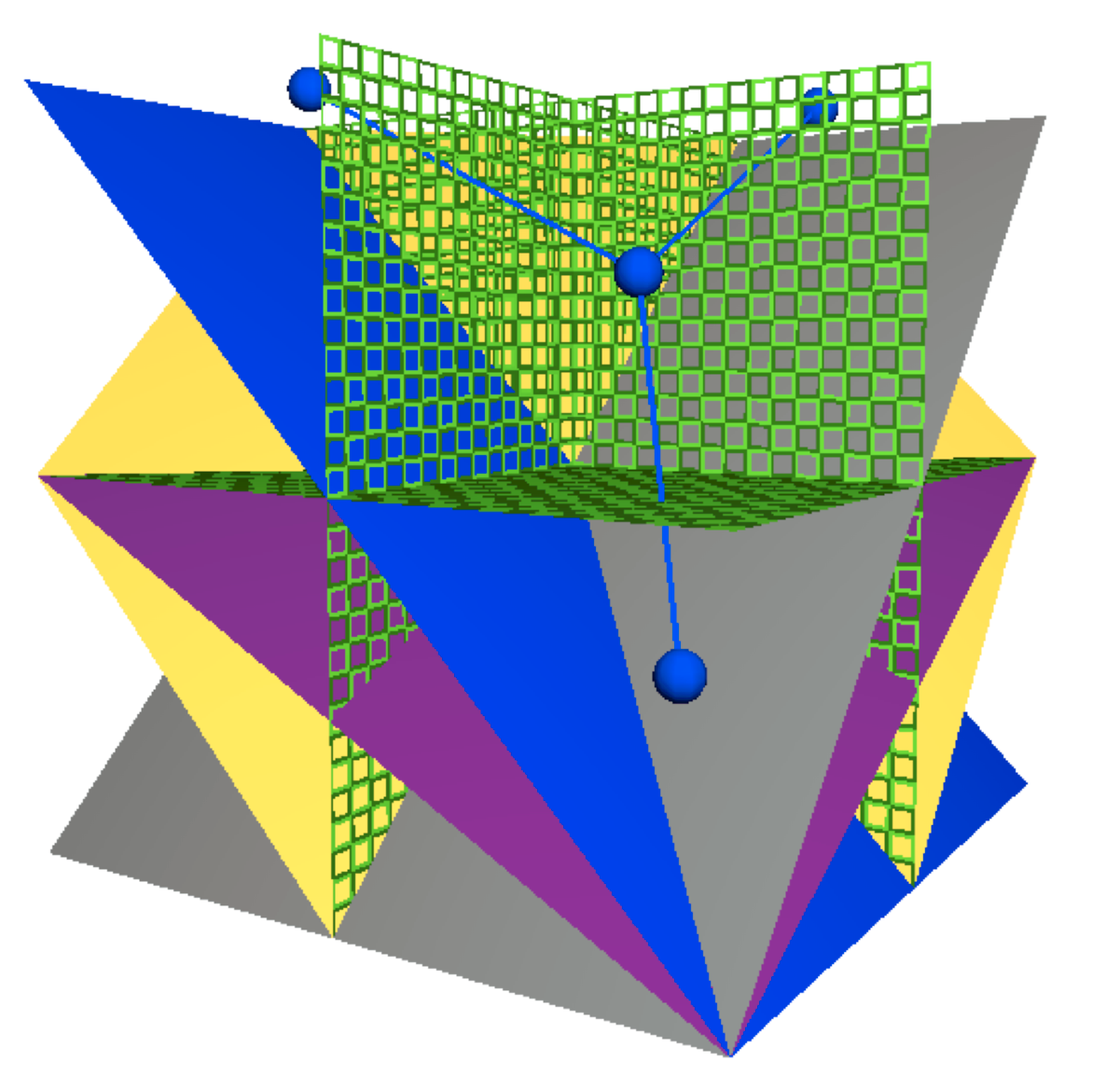}
\end{array}
\end{array}
\]
\caption{The simple mutation graph for $u^2=xyz$.}
\end{figure}
We next claim that this is precisely the simple mutation graph of the MM generators, equivalently we have already found all minimal models of $\Spec R$.

{\bf Step 1b: Contractions}.  We plug in the mutated algebra $\End_R(\upnu_{2}N)$ into Step 1, and repeat.  Due to the relations in the algebra $\upnu_{2}\Lambda=\End_R(\upnu_{2}N)$, it follows that $\dim_\mathbb{C}(\upnu_{2}\Lambda)_{\{1\}}=\infty=\dim_{\mathbb{C}}(\upnu_{2}\Lambda)_{\{3\}}$, thus in $\upnu_{2}\Lambda$ the only curve we can non-trivially mutate is the middle one, which gives us back our original $N$.  Thus the Homological MMP stops, and we have reached all minimal models.
\end{example}

\begin{example}
Consider the $cD_4$ singularity $R:=\mathbb{C}[[u,v,y,z]]/(u^2-v(x^2-4y^3))$.  Since $R\cong\mathbb{C}[[x,y,z]]^{S_3}$ for the subgroup
\[
S_3:=\left\langle \begin{pmatrix}\upvarepsilon_3&0&0\\0&\upvarepsilon_3^{2}&0\\ 0&0&1\end{pmatrix},
\begin{pmatrix}0&1&0\\1&0&0\\ 0&0&-1\end{pmatrix}
\right\rangle\leq\SL(3,\mathbb{C}),
\]
there is an MM generator (in fact, CT module) given by the skew group ring $\Lambda:=\End_R(N)$, and further by \cite{BSW} it can be presented as
\[
\Lambda\cong
\begin{array}{ccc}
\begin{array}{c}
\begin{tikzpicture}[bend angle=10, looseness=1,>=stealth]
\node[name=s,regular polygon, regular polygon sides=3, minimum size=2.5cm] at (0,0) {}; 
\node (1) at (s.corner 1)  {$\scriptstyle N_2$};
\node (2) at (s.corner 2)  {$\scriptstyle R$};
\node (1a) at ($(s.corner 1)+(120:2pt)$)  {};
\node (1b) at ($(s.corner 1)+(60:2pt)$)  {};
\node (3) at (s.corner 3)  {$\scriptstyle N_1$};
\draw[->,black,bend left] (1) to node[gap]{$\scriptstyle A$} (2);
\draw[->,black,bend left] (2) to node[gap]{$\scriptstyle a$} (1);
\draw[->,black,bend left] (2) to node[gap]{$\scriptstyle x$} (3);
\draw[->,black,bend left] (3) to node[gap]{$\scriptstyle y$} (2);
\draw[->,black,bend left]  (3) to node[gap] {$\scriptstyle b$} (1);
\draw[->,black,bend left]  (1) to node[gap] {$\scriptstyle B$} (3);
\draw[->]  (1a) edge [in=100,out=160,loop,looseness=10] node[left]{$\scriptstyle l$} (1a);
\draw[->]  (1b) edge [in=80,out=20,loop,looseness=10] node[right]{$\scriptstyle z$} (1b);
\end{tikzpicture}
\end{array}
&&
{\scriptsize
\begin{array}{c}
\begin{array}{c}
aB=0\\
bA=0\\
Aa+Bb=2l^2
\end{array}\\
\begin{array}{l}
az=xb\\
bz=ya\\
zA=By\\
zB=Ax\\
lz+zl=0
\end{array}
\end{array}
}
\end{array}
\]
given by the superpotential $\Upphi:=Axb+Bya-zAa-zBb+2zl^2$.   By \ref{reconstruction new} we read off that there are two curves intersecting transversely, one with normal bundle $(-3,1)$, the other with normal bundle $(-1,-1)$.  Further, 
\[
\Lambda_{\{2\}}\cong
\begin{array}{ccc}
\begin{array}{c}
\begin{tikzpicture}[bend angle=10, looseness=1,>=stealth]
\node[name=s,regular polygon, regular polygon sides=3, minimum size=2.5cm] at (0,0) {}; 
\node (1) at (s.corner 1)  {$\scriptstyle N_2$};
\node[black!20] (2) at (s.corner 2)  {$\scriptstyle R$};
\node (1a) at ($(s.corner 1)+(120:2pt)$)  {};
\node (1b) at ($(s.corner 1)+(60:2pt)$)  {};
\node[black!20] (3) at (s.corner 3)  {$\scriptstyle N_1$};
\draw[->,black!20,bend left] (1) to node[gap]{$\scriptstyle A$} (2);
\draw[->,black!20,bend left] (2) to node[gap]{$\scriptstyle a$} (1);
\draw[->,black!20,bend left] (2) to node[gap]{$\scriptstyle x$} (3);
\draw[->,black!20,bend left] (3) to node[gap]{$\scriptstyle y$} (2);
\draw[->,black!20,bend left]  (3) to node[gap] {$\scriptstyle b$} (1);
\draw[->,black!20,bend left]  (1) to node[gap] {$\scriptstyle B$} (3);
\draw[->]  (1a) edge [in=100,out=160,loop,looseness=10] node[left]{$\scriptstyle l$} (1a);
\draw[->]  (1b) edge [in=80,out=20,loop,looseness=10] node[right]{$\scriptstyle z$} (1b);
\end{tikzpicture}
\end{array}
&\cong&
\frac{\mathbb{C}\langle \langle l,z\rangle\rangle}{(l^2,lz+zl)}
\end{array} 
\]
which is infinite dimensional, and clearly $\Lambda_{\{1\}}=\mathbb{C}$, which is finite dimensional.  Hence by \ref{contract on f} only the $(-1,-1)$-curve flops.  It is easy to calculate that
\[
\upnu_1\Lambda\cong
\begin{array}{ccc}
\begin{array}{c}
\begin{tikzpicture}[bend angle=10, looseness=1,>=stealth]
\node[name=s,regular polygon, regular polygon sides=3, minimum size=2.5cm] at (0,0) {}; 
\node (1) at (s.corner 1)  {$\scriptstyle N_2$};
\node (2) at (s.corner 2)  {$\scriptstyle R$};
\node (1a) at ($(s.corner 1)+(120:2pt)$)  {};
\node (1b) at ($(s.corner 1)+(60:2pt)$)  {};
\node (3) at (s.corner 3)  {$\scriptstyle K_1$};
\draw[->,black,bend left] (2) to node[gap]{$\scriptstyle s$} (3);
\draw[->,black,bend left] (3) to node[gap]{$\scriptstyle t$} (2);
\draw[->,black,bend left]  (3) to node[gap] {$\scriptstyle c$} (1);
\draw[->,black,bend left]  (1) to node[gap] {$\scriptstyle C$} (3);
\draw[->]  (1a) edge [in=60,out=130,loop,looseness=8] node[above]{$\scriptstyle l$} (1a);
\draw[->]  (2) edge [in=150,out=-150,loop,looseness=6] node[left]{$\scriptstyle (xy)$} (2);
\end{tikzpicture}
\end{array}
&&
{\scriptsize
\begin{array}{c}
\begin{array}{c}
lCc+Ccl=0\\
st=0\\
t(xy)=cCcCt\\
(xy)s=scCcC\\
2l^2C=CcCts+CtscC\\
2cl^2=tscCc+cCtsc
\end{array}
\end{array}
}
\end{array}
\]
given by potential $\Upphi':=-t(xy)s-cl^2C+cCcCts$.  Again by inspection, $(\upnu_1\Lambda)_{\{2\}}=\mathbb{C}[[l]]$, which is infinite dimensional, and $(\upnu_1\Lambda)_{\{2\}}=\mathbb{C}$, which is finite dimensional.

Hence the only way to mutate is back,  so the Homological MMP finishes.  It follows that the full mutation graph, viewed inside the GIT chamber structure $\Uptheta(\Lambda)$, is
\[
\begin{array}{ccc}
\begin{array}{c}
\begin{tikzpicture}[scale=0.75]
\coordinate (A1) at (135:2cm);
\coordinate (A2) at (-45:2cm);
\coordinate (B1) at (153.435:2cm);
\coordinate (B2) at (-26.565:2cm);
\draw[red] (A1) -- (A2);
\draw[green!70!black] (B1) -- (B2);
\draw[->] (-2,0)--(2,0);
\node at (2.5,0) {$\upvartheta_1$};
\draw[->] (0,-2)--(0,2);
\node at (0.5,2.25) {$\upvartheta_2$};
\draw[densely dotted,gray] (0,0) circle (2cm);
\filldraw[blue] (45:1cm) circle (2pt);
\filldraw[blue] (112.5:1cm) circle (2pt);
\draw[blue] (45:1cm)--(112.5:1cm);
\end{tikzpicture}
\end{array}
&&
\begin{array}{cl}
\\
&\upvartheta_1=0\\
&\upvartheta_2=0\\
\begin{array}{c}\begin{tikzpicture}\node at (0,0){}; \draw[red] (0,0)--(0.5,0);\end{tikzpicture}\end{array}&\upvartheta_1+\upvartheta_2=0\\
\begin{array}{c}\begin{tikzpicture}\node at (0,0){}; \draw[green!70!black] (0,0)--(0.5,0);\end{tikzpicture}\end{array}&\upvartheta_1+2\upvartheta_2=0
\end{array}
\end{array}
\]
\end{example}

\begin{remark}
The mutation trees of quivers are usually quite easy to write down, and this then determines all the geometry.  We refer the reader to \cite[\S4.1]{NS} for the calculation of the mutation trees for some other quotient singularities, in particular \cite[4.4]{NS}.  We remark that it follows from Figure~\ref{Fig2} that \cite[\S4.1]{NS} is now enough to establish we have all minimal models.  In particular, we can immediately read off the dual graph and whether curves flop from the quivers there, avoiding all the hard explicit calculations in \cite[\S5--6]{NS}.
\end{remark}

\appendix
\section{Mutation Summary}\label{appendix mut}
This appendix contains the mutation results needed in the text, including \ref{Ext 2 thm} and \ref{Ext 3 thm}, which for the most part are just mild generalisations of some of the results in \cite[\S 6]{IW4}.  Throughout, we maintain the setup of \S\ref{mut prelim} and \ref{stab setup}, so unless stated otherwise $R$ denotes a complete local normal $d$-sCY commutative algebra with $d\geq 2$, $M\in\refl R$ denotes a basic modifying module $M$, and $M_I$ is a summand of $M$.   

The following duality proposition is important, and will be used extensively.

\begin{prop}\label{dualityofapprox}\cite[6.4]{IW4}
With notation as above, 
\begin{enumerate}
\item\label{K0b} Applying $\Hom_{R}(-,M_{I^{c}})$ to \eqref{K0} induces an exact sequence
\begin{align*}
&0\to \Hom_{R}(M_I,M_{I^{c}})\stackrel{\cdot a}\to \Hom_{R}(V_I,M_{I^{c}})\stackrel{\cdot c}\to \Hom_{R}(K_I,M_{I^{c}})\to 0.
\end{align*}
In particular $c$ is a minimal left $\add M_{I^{c}}$-approximation. 
\item\label{K1b} Applying $\Hom_{R}(-,M_{I^{c}}^*)$ to \eqref{K1} induces an exact sequence
\begin{align*}
&0\to \Hom_{R}(M_I^{*},M_{I^{c}}^*)\stackrel{\cdot b}\to \Hom_{R}(U_I^{*},M_{I^{c}}^*)\stackrel{\cdot d}\to \Hom_{R}(J_I,M_{I^{c}}^*)\to 0
\end{align*}
In particular $d$ is a minimal left $\add M_{I^{c}}^*$-approximation.
\item We have that 
\begin{align}
&0\to M_I^{*}\stackrel{a^{*}}\to V_I^{*}\stackrel{c^{*}}\to K_I^{*}\notag\\
&0\to M_I\stackrel{b^{*}}\to U_I\stackrel{d^{*}}\to J_I^{*} \label{K1D}
\end{align}
are exact, inducing exact sequences
\begin{align}
&0\to \Hom_{R}(M_{I^{c}}^*,M_I^{*})\stackrel{a^{*}\cdot}\to \Hom_{R}(M_{I^{c}}^*,V_I^{*})\stackrel{c^{*}\cdot}\to \Hom_{R}(M_{I^{c}}^*,K_I^{*})\to 0\label{K0Da}\\
&0\to \Hom_{R}(K_I^{*},M_{I^{c}}^*)\stackrel{\cdot c^{*}}\to \Hom_{R}(V_I^{*},M_{I^{c}}^*)\stackrel{\cdot a^{*}}\to \Hom_{R}(M_I^{*},M_{I^{c}}^*)\to 0\notag\\
&0\to \Hom_{R}(M_{I^{c}},M_I)\stackrel{b^{*}\cdot}\to \Hom_{R}(M_{I^{c}},U_I)\stackrel{d^{*}\cdot}\to \Hom_{R}(M_{I^{c}},J_I^{*})\to 0 \label{K1Da}\\
&0\to \Hom_{R}(J_I^{*},M_{I^{c}})\stackrel{\cdot d^{*}}\to \Hom_{R}(U_I,M_{I^{c}})\stackrel{\cdot b^{*}}\to \Hom_{R}(M_I,M_{I^{c}})\to 0 \notag
\end{align}
\end{enumerate}
\end{prop}

In this level of generality, usually $\upnu_IM\ncong\upmu_IM$, and $\upnu_I\upnu_IM\ncong M$.  However, we will be interested in when these, and other, nice situations occur.

\begin{prop}\label{proj res thm 1}
In setup of \ref{stab setup}, assume further that $\pd_\Lambda\Lambda_I=2$.  Then 
\begin{enumerate}
\item\label{proj res thm 1 part 1} $\upmu_IM\cong M\cong \upnu_IM$.
\item\label{proj res thm 1 part 2} The minimal projective resolution of $\Lambda_I$ as a $\Lambda$-module is
\[
0\to \Hom_R(M,M_I)\xrightarrow{\cdot c} \Hom_R(M,V_I)\to \Hom_R(M,M_I)\xrightarrow{\cdot a} \Lambda_I\to 0\label{pd2a}
\] 
\item\label{proj res thm 1 part 3} The minimal projective resolution of $\Lambda^{\op}_I$ as a $\Lambda^{\op}\cong\End_R(M^*)$-module is
\[
0\to \Hom_R(M^*,M_I^*)\xrightarrow{\cdot a^*} \Hom_R(M^*,V_I^*)\xrightarrow{\cdot c^*} \Hom_R(M^*,M_I^*)\to \Lambda^{\op}_I\to 0
\] 
\end{enumerate}
\end{prop}
\begin{proof}
Since  $\pd_\Lambda\Lambda_I=2$ there is a minimal projective resolution
\begin{eqnarray}
0\to Q_1\to Q_0\stackrel{f}{\to} P_I\to \Lambda_I\to 0 \label{pd 2sequence}
\end{eqnarray}
where $P_I:=\Hom_R(M,M_I)$ is not a summand of $Q_0$.  Now $M_I=\bigoplus_{i\in I}M_i$, so taking the minimal right $M_{I^c}$-approximations of each $M_i$ gives exact sequences
\begin{eqnarray}
0\to K_i\xrightarrow{c_i} V_i\xrightarrow{a_i} M_i\label{22j}
\end{eqnarray}
which sum together to give the exact sequence
\begin{eqnarray}
0\to (K_I=\bigoplus_{i\in I}K_{i})\xrightarrow{c} (V_I=\bigoplus_{i\in I}V_i)\xrightarrow{a} (M_I=\bigoplus_{i\in I}M_i).\label{22}
\end{eqnarray}
This is the minimal right $\add M_{I^c}$-approximation of $M_I$, so applying $\Hom_R(M,-)$ gives \eqref{pd 2sequence}.  In particular $K_I\in\add M$.  We claim that $K_I\cong M_I$, as this proves (2).

First, each $K_{i}\in\add M_I$. To see this, suppose it is false, which since $K_I\in\add M$ would mean that $K_{i}\in\add M_{I^c}$.  By \ref{dualityofapprox}, dualizing \eqref{22j} gives exact sequences
\begin{eqnarray}
0\to M_i^*\to V_i^*\to K_i^* \label{22a}
\end{eqnarray}
such that
\begin{eqnarray}
0\to \Hom_R(M_{I^c}^*,M_i^*)\to \Hom_R(M_{I^c}^*,V_i^*)\to \Hom_R(M_{I^c}^*,K_i^*)\to 0\label{22b}
\end{eqnarray}
is exact.  Since we are assuming $K_i\in\add M_{I^c}$, this would mean that 
\[
0\to \Hom_R(K_i^*,M_i^*)\to \Hom_R(K_i^*,V_i^*)\to \Hom_R(K_i^*,K_i^*)\to 0
\]
is exact.  Considered as $\End_R(K_i^*)$-modules, the last term is projective, so the sequence splits. It follows that $\Hom_R(K_i^*,M_i^*)$ is a summand of $\Hom_R(K_i^*,V_i^*)$.  By reflexive equivalence, $M_i^*$ is then a summand of $V_i^*$, thus $M_i$ is a summand of $V_i$, which is a contradiction since $V_i\in\add M_{I^c}$. This shows that each $K_i\in\add M_I$.  

Now since each $K_i$ is indecomposable, it remains to show that $K_i\ncong K_j$ for $i\neq j$.  Suppose that it is false, i.e.\ $K_i\cong K_j$ with $i\neq j$.  By \eqref{22a} and \eqref{22b}, the map $c_i^*\colon V_i^*\to K_i^*$ is a minimal right $\add M_{I^c}^*$-approximation for all $i\in I$, thus since $K_i^*\cong K_j^*$ it follows that $M_i^*\cong M_j^*$, which is a contradiction since $M$ is basic.  It follows that $K_I\cong M_I$, so (2) holds.  Further,  \eqref{22} is the exact sequence
\begin{eqnarray}
0\to M_I\xrightarrow{c} V_I\xrightarrow{a} M_I\label{MI twice}
\end{eqnarray}
with $a$ a minimal right $\add M_{I^c}$-approximation of $M_I$, so by definition $\upmu_{I}M=M_{I^c}\oplus \Ker a\cong M$, proving the first half of (1).  Now by \ref{dualityofapprox}, dualizing \eqref{MI twice} gives an exact sequence
\begin{eqnarray}
0\to M_I^*\xrightarrow{a^*} V_I^*\xrightarrow{c^*} M_I^*\label{MI twice dual}
\end{eqnarray}
and by \eqref{K0Da} $c^*$ is a minimal right $\add M_{I^c}^*$-approximation.  Thus applying $\Hom_R(M^*,-)$ gives the minimal projective resolution of $\Lambda_I^{\op}$, proving (3).  Also, by definition $\upnu_{I}M=M_{I^c}\oplus (\Ker(c^*))^*\cong M$, proving the second half of (1).
\end{proof}

The following gives equivalent conditions to when the assumptions of \ref{proj res thm 1} hold.

\begin{lemma}\label{cor to pd thm 1}
In the setup of \ref{stab setup}, the following are equivalent 
\begin{enumerate}
\item $\pd_\Lambda\Lambda_I=2$.
\item $\pd_\Lambda\Lambda_I<\infty$ and $\depth_R\Lambda_I=d-2$.
\item $\upnu_{I}M\cong M$.
\item $\upmu_{I}M\cong M$.
\end{enumerate}
\end{lemma}
\begin{proof}
(1)$\Leftrightarrow$(2) is just the Auslander--Buchsbaum formula  \cite[2.15]{IW4}.\\
(1)$\Rightarrow$(3) is \ref{proj res thm 1}.\\
(3)$\Rightarrow$(4) Since $\upnu_{I}M\cong M$, \eqref{K1D} is simply
\[
0\to M_I\stackrel{b^{*}}\to U_I\stackrel{d^{*}}\to M_I 
\]
where $d^*$ is a minimal $\add M_{I^c}$-approximation by \ref{dualityofapprox}.  Since minimal approximations are unique, $V_I\cong U_I$ and $\upmu_{I}M\cong M_{I^c}\oplus M_I=M$.\\
(4)$\Rightarrow$(1) Applying $\Hom_R(M,-)$ to \eqref{K0} gives an exact sequence
\[
0\to\Hom_R(M,K_I)\to\Hom_R(M,V_I)\to\Hom_R(M,M_I)\to\Lambda_I\to 0.
\]
Since $\upmu_{I}M\cong M$, $K_I\cong M_I$ and so the first three terms are all projective.
\end{proof}

\begin{remark}
We remark that when $\dim R=2$, $R$ is an isolated singularity and so automatically $\dim_\mathbb{C}\Lambda_I<\infty$, which implies that $\depth_R\Lambda_I=0$.  Thus in this case the conditions in \ref{cor to pd thm 1} are equivalent to simply $\pd_\Lambda \Lambda_I<\infty$.
\end{remark}

The last two results combine to prove the following, which was stated in \S\ref{mut prelim}.
\begin{cor}\label{Ext 2 thm mut section}
In the setup of \ref{stab setup}, suppose that $\upnu_I M\cong M$.  Then
\begin{enumerate}
\item\label{Ext 2 thm MS 1} $T_I=\Lambda(1-e_I)\Lambda$ and $\Gamma:=\End_\Lambda(T_I)\cong \Lambda$.
\item\label{Ext 2 thm MS 2} $\Omega_\Lambda\Lambda_I=T_I$, thus $\pd_\Lambda \Lambda_I=2$ and $\Ext^1_\Lambda(T_I,-)\cong\Ext^2_\Lambda(\Lambda_I,-)$.
\end{enumerate}
\end{cor}
\begin{proof}
(1) Adding the exact sequence $0\to 0\to \Hom_R(M,M_{I^c})\xrightarrow{\Id}\Hom_R(M,M_{I^c})\to 0 \to 0$ to the minimal projective resolution in \ref{proj res thm 1}\eqref{proj res thm 1 part 2} gives the projective resolution
\[
0\to \Hom_R(M,M_I)\xrightarrow{\uppsi} \Hom_R(M,V_I\oplus M_{I^c})\to \Lambda\to\Lambda_I\to 0.
\]
By definition $T_I$ is the cokernel of the morphism $\uppsi$, which by inspection is $\Lambda(1-e_I)\Lambda$.   The isomorphism $\End_\Lambda(T_I)\cong \Lambda$ is \cite[6.1(1)]{DW1}.\\
(2) This follows from the exact sequence above, together with dimension shifting.
\end{proof}

By contrast to \ref{Ext 2 thm mut section}, it is often the case that $\upnu_{I}M\ncong M$.  The following is needed, and depends heavily on \ref{dualityofapprox}. 

\begin{lemma}\label{iterate and dual}
In the setup of \ref{stab setup}, suppose further that either 
\begin{enumerate}
\item[(a)] $\upnu_{I}M\cong M$, or
\item[(b)] $\upnu_{I}\upnu_{I}M\cong M$ and $\dim_{\mathbb{C}}\End_R(M)_I<\infty$.
\end{enumerate}
holds.  Then \t{(a)} or \t{(b)} also holds for $N_1:=\upnu_IM$, $N_2:=M^*$ and $N_3:=(\upnu_IM)^*$.
\end{lemma}
\begin{proof}
Suppose that $M$ satisfies assumption (a).  The fact that $\upnu_IM$ also satisfies (a) is obvious.  The fact that $M^*$ does too is a consequence of \ref{dualityofapprox}, so since $(\upnu_IM)^*=M^*$ in this case, so too does $N_3$.

Hence we can assume that $M$ satisfies assumption (b).  We see that $\upnu_I\upnu_IN_1\cong N_1$ simply by applying $\upnu_I$ to both sides of the equation $\upnu_I\upnu_IM\cong M$.  Since $\Lambda_I\cong(\upnu_I\Lambda)_I$ by \cite[6.20]{IW}, the finite dimensionality is preserved too.  For the statement involving $N_2$, we need some notation.  Since $\upnu_IM:=M_{I^c}\oplus J_I^*$, we consider a minimal right $\add(\frac{\upnu_IM}{J_I^*})^*=\add M_{I^c}^*$ approximation of $(J_I^*)^*\cong J_I$ 
\[
0\to \Ker\to W^*_I\to J_I
\]
then since $\upnu_I\upnu_IM\cong M$, $\Ker^*\cong M_I$.  Thus dualizing the above, using \ref{dualityofapprox}, 
\begin{eqnarray}
0\to J_I^*\to W_I\to M_I\label{8O}
\end{eqnarray}
is exact, where the last map is a minimal $\add M_{I^c}$-approximation. By uniqueness of minimal approximations $W_I\cong V_I$, and $\upnu_I(M^*)=M_{I^c}^*\oplus J_I$.  Finally \eqref{K1D} and \eqref{K1Da} show that $\upnu_I\upnu_I(M^*)\cong M^*$.  The finite dimensionality part follows since $(\Lambda_I)^{\op}=(\Lambda^{\op})_I$.  The statement for $N_3$ follows by combining the statements for $N_1$ and $N_2$. .
\end{proof}

\begin{prop}\label{proj res thm}
In the setup of \ref{stab setup}, suppose further that $d\geq 3$, and that assumption \t{(b)} in \ref{iterate and dual} is satisfied.  Then
\begin{enumerate}
\item\label{proj res thm part 1}  Applying $\Hom_R(M,-)$ to the sequence \eqref{K1D} gives an exact sequence
\[
0\to \Hom_R(M,M_I)\to \Hom_R(M,U_I)\to \Hom_R(M,J_I^*)\to 0.
\]
\item\label{proj res thm part 2} $\upnu_{I}M\ncong M$. 
\item\label{proj res thm part 3}  The minimal projective resolution of $\Lambda_I$ as a $\Lambda$-module is
\[
0\to \Hom_R(M,M_I)\to \Hom_R(M,U_I)\to \Hom_R(M,V_I)\to \Hom_R(M,M_I)\to \Lambda_I\to 0.
\] 
\item\label{proj res thm part 4} $d=3$.
\end{enumerate}
\end{prop}
\begin{proof}
(1) This is the argument in \cite[(6.Q)]{IW4}.  Denote $\mathbb{G}:=\Hom_R(M_{I^c},-)$, then applying $\Hom_R(M,-)$ to \eqref{K1D} and applying $\Hom_\Lambda(\mathbb{G}M,-)$ to \eqref{K1Da} and comparing them, by reflexive equivalence
\[
\begin{tikzpicture}
\node (A1) at (0,0) {$0$}; 
\node (A2) at (2,0) {$\Hom_\Lambda(\mathbb{G}M,\mathbb{G}M_I)$};
\node (A3) at (5.25,0) {$\Hom_\Lambda(\mathbb{G}M,\mathbb{G}U_I)$};
\node (A4) at (8.5,0) {$\Hom_\Lambda(\mathbb{G}M,\mathbb{G}J_I^*)$};
\node (A5) at (11.75,0) {$\Ext^1_\Lambda(\mathbb{G}M,\mathbb{G}M_I)$};
\node (B1) at (0,-1.25) {$0$}; 
\node (B2) at (2,-1.25) {$\Hom_R(M,M_I)$}; 
\node (B3) at (5.25,-1.25) {$\Hom_R(M,U_I)$}; 
\node (B4) at (8.5,-1.25) {$\Hom_R(M,J_I^*)$}; 
\node (B5) at (11.5,-1.25) {$C$};
\node (B6) at (13,-1.25) {$0$};
\draw[->] (A1)--(A2);
\draw[->] (A2)--(A3);
\draw[->] (A3)--(A4);
\draw[->] (A4)--(A5);
\draw[->] (B1)--(B2);
\draw[->] (B2)--(B3);
\draw[->] (B3)--(B4);
\draw[->] (B4)--(B5);
\draw[->] (B5)--(B6);
\draw[->] (A2)--node[left]{$\scriptstyle\cong$}(B2);
\draw[->] (A3)--node[left]{$\scriptstyle\cong$}(B3);
\draw[->] (A4)--node[left]{$\scriptstyle\cong$}(B4);
\end{tikzpicture}
\]  
Hence $C$ is a submodule of $\Ext^1_\Lambda(\mathbb{G}M,\mathbb{G}M_I)$.  But since \eqref{K1Da} is exact, it follows that $e_IC=0$ and so $C$ is a finitely generated $\Lambda_I$-module.  Since $\Lambda_I$ has finite length, so too does $C$.  But $\depth_R\Ext^1_\Lambda(\mathbb{G}M,\mathbb{G}M_I)>0$ since $\Hom_\Lambda(\mathbb{G}M,\mathbb{G}M_I)\cong\Hom_R(M,M_I)\in\CM R$, hence $C=0$.\\
(2)  If $\upnu_{I}M\cong M$, then $J_I^*\cong M_I$.  Consequently, viewing the exact sequence in \eqref{proj res thm part 1} as $\End_R(M)$-modules, the last term is projective and so the sequence splits.  By reflexive equivalence this would imply that $M_I$ is a summand of $U_I\in\add M_{I^c}$, which is a contradiction.\\
(3) Since the last map in \eqref{8O} is a minimal right $\add M_{I^c}$-approximation, applying $\Hom_R(M,-)$ to \eqref{8O} gives an exact sequence
\[
0\to \Hom_R(M,J_I^*)\to \Hom_R(M,V_I)\to \Hom_R(M,M_I)\to \Lambda_{\con}^I\to 0.  
\]
Splicing this with the exact sequence in \eqref{proj res thm part 1} gives the result.\\
(4) If $d>3$, applying the depth lemma to the projective resolution in \eqref{proj res thm part 3} gives a contradiction.
\end{proof}

The above results give the following, stated in \S\ref{mut prelim}.

\begin{cor}\label{Ext 3 thm mut section}
Suppose that $d\geq 3$, $\upnu_I \upnu_IM\cong M$ and $\dim_{\mathbb{C}}\Lambda_I<\infty$.  As above, set $\Gamma:=\End_\Lambda(T_I)\cong\End_R(\upnu_IM)$.  Then 
\begin{enumerate}
\item\label{Ext 3 thm MS 1} $T_I\cong \Hom_R(M,\upnu_IM)$.
\item\label{Ext 3 thm MS 2} $\Omega_\Lambda^2\Lambda_I=T_I$, thus $\pd_\Lambda \Lambda_I=3$ and $\Ext^1_\Lambda(T_I,-)\cong\Ext^3_\Lambda(\Lambda_I,-)$.
\end{enumerate}
\end{cor}
\begin{proof}
(1) This follows from the definition of $T_I$, together with \ref{proj res thm}\eqref{proj res thm part 1}.\\
(2) This is now immediate from \eqref{Ext 3 thm MS 1}, \ref{proj res thm}\eqref{proj res thm part 1} and \ref{proj res thm}\eqref{proj res thm part 3}. 
\end{proof}

\section{Conjectures}\label{gen appendix}
This appendix outlines conjectures and further directions.  First and foremost, we are hampered by the fact that the Bridgeland--Chen flop functor is only known to be an equivalence in the setting of Gorenstein terminal singularities.  This paper began by trying to lift the reconstruction algebra of \cite{WemGL2} to $3$-folds, and through the analysis of many non-Gorenstein examples.  There is evidence to suggest the following.

\begin{conj}\label{B1}
Suppose that $X\to X_{\con}$ is a flopping contraction of $3$-folds, where $X$ has at worst CM rational singularities.  Then the flop functor is an equivalence if and only if the universal sheaf of the noncommutative deformation functor associated to the curves is a perfect complex (equivalently, $\pd_\Lambda\Lambda_{\con}<\infty$).
\end{conj}

This would recover Bridgeland \cite{Bridgeland} and Chen \cite{Chen}, since when $X$ has only Gorenstein terminal singularities, the universal sheaf is  guaranteed to be perfect \cite[7.1]{DW1}.

Whilst mutation needs $R$ to be Gorenstein to ensure that it gives a derived equivalence, it can sometimes be an equivalence when $R$ is not Gorenstein.  The relationship between flops and mutation seems to be tight.
\begin{conj}\label{B2}
When $C_i$ is a crepant curve whose universal sheaf is perfect, and $X$ has at worst CM rational singularities, then Theorem~\ref{mut functor thm intro} remains true.
\end{conj}

\begin{conj}\label{B3}
The Homological MMP in Figure~\ref{Fig2} can be used to flop curves and jump between minimal models of non--Gorenstein singularities, again in the CM rational singularities setting, provided that we account for $\pd_\Lambda\Lambda_{\con}<\infty$.
\end{conj}

Even although $\pd_\Lambda\Lambda_{\con}<\infty$ seems necessary to relate mutation to flops, it does not seem so relevant for moduli tracking purposes.  The following is at least true in many examples, and may be true more generally.

\begin{conj}
The moduli tracking theorem \ref{main stab track} is true under the simplifying assumption $\upnu_I\upnu_IM\cong M$. 
\end{conj}

Tracking moduli in the non-Gorenstein setting is substantially harder, since even reasonable algebras like NCCRs need not be closed under derived equivalence.  

In an algebraic direction, \ref{NCvsC min model} should extend to the situation when $R$ is not Gorenstein.  There is a version of the Auslander--McKay Correspondence in dimension two when $R$ is not Gorenstein \cite{Wunram, WemGL2}, obtained by replacing CM modules by Wunram's  notion of a special CM module \cite{Wunram}.  There should be a three dimensional analogue of this.
\begin{conj}
There is a notion of `special MM generator' such that the Auslander--McKay Correspondence \ref{NCvsC min model} holds for non--Gorenstein rational $3$-fold singularities whose minimal models have fibres that are at most one-dimensional.  
\end{conj}

Of course, the above conjecture must also account for $\pd_\Lambda\Lambda_{\con}<\infty$, but again this is guaranteed if we restrict to those $\Spec R$ admitting minimal models with only Gorenstein terminal singularities.

Since the noncommutative deformations in \S\ref{contractions and deformations section} detect contractions for both flips and flops, and has no restriction on singularities, it is reasonable to speculate about modifying the Homological MMP to cover flips.  Indeed, philosophically there we should \emph{not} be changing the GIT stability, since there is no derived equivalence so we do not expect to be able to track the moduli back.  Instead, we change the algebra, keeping the GIT fixed.

\begin{conj}
In the setting of CM rational singularities, given $N=H^0(\cV_X)$, there is some homological modification of mutation that produces $H^0(\cV_{X'})$ where $X'$ is the flip.   Consequently, the Homological MMP in Figure~\ref{Fig2} can be extended to cover both flips and flops. 
\end{conj}

\end{document}